\documentclass[11pt, a4paper]{article}

\usepackage{inputenc} 
\usepackage[english]{babel}
\usepackage{dsfont}

\usepackage[totalheight=24 true cm, totalwidth=17 true cm]{geometry}
\usepackage[shortlabels]{enumitem}

\setlist[itemize]{noitemsep}

\usepackage{bbding}
\usepackage{amsmath,multirow}
\usepackage{amsthm}
\usepackage{amssymb}
\usepackage{subcaption}
\usepackage{mathrsfs}
\usepackage{mathtools}
\usepackage{ stmaryrd }
\usepackage{url}
\usepackage{wrapfig}
\usepackage{starfont}
\usepackage{pifont}
\usepackage{eurosym}
\usepackage{xcolor}

\usepackage{imakeidx}
\makeindex[]

\usepackage{multicol}
\usepackage{relsize}
\usepackage{float}
\usepackage{tikz,pgf}
\usepackage{tikz-cd}
\usepackage{wasysym}
\usepackage{graphicx}
\usepackage{fancyhdr}
\usepackage{verbatim}

\usepackage{pgfplots}
\pgfplotsset{compat=1.15}
\usepackage{mathrsfs}
\usetikzlibrary{arrows}

\usepackage[all,dvips,knot,web,arc,curve,color,frame]{xy}
\usepackage[all,knot,arc,color,web]{xy}
\xyoption{arc}
\xyoption{web}
\xyoption{curve}

\numberwithin{equation}{section}

\theoremstyle{plain}
\newtheorem{theorem}{Theorem}[section]
\newtheorem{proposition}[theorem]{Proposition}
\newtheorem{corollary}[theorem]{Corollary}
\newtheorem{lemma}[theorem]{Lemma}
\newtheorem{notation}[theorem]{Notation}

\newtheorem{definition}[theorem]{Definition}

\theoremstyle{definition}

\newtheorem{remark}[theorem]{Remark}
\newtheorem{note}{Note}
\newtheorem{example}[theorem]{Example}
\newtheorem{question}[theorem]{Question}
\usepackage[bookmarksnumbered=true, pdfencoding=auto]{hyperref} 
\hypersetup{
     colorlinks = true,
     linkcolor = blue,
     anchorcolor = blue,
     citecolor = teal,
     filecolor = blue,
     urlcolor = blue
     }
\newcommand\restr[2]{{
  \left.\kern-\nulldelimiterspace 
  #1 
  \vphantom{\small|} 
  \right|_{#2} 
  }}

\everymath{\displaystyle}
\allowdisplaybreaks

\def\C{\mathbb{C}}
\def\Z{\mathbb{Z}}
\def\S{\mathbb{S}}
\def\Id{\mathbb{1}}
\def\H{\mathcal{H}}

\def\Hom{\operatorname{Hom}}

\def\sgn{\operatorname{sgn}}
\def\op{\operatorname{op}}
\def\Rep{\operatorname{Rep}}
\def\Irr{\operatorname{Irr}}

\def\End{\operatorname{End}}

\def\det{\operatorname{det}}

\def\dim{\operatorname{dim}}

\def\Im{\operatorname{Im}}
\def\Res{\operatorname{Res}}
\def\Rank{\operatorname{Rank}}
\def\Ind{\operatorname{Ind}}

\def\exp{\operatorname{exp}}
\def\Ker{\operatorname{Ker}}

\def\Diag{\operatorname{Diag}}

\def\tr{\operatorname{tr}}
\def\spam{\operatorname{span}}
\def\Id{\operatorname{Id}}

\makeatletter
\def\mathcenterto#1#2{\mathclap{\phantom{#1}\mathclap{#2}}\phantom{#1}}
\let\old@widetilde\widetilde
\def\widetildeto#1#2{\mathcenterto{#2}{\old@widetilde{\mathcenterto{#1}{#2\,}}}}
\let\old@widehat\widehat
\def\widehatto#1#2{\mathcenterto{#2}{\old@widehat{\mathcenterto{#1}{#2\,}}}}
\makeatother

\newcommand{\ip}[2]{\langle  #1,#2 \rangle} 
\newcommand{\size}[1]{\left| #1 \right|} 
\newcommand{\pare}[1]{\left( #1 \right)} 
\newcommand{\set}[1]{{\left\{ #1 \right\}}} 

\newcommand{\corch}[1]{\left[ #1 \right]} 
\def\one{\mathbbm{1}}

\newcommand{\rel}[1]{${\rm{(#1)}}$}

\def\dim{\operatorname{dim}}

\DeclareMathOperator{\sign}{sign}

\DeclareMathOperator{\CC}{\mathbb{C}}
\DeclareMathOperator{\FF}{\mathbb{F}}

\DeclarePairedDelimiter\floor{\lfloor}{\rfloor}

\hypersetup{
    colorlinks=true,
    linkcolor=blue,
    filecolor=magenta,      
    urlcolor=cyan,
    pdftitle={Overleaf Example},
    pdfpagemode=FullScreen,
    }
    
\def\Id{\operatorname{Id}}

\setlist[itemize]{noitemsep}

\normalfont

\usepackage{bbm}
\usepackage{dsfont}

\usepackage{url}
\usepackage{amsmath,blkarray,booktabs, bigstrut}
\usepackage{tikz}
\usepackage{multirow}
\usepackage{xcolor}
\usepackage{color}
\usepackage[maxbibnames=99]{biblatex}
\def\g{\mathfrak{g}}
\def\sp{\mathfrak{sp}}
\def\h{\mathfrak{h}}

\def\Infl{\operatorname{Infl}}

\def\Sp{\operatorname{Sp}}
\def\SO{\operatorname{SO}}
\def\Sl{\operatorname{SL}}

\def\Gl{\operatorname{GL}}

\begin{document}

\title{Representations of Yokonuma--Hecke algebras.}
\author{
Emiliano Liwski \and
 Martín Mereb
 }


\date{}

\maketitle

\begin{abstract}
The Iwahori--Hecke and Yokonuma--Hecke algebras have played crucial roles in algebraic combinatorics and the representation theory of finite groups. In this work, we use classical results from representation theory to compute the character values of the Yokonuma--Hecke algebra at $w_{0}$ and $w_{0}^{2}$, where $w_{0}$ denotes the element of maximal length in the corresponding Weyl group.
\end{abstract}

{\hypersetup{linkcolor=black}
{\tableofcontents}}

\vspace{50pt}

\section{Introduction}

The study of Hecke algebras has long been a central theme in representation theory and algebraic combinatorics. Among these, the Iwahori--Hecke algebras and Yokonuma--Hecke algebras occupy significant positions given their deep connections with Coxeter groups, Weyl groups, and various other fundamental structures in Lie theory and algebraic groups. In this paper, we 
explore the irreducible representations of these 
interesting algebraic structures, aiming to compute the values of the characters of the Yokonuma--Hecke algebra at certain key elements.

Iwahori--Hecke algebras were first introduced in \cite{iwahori1964stucture,iwahori1965some}, to classify the irreducible representations of finite Chevalley
groups and reductive $p$-adic Lie groups. They arise naturally in the context of Coxeter groups, providing a $q$-deformation of the group algebra of a Coxeter group, for $q$ a given prime power. Given a Coxeter group $(W, I)$, the Iwahori--Hecke algebra $\mathcal{H}_q(W)$ is generated by elements corresponding to the generators of $W$ but subject to deformed relations. These algebras play a fundamental role in the representation theory of finite groups of Lie type, \cite{kazhdan1979representations,lusztig1983singularities}, and have profound implications in areas such as the theory of Macdonald polynomials, quantum groups, and the study of Schur--Weyl duality \cite{geck2000characters,carter1985finite}, knot theory and quantum groups \cite{jones1983index,drinfeld1986quantum,jones1989knot}.

On the other hand, Yokonuma--Hecke algebras, \cite{yokonuma1967structure,yokonuma1968commutant}, can be viewed as a natural extension of Iwahori--Hecke algebras. They are constructed by 
extending the Iwahori--Hecke algebra with an abelian group, thus expanding the 
algebraic structure and enriching the scope of research on the representation theory side. These algebras 
arose 
intending to study the representations of $p$-adic groups, \cite{iwahori1965some,matsumoto1969sous}, and have since found applications in representation theory and knot theory \cite{juyumaya2002markov, thiem2004unipotent}.

This paper is 
devoted to the study of the representations of Iwahori--Hecke and Yokonuma--Hecke algebras. We begin by reviewing the 
main constructions and properties of Coxeter groups and their associated Hecke algebras. We then delve into the structural aspects of Iwahori--Hecke algebras and Yokonuma--Hecke algebras, outlining key results and known representations~\cite{lusztig2003hecke,mathas1999iwahori,juyumaya2002markov}, focusing on the interplay between their structures.

The main goal of our work is to compute character values for specific elements within the Yokonuma--Hecke, namely at $w_{0}$ and $w_{0}^{2}$, where $w_{0}$ represents the element associated with the maximal length element of the corresponding Weyl group. Through several applications of Tits deformation theorem, we establish bijections between irreducible characters of these algebras and explore their implications. We follow the techniques of \cite{hausel2019arithmetic, geck2000characters, hausel2011arithmetic} used for the Chevalley group of type $A$. In addition to theoretical explorations, we provide concrete computations,  focusing exclusively on classical types \(A, B, C, D\), to illustrate our results. 

\medskip
\noindent {\bf Main results.} We 
denote by $G$ the Chevalley group constructed from a Lie algebra associated with one of the classical root systems, over the field $\FF_{q}$. In particular, we consider $G$ as one of the explicit Chevalley groups constructed in Subsection~\ref{explicit}. We denote by $B,U,N$ and $W$ the associated Borel subgroup,  unipotent radical, normalizer of the maximal torus, and Weyl group, respectively. Recall that the Iwahori--Hecke algebra and the Yokonuma--Hecke algebra of $G$ are defined as the Hecke algebras $\mathcal{H}(G,B)$ and $\mathcal{H}(G,U)$, respectively. Here is a summary of the main results of this paper.

\begin{itemize}
\item In Section~\ref{sec 3}, we establish a bijection 
between $\Irr(\mathcal{H}(G,U))$ and 
$\Irr(S_{q-1,n})$, where $S_{q-1,n}$ is an extension of the Weyl group. 
In Subsection~\ref{rep normalizador}, 
we compute 
$\Irr(S_{q-1,n})$.
\item In Subsection~\ref{son split}, we give a complete set of irreducible representations of the generic Yokonuma--Hecke algebra $Y_{q-1,n}$ for 
types $A$ and $B$. Observe that 
these representations are in bijection with $\Irr(\mathcal{H}(G,U))$. 
\item In Section~\ref{sec 5}, for each 
classical root system, we compute the character values of the Yokonuma--Hecke algebra $\mathcal{H}(G,U)$ at $w_{0}$ and $w_{0}^{2}$, where $w_{0}$ denotes the element of maximal length in the corresponding Weyl group.
\end{itemize}

{\noindent}
We would like to emphasize that the results presented here work for any finite field of characteristic other than two.

\medskip

\noindent {\bf Outline.} In Section~\ref{prel}, we introduce key objects, such as Chevalley groups, Iwahori--Hecke algebra, and Yokonuma--Hecke algebra. In Section~\ref{sec 3}, we 
define generic deformations of the Iwahori--Hecke algebra and the Yokonuma--Hecke algebra for each of the classical root systems. Using Tits deformation theorem, we 
obtain bijections between the irreducible characters of these algebras and the irreducible characters of the Weyl group and the normalizer of the maximal torus.
In Section~\ref{sec 4}, we use classical techniques from representation theory to compute the irreducible representations of a deformation of the Yokonuma--Hecke algebra. In Section~\ref{sec 5}, we use the results from Section~\ref{sec 4} to compute the values of the irreducible characters of the Yokonuma--Hecke algebra at $w_{0}$ and $w_{0}^{2}$, where $w_{0}$ denotes the maximal length element of the 
Weyl group. In Section~\ref{tables}, we illustrate these values using tables.

\section{Preliminaries}\label{prel}

In this section, we 
review the basic definitions of Hecke algebras and Chevalley groups.
For a more complete description of these objects 
see~\cite{steinberg1967lectures,thiem2004unipotent,hausel2019arithmetic,iwahori1964stucture}.

\begin{notation}\normalfont 
\phantom{ We will first start by fixing some notation. }
\begin{itemize} 
\item For a finite-dimensional $\CC$-algebra $A$, we 
denote by $\Rep(A)$ and $\Irr(A)$ its sets of representations and irreducible representations, respectively. If $A=\C[G]$ we denote them by $\Rep(G)$ and $\Irr(G)$ respectively, and we will also sometimes identify them with 
their set of 
characters. 
\item If $H$ is an $\C[u^{\pm 1}]$-algebra, we 
denote the extension $\C(u)\otimes_{\C[u^{\pm 1}]}H$ by $H(u)$. 
\item We 
denote by $[n]$ the set $\{1,\ldots,n\}$, and by $C_{d}$ the cyclic group of $d$ elements, that we sometimes identify with the set of $d$-th roots of unity.
\item For a Lie algebra $\g$ that is a subalgebra of matrices of $k^{n\times n}$, we 
refer to the representation over $k^{n}$ given by left multiplication by elements of $\g$ as the \textit{defining representation}. 
\item For a group $G$ 
acting on a set $X$, we denote by 
$\textbf{Stab}(x)$ the stabilizer of $x\in X$ under this action. 
\item If $T$ is a group, we 
denote by $T^{\ast}$ the set of group morphisms $\Hom(T,\C^{\times})$.
\item We 
denote by $\mathcal{P}$ the set of all partitions, and by $\mathcal{P}_{n}$ the set of partitions of size $n$. The 
set $\Irr(\mathbb{S}_n)$ is parametrized by $\mathcal{P}_{n}$, and for $\lambda \in \mathcal{P}_{n}$ we 
denote by $\S^{\lambda}\in \Irr(\S_{n})$ the associated representation. 
\item Let $p\in \Rep(G)$ 
and $H\leq G$ be a subgroup, when it is clear we 
denote the restriction $\restr{p}{H}$ as $p\rq$.
\item We say that a finite-dimensional $\C$-algebra $A$ is \textit{split} over a field $\C\subset k$ if for every $\widetilde{p} \in \Rep(A)$ over a field $k\rq$, with $k\subset k\rq$, there exists $p\in \Rep(A)$ defined over $k$ such that $\widetilde{p}=p\otimes_{k}k\rq$. 
\item We 
assume that all finite fields $\mathbb{F}_{q}$ are of odd characteristic, and all the results exhibited here will be under this assumption.
\end{itemize}

\end{notation}

\subsection{Hecke algebras}\label{Álgeb the set of partitions of size $n$ras de Hecke}

Let $H\leq G$ be finite groups. The \textit{Hecke algebra} associated with $G$ and $H$, which we denote by $\mathcal{H}(G,H)$, is the vector space of $H$-bivariate functions from $G$ to $\CC$ with pointwise sum,~i.e, functions $f:G\rightarrow \CC$ that satisfy that $f(h_{1}gh_{2})=f(g)$ for every $h_{1},h_{2}\in H$. The convolution gives the product 
 \[(\phi_{1}\ast \phi_{2})(g)=\frac{1}{\size{H}}\underset{x\in G}{\sum}\phi_{1}(x)\phi_{2}(x^{-1}g).\]

The algebra $\mathcal{H}(G,H)$ has a basis indexed by the double $H$-cosets in $G$. If $W\subset G$ is a complete system of representatives of the double $H$-cosets in $G$, then the indicators of the double cosets
 \begin{equation}\phi_{w}=\one_{HwH},\label{base Hecke}\end{equation}
 form a basis of $\mathcal{H}(G,H)$.

\medskip
\noindent{\bf Representations of Hecke algebras}

\medskip
\noindent Let $H\leq G$ be finite groups. 
Let $(V,\pi)\in \Irr(G)$ 
and let $V^{H}$ be the subspace fixed by $H$. Then, $V^{H}$ is a representation of  $\mathcal{H}(G,H)$ via the action 
\[\phi \cdot v=\frac{1}{\size{H}} \underset{a \in G}{\sum} \phi(a) a\cdot v,\]
for $\phi \in \mathcal{H}(G,H)$ and $v \in V^{H}$. Observe that we can identify
\[\Hom_{H}(\one_{H},\Res_{H}^{G} V)\cong V^{H}.\]
Then, we have a map 
\begin{equation*}
\begin{gathered}
D_{H}:\Rep(G)\rightarrow \Rep(\mathcal{H}(G,H)),\\
(V,\pi)\mapsto V^{H}.
\end{gathered}
\end{equation*}
On the other hand, let us denote 
\[\Irr(G:H)=\set{\eta \in \Irr(G):(\eta,\one_{H}^{G})>0};\]
note that this is equivalent to $\one_{H}$ being an irreducible factor of $\Res_{H}^{G}\eta$. We can give the following characterization of 
$\Irr(\mathcal{H}(G,H))$.
\begin{proposition}\textup{\cite[Proposition~2.10]{bernstein1976representations}}\label{D operador} 
Let $(V,\pi)$ be an irreducible representation of $G$ such that $V^{H}\neq 0$, then $V^{H}$ is an irreducible representation of $\mathcal{H}(G,H)$, and every irreducible representation arises in this way. Hence, $D_{H}$ restricts to a bijection $D_{H}:\Irr(G:H)\rightarrow \Irr(\mathcal{H}(G,H))$.
\end{proposition}

\subsection{Chevalley groups}\label{cheva}

The \textit{Chevalley group} can be constructed from any reductive Lie algebra $\g$, and is a necessary step for the definition of the Yokonuma--Hecke algebra. We follow the setting from \cite{steinberg1967lectures} and  \textup{\cite[§~2.2]{thiem2004unipotent}}.

\begin{definition}\normalfont
Let $\g$ be a semisimple Lie algebra. Let $R$ be the root system associated with $\g$, with $\Delta$ its set of simple roots and $\h$ a Cartan subalgebra. Let us also denote by $\g_{a}$ the eigenspace associated with the weight $a\in R$. For a system of elements $\set{x_{a}\in \g_{a},h_{a_{i}}\in \h: a\in R,a_{i}\in \Delta}$ we define the \textit{structure constants} as the values $c_{a+b}$ such that $\corch{x_{a},x_{b}}=c_{a+b}\ x_{a+b}$ and the values $a(h_{a_{i}})$ that satisfy $\corch{h_{a_{i}},x_{a}}=a(h_{a_{i}})x_{a}$.
\end{definition}

The following theorem gives us the basis with which we construct the Chevalley group.

\begin{theorem}[Chevalley basis]
\label{base de Chevalley}
It is possible to construct a basis of $\g$ of the form $\set{x_{a},\ h_{a_{i}},a\in R,a_{i}\in \Delta}$, with the property that all its structure constants are integer numbers.
\end{theorem}

\medskip
\noindent{}
{\bf Construction of Chevalley groups.}
Let $\g$ be a reductive Lie algebra, with $\h_s$ being a Cartan subalgebra of $\g_s=[g,g]$. Then we have that $\h=Z(\g)\oplus \h_s$ is a Cartan subalgebra of $g$. 
Let $R$ be the set of roots of $\g_{s}$, and let  $\set{a_1,
\ldots, a_l}$ be the set of simple roots of $R$. Since $\g_{s}$ is semisimple, we know that there exists a Chevalley basis
$\set{x_{a},h_{a_{i}}:a\in R,1\leq i\leq l}$,
like the one in Theorem~\ref{base de Chevalley}. Let $V$ be a finite-dimensional $\g$ module with representation map $\phi$. Assume that $V$ has a basis $\set{v_1,
\ldots, v_r}$ that satisfies
\begin{enumerate}[a)]
\item There exists a basis $\set{H_1,
\ldots, H_n}$ of $\h$ such that
\begin{enumerate}[1)]
\item $h_{a_i}\in \mathbb{Z}_{\geq 0}   \spam \set{H_1,
\ldots, H_n}$.
\item $H_{i}v_j\in \mathbb{Z}v_j$ for all $i\in [n]$ and $j\in [r]$.
\item $\Rank_{\mathbb{Z}}(\mathbb{Z} \spam \set{H_1,
\ldots, H_n})\leq \dim_{\mathbb{C}} \h$.
\end{enumerate}
\item $\dfrac{x_{a}^{n}}{n!} \ v_i\in \mathbb{Z} \spam \set{v_1,
\ldots ,v_r}$ if $n\geq 0$ and $a\in R$.
\item $\Rank_{\mathbb{Z}}\ (\mathbb{Z} \spam \set{v_1,
\ldots, v_r})\leq \dim_{\mathbb{C}}V$.
\end{enumerate}
Observe that $\set{v_{1},
\ldots, v_{r}}$ generate a lattice  $V(\mathbb{Z})$ invariant by $\exp(t\cdot \phi(x_{a}))$ and $\phi(H_{i})$. If $\g$ is semisimple, the existence of the lattice is guaranteed. For a prime power $q$, let 
\[\h_{\mathbb{Z}}=\mathbb{Z} \spam \set{H_1,
\ldots ,H_n}, \quad 
       \text{and} \quad V_q=\mathbb{F}_q  \spam \set{v_1,
       \ldots, v_r}.\] 
The {\em Chevalley group} $G_{V}\in \Gl_{r}(V_q)$ is defined by 
\[G_{V}=\ip{x_{a}(t),h_{H}(s):a\in R}{H\in \h_{\mathbb{Z}},t\in \mathbb{F}_q,s\in \mathbb{F}_q^{\times}},\]
where
\begin{equation*}
\begin{gathered}
h_{H}(s)=\Diag(s^{\lambda_{1}(H)},
\ldots, s^{\lambda_{r}(H)}) \  \text{if} \ Hv_i=\lambda_i(H)v_i,\\
x_a(t)=\sum_{n\geq 0}{t^n \frac{\phi(x_{a}^n)}{n!}}.
\end{gathered}
\end{equation*} 
Observe that the same procedure can be applied over any field $k$ over $V(\mathbb{Z})\otimes k$.
\begin{remark}\normalfont
\begin{enumerate}
\item If $\g=\g_s$ we have that $G_{V}=\ip{x_a(t)}{a\in R}$.
\item We can think of the construction of the Chevalley group as a way of going from Lie algebras to algebraic groups, by considering the subgroup generated by the exponential of nilpotent elements. 
Assume that the Lie algebra $\g$ is associated with a complex Lie group $\SO_{2n+1}(\C)$ or $\Sl_{n}(\C)$, and we construct the Chevalley group using the defining representation over a field $k$. In that case, we obtain a subgroup of $\SO_{2n+1}(k)$ or $\Sl_{n}(k)$ respectively, see \cite{ree1957some}. 
\end{enumerate}
\end{remark}

\subsection{Iwahori--Hecke algebras}\label{iwa}

In this subsection, we 
introduce the \textit{Iwahori--Hecke} algebra, which is defined from a Chevalley group. We 
now recall  
the presentations of 
Weyl groups associated with the classical types $A,B,C,D$. Recall that these groups arise as the groups generated by the simple symmetries of the classical root systems, and also that they are Coxeter groups. 

\medskip
{\bf Type A:}
The Weyl group of type $A_{n}$ is the symmetric group $\S_{n}$. The simple symmetries are the transpositions $s_{1},
\ldots, s_{n-1}$, where $s_{i}$ is the transposition $(i,i+1)$. The Weyl group has the following presentation:
\[
s_{i}s_{j}=s_{j}s_{i} \ \textup{if} \ \size{i-j}>1,\quad s_{i}s_{i+1}s_{i}=s_{i+1}s_{i}s_{i+1}, \quad s_{i}^{2}=1.
\]

\medskip
{\bf Types B and C:} The Weyl groups of types $B_{n}$ and $C_{n}$ coincide, and this group is the subgroup of $n\times n$ matrices that contain exactly one non-zero entry in each row and column, with these entries being $\pm 1$. This group is denoted by $W_{n}$, and we have that $W_{n}=N_{n}\rtimes \S^{n}$, where $N_{n}\cong \{\pm 1\}^{n}$ is the subgroup of diagonal matrices of $W_{n}$. Let $t\in W_{n}$ be the element $\Diag(-1,1,1,\ldots,1)$. The simple symmetries are $t,s_{1},\ldots ,s_{n-1}$, and the group $W_{n}$ has a presentation given by the relations
\begin{align*}
s_{1}ts_{1}t=s_{1}&ts_{1}t,\quad ts_{i}=s_{i}t \ \text{if} \ i\neq 1,\quad s_{i}s_{j}=s_{j}s_{i} \ \text{if} \ \size{i-j}>1,\\
&s_{i}s_{i+1}s_{i}=s_{i+1}s_{i}s_{i+1},\quad t^{2}=1, \quad s_{i}^{2}=1.
\end{align*}

\medskip
{\bf Type D:}
From the relations defining $W_{n}$ we can define a morphism $\epsilon:W_{n}\rightarrow \{\pm 1\}$ such that $\epsilon(t)=-1$ and $\epsilon(s_{i})=1$ for all $i$. The Weyl group of type $D_{n}$, which we denote by $W_{n}'$, is the kernel of $\epsilon$. The group $W_{n}'$ is a subgroup of index two, and an element $w\in W_{n}$ belongs to $W_{n}'$ if and only if it has an even number of $-1$. Defining $N_{n}'=N_{n}\cap W_{n}'$, we have that $W_{n}'=N_{n}'\rtimes \S^{n}$. Let $u=ts_{1}t$, then the simple symmetries of $W_{n}\rq$ are $u,s_{1},\ldots,s_{n-1}$, and the group has a presentation given by the following relations:
\begin{align*}&us_{1}=s_{1}u,\quad us_{2}u=s_{2}us_{2}, \quad \text{$us_{i}=s_{i}u$ if $i>2$},\\ 
 s_{i}s_{j}=s_{j}s_{i} &\ \ \text{if} \ \size{i-j}>1,
\quad s_{i}s_{i+1}s_{i}=s_{i+1}s_{i}s_{i+1},\quad u^{2}=1, \quad s_{i}^{2}=1.
\end{align*}

\medskip

Let $G$ be a Chevalley group constructed as in Subsection~\ref{cheva} for a prime power $q$, and let $B$ be a Borel subgroup of $G$. We now introduce the \textit{Iwahori--Hecke algebra}. 

\begin{definition}\normalfont
The \textit{Iwahori--Hecke algebra} associated with the Chevalley group $G$ is defined as $\mathcal{H}(G,B)$.
\end{definition}

The Bruhat decomposition for $G$, with respect to $B$, enables us to think of the Weyl group $W$ as a set of double $B$-coset representatives in $G$, and therefore $\{T_{w}=\one_{HwH}:w\in W\}$ gives a basis of
$\mathcal{H}(G, B)$, as we saw in Equation~\eqref{base Hecke}. The result given in  \textup{\cite[Theorem~3.2]{iwahori1964stucture}} gives us the following characterization
of $\mathcal{H}(G,B)$ in terms of generators and relations. For these relations, we denote the simple symmetries by $\Delta$.
\begin{enumerate}[a)]
\item If $w=s_{1}\cdots s_{k}$ is a reduced expression of $w\in W$, then $T_{w}=T_{s_{1}}\cdots T_{s_{k}}$. In particular, $\mathcal{H}(G,B)$ is generated by $\{T_{s}:s\in \Delta\}$.
\item For $s\in \Delta$, $T_{s}^{2}=q+(q-1)T_{s}$.
\end{enumerate}

\subsection{Yokonuma--Hecke algebras}\label{yoko}

Let $G$ be a Chevalley group constructed as in Subsection~\ref{cheva}, and let $B=U\rtimes T$ be a Borel subgroup, where $T$ and $U$ are its maximal torus and unipotent radical, respectively. 
Let $N$ be the normalizer of $T$ in $G$, we have that $N/T\cong W$. For an even number $d$ we 
define the group $N_{d,n}$, which has a presentation in terms of generators $\{t_{j},\xi_{i}: j\in [n],i\in [l]\}$, where $\Delta=\{a_{1},\ldots,a_{l}\}$ is the set of simple roots, $n$ is the rank of the root system associated with $W$, and the following relations.

\noindent
\begin{minipage}[t]{0.32\textwidth}
\begin{enumerate}[label=(n\arabic*), ref=(n\arabic*)]
    \item $t_j t_k = t_k t_j$. 
    \item $t_j^d = 1$. 
\end{enumerate}
\end{minipage}%
\hfill
\begin{minipage}[t]{0.32\textwidth}
\begin{enumerate}[label=(n\arabic*), ref=(n\arabic*)]
    \setcounter{enumi}{2}
    \item $\xi_i t_j = t_{s_{i}(H_j)} \xi_i$.
    \item $\xi_i^2 = t_{H_{a_i}}^{d/2}$.      
\end{enumerate}
\end{minipage}
\hfill
\begin{minipage}[t]{0.32\textwidth}
\begin{enumerate}[label=(n\arabic*), ref=(n\arabic*)]
    \setcounter{enumi}{4}
    \item $\underbrace{\xi_i \xi_j \cdots}_{\text{$m_{ij}$}} = \underbrace{\xi_j \xi_i \cdots}_{\text{$m_{ij}$}}$,  
\end{enumerate}
\end{minipage}

\medskip
\noindent{}
where $m_{ij}$ is the order of $s_{i}s_{j}$ in $W$, and we are using the following notation.

\begin{notation}\label{not t}
The elements $H_{a_{i}}\in \h_{\Z}$ and $s_{i}\in W$ denote the distinguished element and the simple symmetry associated with $a_{i}\in \Delta$, respectively. We also denote by $t_{a_{1}H_{1}+\cdots +a_{n}H_{n}}$ the element $t_{1}^{a_{1}}\cdots t_{n}^{a_{n}}$.
\end{notation}

For the exceptional case of the Chevalley group constructed from the semisimple Lie algebra of $\Sl_{n}$, we define $N_{d,n}$ with the same relations except for the relations corresponding to the maximal torus, relations~\rel{n1} and~\rel{n2}, which we define to be the diagonal matrices of determinant one over the ring $C_{d}$ (the set of $d$-th roots of unity). Using the Chevalley group relations from \textup{\cite[§~3.3.1]{thiem2004unipotent}}, we have the following lemma.

\begin{lemma}\label{iso n}
For a 
Chevalley group $G$ constructed for a prime power $q$, we have an isomorphism $N_{q-1,n}\cong N$.
\end{lemma}


\begin{definition}\normalfont
The \textit{Yokonuma--Hecke algebra} associated with a Chevalley group $G$ is defined as $\mathcal{H}(G,U)$.

\end{definition}

Using the relations given in \textup{\cite[Subsection~3.1.4]{thiem2004unipotent}}, we have that the Yokonuma--Hecke algebra constructed for any of the classical root systems has a presentation in terms of generators $\{t_{j},\xi_{i}: j\in [n],i\in [l]\}$, where $\Delta=\{a_{1},\ldots,a_{l}\}$ is the set of simple roots, $n$ denotes the rank of the root system and $q$ is the prime power with which we constructed $G$, and the following relations. 

\noindent
\begin{minipage}[t]{0.32\textwidth}
\begin{enumerate}[label=(y\arabic*), ref=(y\arabic*)]
    \item $t_j t_k = t_k t_j$. 
    \item $t_j^{q-1} = 1$.
\end{enumerate}
\end{minipage}%
\hfill
\begin{minipage}[t]{0.32\textwidth}
\begin{enumerate}[label=(y\arabic*), ref=(y\arabic*)]
    \setcounter{enumi}{2}
    \item $\xi_i t_j = t_{s_{i}(H_j)} \xi_i$.
    \item $\xi_{i}^2=q^{-1}t_{H_{a_{i}}}^{q-1/2}+q^{-1}e_{i}\xi_{i}$.     
\end{enumerate}
\end{minipage}
\hfill
\begin{minipage}[t]{0.32\textwidth}
\begin{enumerate}[label=(y\arabic*), ref=(y\arabic*)]
    \setcounter{enumi}{4}
    \item $\underbrace{\xi_i \xi_j \cdots}_{\text{$m_{ij}$}} = \underbrace{\xi_j \xi_i \cdots}_{\text{$m_{ij}$}}$,  
\end{enumerate}
\end{minipage}

\medskip

\noindent{}
where $m_{ij}$ is the order of $s_{i}s_{j}$ in $W$, and we are using Notation~\ref{not t}, and 
\[ e_{i}=\frac{1}{q-1}\sum_{j=1}^{q-1}t_{H_{a_{i}}}^{j}.\]

For the exceptional case of the Chevalley group constructed from the semisimple Lie algebra of $\Sl_{n}$ we obtain the same relations except for the relations of the maximal torus, which is formed by the diagonal matrices of determinant one over the ring $\FF_{q}^{\times}$.

\section{Generic deformations using Tits deformation theorem}\label{sec 3}

In this section, we 
construct generic deformations of the Iwahori--Hecke algebra and of the Yokonuma--Hecke algebra. Then, using 
Tits deformation theorem \textup{\cite[Theorem~7.4.6]{geck2000characters}}, we 
obtain bijections between $\Irr(\mathcal{H}(G,B))$ and $\Irr(W)$, and 
between $\Irr(\mathcal{H}(G,U))$ and $\Irr(N)$. We 
also give a generic deformation of $\C[N_{d,n}]$, and an explicit construction of the Chevalley group for each 
classical type, together with some of its properties.

\subsection{A 
deformation of the Iwahori--Hecke algebra}

Let $W$ be the Weyl group of one of the classical root systems associated with a Chevalley group $G$. If $u$ is an indeterminate over $\C$ we define the $\C[u^{\pm 1}]$-algebra $H_{n}$, where $n$ is the rank of the root system associated with $W$, generated by the elements $\{T_{w}:w\in W\}$ subject to the relations:
\begin{enumerate}[a)]
\item If $w=s_{1}\cdots s_{k}$ is a reduced expression of $w\in W$, then $T_{w}=T_{s_{1}}\cdots T_{s_{k}}$. In particular, $H_{n}$ is generated by $\{T_{s}:s\in \Delta\}$.
\item For $s\in \Delta$, $T_{s}^{2}=u+(u-1)T_{s}$.
\end{enumerate}  
$H_{n}$ is called the \textit{generic Iwahori--Hecke algebra} associated with $W$ with parameter $u$. Setting $u=1$ we obtain the relations of $\C[W]$, while when setting $u=q$ we obtain the relations of $\mathcal{H}(G,B)$, by using \textup{\cite[Theorem~3.2]{iwahori1964stucture}}. Then, we have
\[H_{n}(1)\cong \C[W],\quad \text{and} \quad   H_{n}(q)\cong \mathcal{H}(G,B).
\]
We consider the specializations at $u=1$ and $u=q$ and we call them $\sigma_{1}$ and $\sigma_{q}:\C[u^{\pm 1}]\rightarrow \C$. Using \textup{\cite[Theorem~2.9]{benson1972degrees}}, we have that $H_{n}(u)$ is split and semisimple, then applying Tits deformation theorem \textup{\cite[Theorem~7.4.6]{geck2000characters}}, we obtain bijections
\[d_{\sigma_{q}}:\Irr(H_{n}(u))\rightarrow \Irr(H_{n}(q))=\Irr(\mathcal{H}(G,B)), \quad \text{and} \quad d_{\sigma_{1}}:\Irr(H_{n}(u))\rightarrow \Irr(H_{n}(1))=\Irr(W).\]
Since $\C[u^{\pm 1}]$ is integrally closed in $\C(u)$, we obtain that the values of the characters of $H_{n}$ belong to $\C[u^{\pm 1}]$, by using \textup{\cite[Proposition~7.3.8]{geck2000characters}}. Moreover, if $\chi: H_{n}\rightarrow \C[u^{\pm 1}]$ is a character, we obtain the specialized characters 
\[\chi_{1}=d_{\sigma_{1}}(\chi)=\sigma_{1}(\chi), \quad \text{and} \quad  \chi_{q}=d_{\sigma_{q}}(\chi)=\sigma_{q}(\chi).\] 
Thus, we can define a bijection 
\[T_{B}=d_{\sigma_{q}}\circ d_{\sigma_{1}}^{-1}:\Irr(W)\rightarrow \Irr(\mathcal{H}(G,B)).\]
In particular 
$\Irr(\mathcal{H}(G,B))$ is parametrized by 
$\Irr(W)$. Using Proposition~\ref{D operador}, we have the following diagram of bijections 
\begin{center} \begin{tikzcd}
\Irr(G:B)\arrow[swap]{dr}{D_{B}} & & \Irr(W)\arrow{dl}{T_{B}}\\
& \Irr(\mathcal{H}(G,B)) &
\end{tikzcd}\end{center}

\subsection{A 
deformation of the Yokonuma--Hecke algebra}\label{gen yoko}

Let $G$ be a Chevalley group associated with one of the classical root systems. If $u$ is an indeterminate over $\C$, for $d$ even we define the $\C[u^{\pm 1}]$-algebra $Y_{d,n}$ generated by the elements $\{t_{j},\xi_{i}: j\in [n],i\in [l]\}$, where $\Delta=\{a_{1},\ldots,a_{l}\}$ is the set of simple roots, $n$ denotes the rank of the root system, subject to the relations:

\noindent
\begin{minipage}[t]{0.32\textwidth}
\begin{enumerate}[label=(Y\arabic*), ref=(Y\arabic*)]
    \item $t_j t_k = t_k t_j$. 
    \item $t_j^{d}=1$. 
\end{enumerate}
\end{minipage}%
\hfill
\begin{minipage}[t]{0.32\textwidth}
\begin{enumerate}[label=(Y\arabic*), ref=(Y\arabic*)]
    \setcounter{enumi}{2}
    \item $\xi_i t_j = t_{s_{i}(H_j)} \xi_i$.
    \item $\xi_{i}^2=ut_{H_{a_{i}}}^{d/2}+(1-u)e_{i}\xi_{i}$.     
\end{enumerate}
\end{minipage}
\hfill
\begin{minipage}[t]{0.32\textwidth}
\begin{enumerate}[label=(Y\arabic*), ref=(Y\arabic*)]
    \setcounter{enumi}{4}
    \item $\underbrace{\xi_i \xi_j \cdots}_{\text{$m_{ij}$}} = \underbrace{\xi_j \xi_i \cdots}_{\text{$m_{ij}$}}$.  
\end{enumerate}
\end{minipage}

\medskip

\noindent{} 
where $m_{ij}$ is the order of $s_{i}s_{j}$ in $W$, we are using Notation~\ref{not t}, and 
\[ e_{i}=\frac{1}{d}\sum_{j=1}^{d}t_{H_{a_{i}}}^{j}.\]
For the exceptional case of the Chevalley group constructed from the semisimple Lie algebra of $\Sl_{n}$ we take the same relations except for the relations of the maximal torus, relations~\rel{Y1} and~\rel{Y2}, which is changed by the diagonal matrices of determinant one over the ring $C_{d}$. 

The ring $Y_{d,n}$ is the \textit{generic Yokonuma--Hecke algebra} associated to $G$ with parameter $u$. From what we saw in Subsection~\ref{yoko}, we have that by setting $u=1$ in $Y_{d,n}$ we obtain the relations of $\C[N_{d,n}]$,  while when setting $u=q^{-1}$ in $Y_{q-1,n}$ we obtain the relations of $\mathcal{H}(G,U)$. Then, we have 
\[Y_{d,n}(1)\cong \C[N_{d,n}],\quad \text{and} \quad  Y_{q-1,n}(q^{-1})\cong \mathcal{H}(G,U).
\]
We consider the specializations at $u=1$ and $u=q^{-1}$ and we call them $\phi_{1}$ and $\phi_{q}:\C[u^{\pm 1}]\rightarrow \C$, respectively. Following the argument of 
\textup{\cite[Theorem~4.4.6]{geck2000characters}} we can deduce that $Y_{d,n}$ is a free algebra over $\C[u^{\pm 1}]$. We also have that $\C[N_{d,n}]$ is split semisimple, and by \textup{\cite[Corollary~3.5.3]{hausel2019arithmetic}} the algebra $\mathcal{H}(G,U)$ is semisimple. For the cases in which $Y_{d,n}(u)$ is split (we will see it later for 
types $A$ and $B$), using \textup{\cite[Proposition~7.3.8]{geck2000characters}} we obtain that the values of the characters of $Y_{d,n}$ belong to $\C[u^{\pm 1}]$. If it is not split over $\C(u)$ we consider $\mathbb{K}$ a finite Galois extension of  $\C(u)$  such that $\mathbb{K}Y_{d,n}$ is split, and by applying Tits deformation theorem, \textup{\cite[Theorem~7.4.6]{geck2000characters}}, we obtain the following bijections:
\begin{align*}
 d_{\phi_q}:&\Irr(\mathbb{K}Y_{q-1,n})\xrightarrow{\sim} \Irr(\mathbb{K}Y_{q-1,n}(q^{-1}))=\Irr(\mathcal{H}(G,U)),\\ 
 d_{\phi_1}:&\Irr(\mathbb{K}Y_{d,n})\xrightarrow{\sim} \Irr(\mathbb{K}Y_{d,n}(1))=\Irr(N_{d,n}).
\end{align*}
We call $\phi_q,\phi_1:A^{\ast} \rightarrow \C$ the specializations of $u$ at $q^{-1}$ and $1$ respectively, extended to the integral closure $A^{\ast}$ of $A=\C[u^{\pm 1}]$ in $\mathbb{K}$. Moreover, if $\chi: \mathbb{K}Y_{d,n}\rightarrow A^{\ast}$ is a character, we obtain the specialized characters
\begin{center}
$d_{\phi_q}(\chi)=\phi_q\circ \chi$,\quad \text{and} \quad  \ $d_{\phi_1}(\chi)=\phi_1 \circ \chi$.
\end{center}
In particular for $d=q-1$, using Lemma~\ref{iso n}, we can define a bijection
\begin{equation}T_{U}=d_{\phi_{q}}\circ d_{\phi_{1}}^{-1}:\Irr(N)\rightarrow \Irr(\mathcal{H}(G,U)).\label{bij T2}
\end{equation}
In particular 
$\Irr(\mathcal{H}(G,U))$ is parametrized by $\Irr(N)$.
\begin{center} \begin{tikzcd}
\Irr(G:U)\arrow[swap]{dr}{D_{U}} & & \Irr(N)\arrow{dl}{T_{U}}\\
& \Irr(\mathcal{H}(G,U)) &
\end{tikzcd}\end{center}

\subsection{A 
deformation of $\C[N_{d,n}]$}\label{gen nor}

The goal is to compute 
$\Irr(N_{d,n})$, that for 
$d=q-1$ is $\Irr(N)$, which is in bijection with $\Irr(\H(G,U))$ by Equation~\eqref{bij T2}. But looking at the presentation of $N_{d,n}$ from Subsection~\ref{yoko}, we note that it is not a semidirect product because of the relation~\rel{n4}. We would have a semidirect product if that relation were $\xi_{i}^{2}=1$. 
We introduce the group $S_{d,n}$, whose group algebra is a deformation of that of $N_{d,n}$.

\medskip
Let $G$ be a Chevalley group associated with one of the classical root systems. We define the group $S_{d,n}$ generated by the elements $\{t_{j},\xi_{i}: j\in [n],i\in [l]\}$, where $\Delta=\{a_{1},\ldots,a_{l}\}$ is the set of simple roots, $n$ is the rank of the associated root system, subject to the relations:

\noindent
\begin{minipage}[t]{0.32\textwidth}
\begin{enumerate}[label=(S\arabic*), ref=(S\arabic*)]
    \item $t_j t_k = t_k t_j$. 
    \item $t_j^{d} = 1$. 
\end{enumerate}
\end{minipage}%
\hfill
\begin{minipage}[t]{0.32\textwidth}
\begin{enumerate}[label=(S\arabic*), ref=(S\arabic*)]
    \setcounter{enumi}{2}
    \item $\xi_i t_j = t_{s_{i}(H_j)} \xi_i$.   
    \item $\xi_{i}^2=1$.    
\end{enumerate}
\end{minipage}
\hfill
\begin{minipage}[t]{0.32\textwidth}
\begin{enumerate}[label=(S\arabic*), ref=(S\arabic*)]
    \setcounter{enumi}{4}
    \item $\underbrace{\xi_i \xi_j \cdots}_{\text{$m_{ij}$}} = \underbrace{\xi_j \xi_i \cdots}_{\text{$m_{ij}$}}$,    
\end{enumerate}
\end{minipage}

\medskip
\noindent
where $m_{ij}$ is the order of $s_{i}s_{j}$ in $W$, and we are using Notation~\ref{not t}. For the exceptional case of the Chevalley group constructed from the semisimple Lie algebra of $\Sl_{n}$ we set the same relations except for relations~\rel{S1} and~\rel{S2}, which are changed by the relations of the group of diagonal matrices of determinant one over the ring $C_{d}$. 
To obtain a bijection between $\Irr(N_{d,n})$ and $\Irr(S_{d,n})$, we will define a generic deformation of $\C[N_{d,n}]$, that we will call $Z_{d,n}$.

Let $G$ be a Chevalley group associated with one of the classical root systems. If $u$ is an indeterminate over $\C$ we define the $\C[u^{\pm 1}]$-algebra $Z_{d,n}$ generated by the elements $\{t_{j},\xi_{i}: j\in [n],i\in [l]\}$, where $\Delta=\{a_{1},\ldots,a_{l}\}$ is the set of simple roots,  $n$ is the rank of the associated root system, subject to the relations:

\noindent
\begin{minipage}[t]{0.32\textwidth}
\begin{enumerate}[label=(Z\arabic*), ref=(Z\arabic*)]
    \item $t_j t_k = t_k t_j$. 
    \item $t_j^{d} = 1$. 
\end{enumerate}
\end{minipage}%
\hfill
\begin{minipage}[t]{0.32\textwidth}
\begin{enumerate}[label=(Z\arabic*), ref=(Z\arabic*)]
    \setcounter{enumi}{2}
    \item $\xi_i t_j = t_{s_{i}(H_j)} \xi_i$.   
    \item $\xi_{i}^2=ut_{H_{a_{i}}}^{d/2}+(1-u)$.
\end{enumerate}
\end{minipage}
\hfill
\begin{minipage}[t]{0.32\textwidth}
\begin{enumerate}[label=(Z\arabic*), ref=(Z\arabic*)]
    \setcounter{enumi}{4}
    \item $\underbrace{\xi_i \xi_j \cdots}_{\text{$m_{ij}$}} = \underbrace{\xi_j \xi_i \cdots}_{\text{$m_{ij}$}}$,    
\end{enumerate}
\end{minipage}
\medskip

\noindent
where $m_{ij}$ is the order of $s_{i}s_{j}$ in $W$,
and we are using Notation~\ref{not t}. For the exceptional case of the Chevalley group constructed from the semisimple Lie algebra of $\Sl_{n}$ we take the same relations except for relations~\rel{Z1} and~\rel{Z2}, which are changed by the relations of the group of diagonal matrices of determinant one over the ring $C_{d}$. From now on, we 
adopt the following notation.
\begin{notation}\label{not hi}
For the algebras $Y_{d,n},Z_{d,n},\C[N_{d,n}],\H(G,U)$ and $\C[S_{d,n}]$, we 
denote the element $t_{H_{a_{i}}}^{d/2}$ by $h_{i}(-1)$. Observe that $h_{i}(-1)^{2}=1$, since $t_{H_{a_{i}}}$ has order $d$.
\end{notation}

Setting $u=1$ in $Z_{d,n}$, we obtain the relations of $\C[N_{d,n}]$, while when setting $u=0$ we obtain the relations of $\C[S_{d,n}]$. Then, we have
\[Z_{d,n}(1)\cong \C[N_{d,n}],\quad \text{and} \quad   Z_{d,n}(0)\cong \C[S_{d,n}].
\]
We consider the specializations at $u=1$ and $u=0$, and we call them $\tau_{1},\tau_{0}:\C[u^{\pm 1}]\rightarrow \C$. Using the same argument as above with $\mathbb{K}$ a finite Galois extension of $\C(u)$, and applying Tits deformation theorem \textup{\cite[Theorem~7.4.6]{geck2000characters}}, we obtain bijections
\begin{align*}
& d_{\tau_0}:\Irr(\mathbb{K}Z_{d,n})\xrightarrow{\sim} \Irr(\mathbb{K}Z_{d,n}(0))=\Irr(S_{d,n}),\\ & d_{\tau_1}:\Irr(\mathbb{K}Z_{d,n})\xrightarrow{\sim} \Irr(\mathbb{K}Z_{d,n}(1))=\Irr(N_{d,n}).
\end{align*}
We call $\tau_1,\tau_0:A^{\ast} \rightarrow \C$ the specializations of $u$ at $1$ and $0$ respectively, extended to the integral closure $A^{\ast}$ of $A=\C[u^{\pm 1}]$ in $\mathbb{K}$. Moreover, if $\chi: \mathbb{K}Z_{d,n}\rightarrow A^{\ast}$ is a character, we obtain the specialized characters
\begin{center}
$d_{\tau_0}(\chi)=\tau_0\circ \chi$,\quad \text{and} \quad   $d_{\tau_1}(\chi)=\tau_1 \circ \chi$.
\end{center}
In particular, we can define a bijection
\begin{equation}\label{bij T}
T=d_{\tau_{0}}\circ d_{\tau_{1}}^{-1}:\Irr(N_{d,n})\rightarrow \Irr(S_{d,n}).
\end{equation}
In particular, 
$\Irr(N_{d,n})$ is also parametrized by $\Irr(S_{d,n})$. The following diagram shows the situation:
\begin{equation}
\begin{tikzcd}
 \Irr\pare{\mathbb{K}Y_{d,n}} \arrow{dr}{d_{\phi_{1}}} & & \Irr\pare{\mathbb{K} Z_{d,n}}\arrow{dr}{d_{\tau_{0}}} \arrow{dl}[swap]{d_{\tau_{1}}} &\\
  & \Irr\pare{N_{d,n}} & & \Irr\pare{S_{d,n}}
\end{tikzcd}\label{diagrama}\end{equation}

And if $d=q-1$, we have 

\begin{equation}
\begin{tikzcd}
& \Irr\pare{\mathbb{K}Y_{q-1,n}} \arrow{dr}{d_{\phi_{1}}} \arrow{dl}[swap]{d_{\phi_{q}}} & & \Irr\pare{\mathbb{K} Z_{q-1,n}} \arrow{dr}{d_{\tau_{0}}} \arrow{dl}[swap]{d_{\tau_{1}}}& \\
\Irr\pare{\mathcal{H}(G,U)} & & \Irr\pare{N} & &  \Irr(S_{q-1,n})
\end{tikzcd} \label{diagrama 2}\end{equation}
where the arrows are bijections. In the following section, we compute $\Irr(S_{d,n})$ using the characterization of the irreducible representations of a semidirect product. 

\subsection{Basic properties of $Y_{d,n}$}\label{basic prop}

In this subsection, we 
present some basic properties of $Y_{d,n}$. Recall the notation from \ref{gen yoko}, and Notation~\ref{not hi}. It is clear that $h_{i}(-1)^{2}=1$ and that $e_{i}$ is idempotent since $t_{H_{a_{i}}}$ has order $d$. 

\begin{lemma}\label{conmute}
The idempotent $e_i$ commutes with $\xi_i$, and 
satisfies that  $e_{i}h_{i}(-1)=h_{i}(-1)e_{i}=e_{i}$.
\end{lemma}

\begin{proof}
Using relation~\rel{Y3} from \ref{gen yoko}, we have that  
\[(\underset{j}{\sum}t_{H_{a_i}}^{j})\xi_i=\xi_i (\underset{j}{\sum}t_{s_i(H_{a_i})} ^{j})=\xi_i(\underset{j}{\sum}t_{-H_{a_i}}^{j})=\xi_i(\underset{j}{\sum}t_{H_{a_i}}^{-j})=\xi_i(\underset{j}{\sum}t_{H_{a_i}}^{j}),\]
where the last equality holds because $t_{H_{a_i}}$ has order $d$. Then, $\xi_i$ and $e_i$ commute. We also have that
\[h_i(-1)  \sum_{j} t_{H_{a_i}}^{j}=t_{H_{a_i}}^{d/2} \sum_{j} t_{H_{a_i}} ^{j}= \underset{j}{\sum} t_{H_{a_i}} ^{j},\] which implies that 
\[h_{i}(-1) e_i=e_i h_{i}(-1)=e_i.\]
\end{proof}

Let $G$ be a Chevalley group with Weyl group $W$ of one of the classical types. The group $W$ has a maximal length element $w_{0}$, and its length is the number of positive roots of the associated root system. In type $A$ the element $w_{0}$ is the anti-diagonal $\textstyle \prod_{i=1}^{\floor*{n/2}}(i,n+1-i)$, for types $B$ and $C$ the element $w_{0}$ is $-\Id$, and for type $D$ the element $w_{0}$ is $-\Id$ if the rank of the root system is even and it is $\Diag(1,-1,\ldots,-1)$ otherwise. We may choose a reduced expression 
\[w_{0}=s_{i_{1}}\cdots s_{i_{k}}.\]
With the same indices, we define



\begin{equation}\label{tw0}
\begin{array}{ccc}
T_0=\xi_{i_1} 
\cdots \xi_{i_k}\in Y_{d,n},&
\sigma_{0}=\xi_{i_1}
\cdots \xi_{i_k}\in N_{d,n}, &
 T_{0}'=\xi_{i_1} 
 \cdots \xi_{i_k}\in Z_{d,n},  \\
\widetilde{w_{0}}=\xi_{i_1}
\cdots \xi_{i_k}\in S_{d,n},&
\quad T_{w_0}=\xi_{i_{1}}
\cdots \xi_{i_{k}}\in \mathcal{H}(G,U). &
\end{array}
\end{equation}

Using the braid relations of $W$, and arguing as in Matsumoto theorem \textup{\cite[Theorem~1.2.2]{geck2000characters}}, we can see that the definition of these elements is independent of the election of the reduced expression of $w_0$. The specializations of Subsection~\ref{gen yoko} show that $T_0$ corresponds with $T_{w_{0}}$ and $\sigma_{0}$ with the specializations at $q^{-1}$ and $1$, respectively.
\begin{lemma}\label{central}
The element $T_{0}^2$ is central in $Y_{d,n}$. It follows from the specialization that $T_{w_0}^2$ is central in $\mathcal{H}(G,U)$. \label{T0^{2} es central}
\end{lemma}
\begin{proof}
Following the argument of \textup{\autocite[§ 4.1]{geck2000characters}}, we define the monoid $B^{+}$ (we call it the same as the braid monoid) generated by $\xi_i$ and $t_j$, and subject to the relations~\rel{Y1},~\rel{Y3} and~\rel{Y5} 
from \ref{gen yoko}. 
Then, we define the monoid algebra $\C[u^{\pm 1}][B^{+}]$, of which $Y_{d,n}$ is a quotient. If we take 
\[T_0=\xi_{i_1} 
\cdots \xi_{i_k}\in \C [u^{\pm 1}][B^{+}],\]
 it is enough with showing that $T_0^2$ is central in $\C[u^{\pm 1}][B^{+}]$. Following the argument of \textup{\cite[Lemma~4.1.9]{geck2000characters}}, we get that $T_{0}^{2}$ commutes with the elements $\xi_{i}$. On the other hand, we have that $t_{H} T_0^{2}=T_0^{2}t_{w_0^{2}(H)}=T_{0}^{2}t_{H}$, where we are using that $w_0^2=1$, which holds by \textup{\cite[Lemma~1.5.3]{geck2000characters}}. Therefore, the element $T_0^{2}$ is central.\end{proof}
\begin{remark}\normalfont
Making an analogous argument for $Z_{d,n}$, we obtain that  $T_{0}'^{2}$ is central.
\end{remark}

\subsection{Explicit construction of Chevalley groups}\label{explicit}

In this subsection, for each classical type, we explicitly construct the Chevalley group 
and we describe which group we obtain with this procedure. 
We follow the definitions and notations of Subsection~\ref{cheva}, and for the following constructions, we always 
consider the defining representation of $\g\subset \C^{n\times n}$, which is given by the left multiplication of $\g$ over $V=\C^{n}$. 

\medskip
{\bf Type $A$:} 
We consider the Lie algebra of type $A$, which is $\g=\{A: A\in M_{n}(\C)\}$. We take the vectors $v_{i}$ as the canonical vectors and the matrices $H_{i}$ as $E_{i,i}$, from the definition of the Chevalley group of \ref{cheva}. The obtained Chevalley group for a fixed prime power $q$ is the group $\Gl_{n}(\FF_{q})$. 

\medskip
{\bf Type $A^{\ast}$:} We 
denote by $A^{\ast}$ the Chevalley group constructed from the semisimple Lie algebra of type $A$, which is $\g=\{A\in M_{n}(\C):\tr(A)=0\}$. Again we consider the defining representation and the same election of vectors $v_{i}$ and matrices $H_{i}$ as in the previous case. The obtained Chevalley group for a fixed prime power $q$ is the group $\Sl_{n}(\FF_{q})$.

\medskip
{\bf Type $C$:} We consider the Lie algebra of type $C$, which is $\sp_{2n}(\C)=\{A\in M_{2n}(\C):\Omega A+A^{t}\Omega=0\}$, where \[\Omega=\begin{psmallmatrix}
0 & \Id_{n} \\[6pt]
-\Id_{n} & 0
\end{psmallmatrix}.\]
Following the construction of \textup{\cite[§~5]{ree1957some}}, we take the defining representation, and the matrices $H_{i}=E_{i,i}-E_{2n-i,2n-i}$. Using the results from \textup{\cite[§~5]{ree1957some}}, we obtain that the Chevalley group obtained for a fixed prime power $q$ is the group $\Sp_{2n}(\FF_{q})$.

\medskip
{\bf Type $B$:} We consider the Lie algebra of type $B$, which is \[\mathfrak{so}_{2n+1}(\C)=\{A\in M_{2n+1}(\C):JA+A^{t}J=0\},\]
where $J=\text{AntiDiag}(1,\ldots,1,2,1,\ldots,1)$, the anti-diagonal matrix with a two in the central entry and ones in the remaining entries of the anti-diagonal.
Following the construction of \textup{\cite[§~6]{ree1957some}}, we take the defining representation, 
with the matrices $H_{i}=E_{i,i}-E_{2n+1-i,2n+1-i}$. Using the results from \textup{\cite[§~6]{ree1957some}}, we obtain that the Chevalley group obtained for a fixed prime power $q$ is the commutator of the group $\SO_{2n+1}(\FF_{q})$, defined with respect to the matrix $J$, also denoted by $\SO_{2n+1}(\FF_{q})\rq$.

\medskip
{\bf Type $D$:} We consider the Lie algebra of type $D$, which is  
\[\mathfrak{so}_{2n}(\C)=\{A\in M_{2n}(\C): LA+A^{t}L=0\},\]
where $L=\text{AntiDiag}(1,\ldots,1)$, the anti-diagonal matrix with ones in the non-zero entries. Following the construction of \textup{\cite[§~7]{ree1957some}}, we take the defining representation, 
with the matrices $H_{i}=E_{i,i}-E_{2n-i,2n-i}$. Using the results from \textup{\cite[§~7]{ree1957some}}, we obtain that the Chevalley group obtained for a fixed prime power $q$ is the commutator of the group $\SO_{2n}(\FF_{q})$, defined with respect to the matrix $L$, also denoted by $\SO_{2n}(\FF_{q})\rq$.

\begin{notation}\label{not V}
Henceforth, when we refer to the Chevalley group associated with some classical type, we will be considering the Chevalley groups constructed in this subsection. For $V=A,A^{\ast},B,C$ and $D$ let $Y_{d,n}^{V},Z_{d,n}^{V},N_{d,n}^{V}$ and $S_{d,n}^{V}$ denote $Y_{d,n},Z_{d,n},N_{d,n}$ and $S_{d,n}$, respectively constructed from the Chevalley group of type $V$.  

\end{notation}

\begin{remark}\normalfont \label{h indice 2}
Let $k$ be a field such that a quadratic form over $k^{n}$ vanishes in a non-zero vector. Then, in \textup{\cite[Lemma~6.21]{higman1978classical}}, it was proven that the commutator of $\SO_{n}(k)$, defined by this quadratic form, has index two.
We have that this is the situation for the quadratic forms given by $J$ and $L$, so we can affirm that for types $B$ and $D$  the Chevalley groups have index two in $\SO_{n}(\FF_{q})$.
\end{remark}

Now we 
prove a lemma that will be useful to prove Lemma~\ref{ktu}.

\begin{lemma}\label{iso indice} Let $G_{1}$ be a subgroup of index two of a group $G$, and let $H_{1}\leq G_{1}$ and $H\leq G$ subgroups such that $H_{1}\subset H$ and $H\not\subset G_{1}$. Assume that $\dim \H(G,H)=\dim \H(G_{1},H_{1})$, and that a system of representatives of $H$-double cosets in $G$ is also a system of representatives of $H_{1}$-double cosets in $G_{1}$. In that case, we have $\H(G,H)\cong \H(G_{1},H_{1})$.
\end{lemma}
\begin{proof}
We denote by $W$ the set of common representatives of both double cosets and let  $g\in H\backslash G_{1}$. Let $HwH$ be a double $H$-coset. We know that $H_{1}wH_{1}\amalg H_{1}wH_{1}g\subset HwH$, the union is disjoint since $H_{1}wH_{1}\subset G_{1}$ and $g\notin G_{1}$. Then, we have 
\[G_{1}\amalg G_{1}g=\coprod_{w\in W} H_{1}wH_{1}\amalg H_{1}wH_{1}g \subset \coprod_{w\in W}HwH=G.\]
Since $\pare{G:G_{1}}=2$, we have that the last inclusion is an equality, and in particular we have that 
\begin{equation} HwH=H_{1}wH_{1}\amalg H_{1}wH_{1}g.\label{Tw}\end{equation}
Let us call $\phi_{w}$ and $T_{w}$ the indicator functions of the cosets of $w$ in $G_{1}$ and $G$ respectively. Let us define \begin{align*}\sigma:\H(G,H)&\rightarrow  \H(G_{1},H_{1}),\\
T&\mapsto T|_{G_{1}}.\end{align*}
Since $H_{1}\subset H$, and $T$ is $H$-bivariate we have that $T|_{G_{1}}$ is also bivariate, which implies that the function is well-defined. Using Equation~\eqref{Tw}, we have that $T_{w}\mapsto \phi_{w}$, implying that $\sigma$ is a surjective function, and since the dimensions are equal we obtain that it is an isomorphism of vector spaces. To see that it is an isomorphism of algebras, it is enough to see that $\sigma$ respects the convolution product. Let $ f_{1},f_{2}\in \H(G,H)$ and $w\in W$, we have 
\begin{align*} 
 f_{1}\ast f_{2}(w)&=\frac{1}{\size{H}}\pare{\sum_{x\in G_{1}}f_{1}(x)f_{2}(x^{-1}w)+\sum_{y\in G_{1}g}f_{1}(y)f_{2}(y^{-1}w)}  \\
&= \frac{1}{\size{H}}\pare{\sum_{x\in G_{1}}f_{1}(x)f_{2}(x^{-1}w)+\sum_{y\in G_{1}g}f_{1}(yg)f_{2}(g^{-1}y^{-1}w)}  \\
&=\frac{1}{\size{H}} 2\sum_{x\in G_{1}}f_{1}(x)f_{2}(x^{-1}w)=\frac{2\size{H_{1}}}{\size{H}}\sigma(f_{1})\ast \sigma(f_{2})(w).
\end{align*}
Therefore, $\sigma$ respects the convolution if we rescale by a factor of $\textstyle \sqrt{2\size{H_{1}}/\size{H}}$. 
\end{proof}

\begin{lemma} Let $G$  be any of the groups $\SO_{2n+1}(\mathbb{F}_{q})$ or $\SO_{2n}(\mathbb{F}_{q})$, and let $G_{1}$  be the Chevalley group of the same type. Let $U$ and $U_{1}$ be  the maximal unipotents of $G$ and $G_{1}$, respectively. Then, there is an isomorphism $\H(G,U)\cong \H(G_{1},U_{1})$.\label{ktu}
\end{lemma}
\begin{proof}
Note that the groups $G$ and $G_{1}$ have the same maximal torus and their Weyl groups are isomorphic. Since  $N/T\cong W$, their maximal torus normalizers have the same cardinality. It is straightforward to see that the normalizer of the maximal torus of $G_{1}$ is included in the one of $G$, hence both coincide and we call it $N$. The Bruhat decomposition shows that $N$ is a system of representatives of double cosets of $U$ in $G$ and of $U_{1}$ in $G_{1}$. It is also clear that $U_{1}\subset U$, and the inclusion is strict since $G$ and $G_{1}$ would be equal otherwise. Then, Remark~\ref{h indice 2} and Lemma~\ref{iso indice} give us the isomorphism. Moreover, this isomorphism is the restriction to $G_{1}.$
\end{proof}

A natural question is if we can change the bilinear forms that we have for the Chevalley groups of types $B$ and $D$. We know that if two non-degenerate symmetric bilinear forms are equivalent then the orthogonal groups are conjugated, and therefore isomorphic, so in this case we can change one bilinear form with the other. 

From \textup{\cite[Chapter~3,Section~6]{artin2016geometric}}, we know that over a finite field $k$ of odd characteristic, there are exactly two equivalence classes of non-degenerate symmetric bilinear forms. 
More in detail, we have that two non-degenerate symmetric bilinear forms  $B_{1}$ and $B_{2}$ determined by the matrices $J_{1}$ and $J_{2}$ are equivalent if and only if $\det  (J_{1}J_{2})\in (k^{\times})^{2}$. Using this, we can deduce the following lemma.

\begin{lemma}
In type $D$, the matrix $L$ has determinant $(-1)^{\frac{n}{2}}$. Let  $k=\FF_{q}$, where $q$ is a prime, then $(-1)^{ \frac{n}{2}}$ is a quadratic residue module $q$ if and only if 
$n$ is a multiple of $4$ or $q$ is congruent to $1$ module $4$. For these cases, the bilinear form given by $L$ is equivalent to the canonical, so we can change it for the canonical.

In type $B$, the matrix  $J$ has determinant $2(-1)^{\frac{n-1}{2}}$, so this bilinear form is equivalent to the canonical if and only if $2(-1)^{\frac{n-1}{2}}\in (\mathbb{F}_{q}^{\times})^{2}$, and this happens in the following cases:
\begin{itemize}
\item $q \equiv 1 \bmod{  (8) }.$
\item $q \equiv 7 \bmod{ (8)}$ and $n\equiv 1 \bmod {(4)}.$
\item $q\equiv 3 \bmod {(8)}$ and $n\equiv 3 \bmod {(4)}.$
\end{itemize}
We can change the bilinear form given by $L$ with the canonical for these cases.
\end{lemma}
\begin{remark}\normalfont
If  $x\in \FF_{q}^{\times}$, we know that $x^{q-1}=1$, which implies that 
\[x^{\frac{q^{2}-1}{2}}=\pare{x^{q-1}}^{\frac{q+1}{2}}=1.\]
Therefore $x$ is a square in $\mathbb{F}_{q^{2}}^{\times}$. Then if we want to change the bilinear forms of types $B$ and $D$ with the canonical, for the cases in which they are not equivalent over $\FF_{q}$, we need to go to $\mathbb{F}_{q^{2}}$. Taking $k=\mathbb{F}_{q^{2}}$, and constructing the Chevalley group of types $B$ and $D$ over $k$ we get that the determinants of $J$ and $L$ belong to $\mathbb{F}_{q}$, then they are squares on $\mathbb{F}_{q^{2}}^{\times}$, so we can change these bilinear forms with the canonical.
\end{remark}

\section{Representations of \texorpdfstring{$S_{d,n}$}{S}}\label{sec 4}
In this section, we 
compute 
$\Irr(S_{d,n})$ using the characterization of 
$\Irr(W)$ given in Subsection~\ref{Representaciónes de los grupos de Weyl}. This will be necessary to compute the character values at some key elements in the next section.

\subsection{Representations of a semidirect product and of subgroups with cyclic quotient}\label{quo}
To characterize 
$\Irr(W)$ and $\Irr(S_{d,n})$, we 
first explain how to find the irreducible representations of a semidirect product and the irreducible representations of a normal subgroup $H\leq G$ with cyclic quotient, knowing those of $G$. Observe that $S_{d,n}$ has a decomposition as the semidirect product $C_{d}^{n}\rtimes W$, for all types other than $A^{\ast}$. 
\subsubsection{Representations of a semidirect product}\label{caracteres de producto semidirecto}
Now we will explain how to obtain 
$\Irr(G)$, for a group $G$ of the form $G=A\rtimes H$, with $A$ abelian. We 
follow the argument of \textup{\cite[Proposition~8.2]{serre1977linear}}. Since $A$ is abelian, all the elements of $\Irr(A)$ have degree $1$ and are identified with $X=\Hom(A,\C^{\times})$. The group $H$ acts on $X$ by 
\[(s\cdot\chi)(a)=\chi(s^{-1}as).\] 
Let $(\chi_i)_{i\in X/H}$ be a system of representatives of the orbits of $H$ in $X$. For each $i\in X/H$, let $H_i=\textbf{Stab}(\chi_{i})$ and we define $G_i=A H_i$. We extend $\chi_i$ to $G_i$ defining it by 
 \[\chi_i(ah)=\chi_i(a),\ \text{if}\  a\in A,\ \text{and}\ h\in H_i.\]
This determines a character of degree $1$ of $G_i$. Let $p\in \Irr(H_{i})$, 
we have a canonical projection $G_i\rightarrow H_i$, and composing $p$ with this projection we obtain $\tilde{p}\in \Irr(G_{i})$. 
Taking 
tensor product, we obtain $\chi_i \otimes \tilde{p}\in \Irr(G_{i})$. 
We induce this representation from $G_i$ to $G$, and we call the obtained representation $\phi_{i,p}$. 
Then, \textup{\cite[Proposition~8.2]{serre1977linear}} says that the representations $\phi_{i,p}$ are irreducible and that they form the set $\Irr(G)$. Moreover, if $\phi_{i,p}$ and $\phi_{i',p'}$ are isomorphic, then $i=i\rq$ and $p\cong p\rq$.

\begin{example}[Characters of $C_{d}^n\rtimes \mathbb{S}_n$]\label{caracteres de N en tipo A}
The group $N=C_{d}^n\rtimes \mathbb{S}_n$ satisfies the hypotheses of \textup{\cite[Proposition~8.2]{serre1977linear}}. To find $\Irr(N)$  we start by describing $X=\Hom\pare{C_{d}^n,\C^{\times}}$. The set $\Hom(C_{d},\C^{\times})$ is identified with $C_{d}$. Then, $\Hom(C_{d}^n,\C^{\times})$ consists in given a $n$-tuple of elements of $C_{d}$. If $\psi=(\psi_{1},
\ldots, \psi_{n})$ is one of these tuples, then the corresponding character is defined by \[\pare{t_1,
\ldots,t_n}\mapsto t_1^{\psi_1} 
\cdots t_n^{\psi_n}.\] 

The group $\mathbb{S}_n$ acts by conjugation on $C_{d}^n$ by permuting the entries of the diagonal. Therefore, the action of $\sigma \in \mathbb{S}_n$ on $X$ consists in permuting the entries of the $n$-tuples by applying $\sigma$. Hence, the orbits of the action are determined by the number of times that each element of $C_{d}$ appears in the tuple. Let $t_g$ be a generator of $C_{d}$ and for each orbit, we take the representative in which the elements  $t_g^{1},
\ldots,t_g^{d}$ are ordered from left to right. For $\psi\in C_{d}^{n}$ one of these representatives, we have that 
\[{\fontshape{ui}\selectfont {\bf Stab} (\psi)}=\mathbb{S}_{n_1}
\times \cdots \times \mathbb{S}_{n_{d}},\] where $n_j$ is the multiplicity of  $t_g^{j}$ in $\psi$. Then, the characters $\chi \in \Irr(\mathbf{Stab}(\psi))$ are of the form \[\mathbb{S}^{\lambda_1}\otimes 
\cdots \otimes \mathbb{S}^{\lambda_{d}},\]
where $\lambda_i \in \mathcal{P}_{n_i}$. This tells us that $\Irr(N)$ is parametrized by the functions $\Lambda:C_{d}\rightarrow \mathcal{P}$ of size $n$, that is $\textstyle \sum_{g\in C_{d}} \lvert \Lambda(g)\rvert=n$. We denote this set of functions by $\mathcal{Q}_{d,n}$, and we have a bijection between  $\Irr(N)$ and $\mathcal{Q}_{d,n}$ given by $\Lambda \mapsto \chi_{\Lambda}$. Now we will see which is the representation $\chi_{\Lambda}$. Let $N_i=C_{d}^{n_i} \rtimes \mathbb{S}_{n_i}$. Observe that if $\psi \in 
  C_{d}^{n}$  has multiplicity $n_i$ in $t_{g}^{i}$, then $C_{d}^{n}\rtimes \mathbf{Stab}(\psi)=\textstyle \prod_{i} N_i$. Then, $t_{g}^{i}$ defines a character in $N_i$ given by 
  \[\pare{t_1,
  \ldots ,t_{n_i},\sigma}\mapsto \prod_{j=1}^{n_{i}} t_{j}^{i}.\]
On the other hand, $\S^{\lambda_i}\in \Irr(\mathbb{S}_{n_i})$, and taking the canonical projection we obtain an element of $\Irr(N_i)$. Taking the tensor product, we define 
\[\chi_{\lambda_i,t_{g}^{i}}^{N_i}=\mathbb{S}^{\lambda_i}\otimes t_{g}^{i} \in \Irr(N_i).\]
Taking the tensor product and inducing to $N$ we obtain the irreducible character
\begin{equation}
\chi_{\Lambda}=\Ind_{\prod N_i}^{N} \underset{i}{\bigotimes}  \chi_{\lambda_i,t_{g}^{i}}^{N_i}\in \Irr(N).
\label{expresion irreducibles de N}
\end{equation}

\begin{notation}\label{not}
With more generality, we will denote by 
$\chi_{\lambda_i,\phi}^{N_i}$ the representation  
\[\mathbb{S}^{\lambda_i}\otimes \phi \in \Irr(N_i),\]
for $\lambda_{i}\in \mathcal{P}_{n_{i}}$ and $\phi \in C_{d}$.
\end{notation}

\end{example}

\subsubsection{Representations of subgroups with cyclic quotient}\label{caso ciclico}

\normalfont Suppose that $N\unlhd G$ with $G/N$ cyclic. We want to give a description of $\Irr(N)$ in terms of $\Irr(G)$. If $\mu \in \Irr(N)$, we denote by $\mu^{g}$ the character obtained by conjugating by  $g$. Using \textup{\autocite[Theorems~2.14,~3.2]{feit1982representation}}
 we have that for each $\chi\in \Irr(G)$, and $\mu$ an irreducible factor of $\restr{\chi}{N}$, we have
\[\restr{\chi}{N}= \sum_{i=1}^{t}\mu^{g_{i}},\]
where the $\mu^{g_{i}}$, for $i\in [t]$, are the conjugates of  $u$. 
We have that every $\mu \in \Irr(N)$ appears as an irreducible factor of $\restr{\chi}{N}$, for some $\chi \in \Irr(G)$. We call $d=\size{G/N}$, and we denote by $\epsilon$ the morphism 
\[\epsilon:G\rightarrow  C_{d}\cong  G/N.\]
We also have that if $\chi$ is a factor of $\Ind_{N}^{G}(\mu)$, then the remaining factors of $\Ind_{N}^{G}(\mu)$ are of the form $\chi \otimes \epsilon^{j}$ for some $j\geq 0$. Moreover, we have that there are exactly $d/t$ of them. On the other hand, observe that the number of different characters in $\Ind_{N}^{G}\mu$  is the order of the sequence of characters 
\[\chi,\chi \otimes \epsilon,\ldots ,\chi \otimes \epsilon^{k},\ldots \]
where $\chi$ is a character in $\Ind_{N}^{G}\mu$. Given $\chi \in \Irr(G)$ we call $o(\chi)$ to this number, and we also have that $o(\chi)=d/t$, where $t$ is  the number of irreducible factors in $\restr{\chi}{N}$.

\begin{example}\label{Sl_{n} caracteres}\normalfont
An example in which we have this situation is with 
\[N=C_{d}^{n}\rtimes \S_{n},\quad \text{and} \quad H=T'\rtimes \S_{n},\]
where $T'$ is the subgroup of $C_{d}^{n}$ of diagonal matrices of determinant $1$. The group $\S_{n}$ acts by conjugation on 
$C_{d}^{n}$ by permuting its entries. We consider the morphism
$\epsilon: C_{d}^{n}\rtimes \S_{n}\rightarrow C_{d}$, defined by 
\[s_{i}\mapsto 1,\quad \text{and} \quad t\in C_{d}^{n}\mapsto \det(t),\]
which is well-defined if we see the relations on the generators. We have 
$\Ker(\epsilon)=H,$ which implies that $N/H\cong C_{d}$.
Given  $\chi_{\Lambda} \in \Irr(C_{d}^{n}\rtimes \S_{n})$, where $\Lambda=(\lambda_{1},
\ldots ,\lambda_{d})\in \mathcal{Q}_{d,n}$ we will now compute $o(\chi)$. For doing that, we have to see what is $\chi_{\Lambda} \otimes \epsilon$. Using Equation~\eqref{expresion irreducibles de N}, we have that
\[\chi_{\Lambda}=\Ind_{\prod N_i}^{N} \underset{i}{\bigotimes} \ \chi_{\lambda_i,t_{g}^{i}}^{N_i}.\] 
Then, we have 
\[\Ind_{\prod N_i}^{N} \underset{i}{\bigotimes} \chi_{\lambda_i,t_{g}^{i}}^{N_i}\otimes \epsilon=\Ind_{\prod N_i}^{N}\pare{\underset{i}{\bigotimes} \chi_{\lambda_i,t_{g}^{i}}^{N_i} \otimes \Res_{\prod N_i}^{N} (\epsilon)}.\]
Therefore, after tensoring with  $\epsilon$, we obtain 
\[\Ind_{\prod N_i}^{N}\underset{i}{\bigotimes}  \chi_{\lambda_i,t_{g}^{i+1}}^{N_i} \, ,\]
where we are using Notation~\ref{not}. After tensoring $k$ times, we obtain 
\[\Ind_{\prod N_i}^{N}\pare{\underset{i}{\bigotimes}  \chi_{\lambda_i,t_{g}^{i+k}}^{N_i}}.\]
Therefore the order of the sequence of characters
$\chi_{\Lambda},\chi_{\Lambda} \otimes \epsilon,\ldots ,\chi_{\Lambda} \otimes \epsilon^{k},\ldots $
is the same as the order of the sequence of partitions 
$(\lambda_{1},
\ldots ,\lambda_{d}),$
and we call this number $o(\Lambda)$.

\end{example}

\subsection{Representations of Weyl groups}\label{Representaciónes de los grupos de Weyl}
In this subsection, we 
use the results of Subsection~\ref{quo}  to characterize 
$\Irr(W)$ for each type.

\medskip
{\bf Type A:} The 
set $\Irr(\mathbb{S}_n)$ is parametrized by 
$\mathcal{P}_{n}$. These representations are the Specht modules, 
\textup{\autocite[Chapter 10]{steinberg2012representation}}. We will denote the representation associated with $\lambda$ by $\mathbb{S}^{\lambda}$.

\medskip
\textbf{Types B and C}: The Weyl group of these types is $W_n$, defined in Subsection~\ref{iwa}. We know that $W_n=\set{\pm 1 }^{n}\rtimes \mathbb{S}_n$, where $\{\pm 1\}^n$ is the subgroup of diagonal matrices. Since $\{\pm 1 \}$ is cyclic, we can use the same argument of Example~\ref{caracteres de N en tipo A}, and we have that $\Irr(W_n)$ is parametrized by the pairs of partitions $(\lambda_1,\lambda_2)$ such that $\lvert \lambda_1 \rvert+ \lvert \lambda_2 \rvert=n$. For the pair $(\lambda_1,\lambda_2)$, the corresponding representation is
\[\Ind_{W_{n_1}\times W_{n_2} }^{W_n} \mathbb{S}^{\lambda_1}\otimes (\mathbb{S}^{\lambda_2}\otimes \sgn),
\label{chi en tipo B}
\]
where $n_i=\lvert \lambda_i \rvert$ for $i=1,2$, and $\sgn\in \Rep(W_{n_{2}})$ is 
defined by 
\[(t_1,
\ldots,t_{n_2}, \sigma)\mapsto t_1
\cdots t_{n_2}.\]
We will denote this representation by $\chi^{}_{\lambda_{1},\lambda_{2}}$.

\medskip
{\bf Type D:} The Weyl group of this type is $W_n\rq$, defined in Subsection~\ref{iwa}. We know that $W_{n}'=N_{n}'\rtimes \mathbb{S}_n$, where $N_{n}'$ are the diagonal matrices with determinant one. Then we are in the situation of Example~\ref{Sl_{n} caracteres}. Let $\chi_{\lambda,\mu}\in \Irr(W_{n})$, then we have
\begin{align*}
 \chi_{\lambda,\mu} \otimes \epsilon&=(\Ind_{W_{n_1}\times W_{n_2} }^{W_n} \mathbb{S}^{\lambda_1}\otimes (\mathbb{S}^{\lambda_2}\otimes \sgn))\otimes \epsilon \cong  
  \Ind_{W_{n_1}\times W_{n_2} }^{W_n} (\mathbb{S}^{\lambda_1}\otimes (\mathbb{S}^{\lambda_2}\otimes \sgn)) \otimes \Res_{W_{n_1}\times W_{n_2}}^{W_{n}}  \epsilon \\&=\Ind_{W_{n_1}\times W_{n_2} }^{W_n} (\mathbb{S}^{\lambda_1}\otimes \sgn) \otimes \mathbb{S}^{\lambda_2}=\chi_{\mu,\lambda}. 
\end{align*}

We denote by $\chi_{\lambda,\mu}'$ the restriction of $\chi_{\lambda,\mu}$ to $W_{n}'$. From what we have seen in Example~\ref{Sl_{n} caracteres}, we know that the elements of $\Irr(W_{n}\rq)$ are:
\begin{enumerate}
\item $\chi_{\lambda,\mu}'=\chi_{\mu,\lambda}'$, if $\lambda \neq \mu.$ 
\item $\chi_{\lambda,\lambda}^{+},\chi_{\lambda,\lambda}^{-}$, where these are the irreducible factors of $\chi_{\lambda,\lambda}'$ in the case of $n$ being even.
\end{enumerate}

\subsection{Representations of {$S_{d,n}$}}\label{rep normalizador}
The goal of this subsection is to use the results of the previous subsections to characterize 
$\Irr(S_{d,n})$. Since by Equation~\eqref{bij T} the sets $\Irr(S_{d,n})$ and $\Irr(N_{d,n})$ are in bijection, we would characterize in this way 
$\Irr(N_{d,n})$. Recall from Subsection~\ref{gen nor} that for each type other than $A^{\ast}$, the group $S_{d,n}$ is defined as the semidirect product $C_{d}^{n}\rtimes W$, where $W$ acts on $C_{d}^{n}$ by conjugation, 
on the same way that the Weyl group 
$W$ acts on $\{1,\ldots,n\}$. 
For 
type $A^{\ast}$, the group $C_{d}^{n}$ is changed by the group of diagonal matrices of determinant one. We will now compute $\Irr(S_{d,n})$ for each type. Recall Notation~\ref{not V}.

\medskip
{\bf Type $A$:} In this case $\S_{n}$ acts by conjugation on $C_{d}^{n}$ in the same way that in Example~\ref{caracteres de N en tipo A}. Then, the irreducible representations are parametrized by the set $\mathcal{Q}_{d,n}$, of functions $\Lambda: C_{d}\rightarrow \mathcal{P}$ of size $n$. 

\medskip
{\bf Type $A^{\ast}$:} In this case, the group is $T'\rtimes \, \S_{n}$, where $T'$ is the group of diagonal matrices of determinant one in $C_{d}^{n}$. From what we saw in Example~\ref{Sl_{n} caracteres}, we know that these irreducible representations are factors of the restriction of some representation of $\Irr(C_{d}^{n}\rtimes \S_{n})$. We also have that the restriction of $\chi_{\Lambda}$ to $T'\rtimes \S_{n}$ decomposes in $d/o(\Lambda)$ different irreducible conjugated factors.

\medskip
{\bf Types $B$ and $C$:}\label{caracteres de tipo B}
In this case the group is $C_{d}^{n}\rtimes W_{n}$, where the group $W_{n}$ acts on $C_{d}^{n}$ by permuting and inverting its elements, as described in Subsection~\ref{gen nor}.  The elements $\pm 1$ are the only elements of $C_{d}^{n}$ that are preserved when inverting (we always assume $d$ even). Following the argument of Subsection~\ref{caracteres de producto semidirecto}, we can identify the orbits of this action by the number of entries that the tuple of $C_{d}^{n}$ has on each of the following groups:  
\[\set{1},\ \set{-1},  \set{t_g^{1},t_g^{-1}}, \ldots, \set{t_g^{k},t_g^{-k}},\ldots\]
where $t_g$ is a generator of $C_{d}$. We take as a representative of each orbit the tuple that orders the elements from left to right, respecting this order of the groups and making only one of the two numbers of each group appear.
Suppose that $n_1,
\ldots,n_{(d+2)/2}$ are the quantities that determine the orbit, now we will see what the stabilizer of this orbit is. If $j>2$, any element in the stabilizer has to permute the $n_{j}$ elements of the group $j$ among them, since this element is not an inverse of itself. While for $j=1,2$, any element in the stabilizer has to permute the $n_{j}$ elements of the group $j$ among them, but also with the possibility of inverting them, since they are their own inverses. From this argument, we deduce that the stabilizer is 
\[W_{n_1}\times W_{n_2} \times \mathbb{S}_{n_3} \times \cdots \times \mathbb{S}_{n_{(d+2)/2}}. \]
From what we have seen in Subsection~\ref{Representaciónes de los grupos de Weyl}, we have that the irreducible representations of this group are of the form
\[\chi_{\lambda_1,\mu_1} \otimes \chi_{\lambda_2,\mu_2} \otimes \mathbb{S}^{\lambda_3}
\otimes \cdots \otimes \mathbb{S}^{\lambda_{(d+2)/2}},\] 
where $\lambda_{i}\in \mathcal{P}_{n_{i}}$ for $i>2$, and $\size{\lambda_i }+\size{\mu_i}=n_i$ for $i=1,2$. Let $\psi_i$ be the element of $C_{d}$ that appears in the $n_i$ positions of the group $i$, where for example $\psi_1=1$ and $\psi_2=-1$. Let $N_i=C_{d}^{n_{i}} \rtimes W_{n_i}$ for $i=1,2$, and let $N_i=C_{d}^{n_{i}} \rtimes \mathbb{S}_{n_i}$ for $i>2$. Following the argument of Subsection~\ref{caracteres de producto semidirecto}, for $i>2$ we define 
\[\chi_{\lambda_i,\psi_{i}}^{N_i}=\mathbb{S}^{\lambda_i}\otimes \psi_i \in \Irr(N_i),\]
and for $i=1,2$ \[\chi_{\lambda_i,\mu_i,\psi_i}^{N_i}=\chi_{\lambda_i,\mu_i}^{} \otimes \psi_i \in \Irr(N_i).\]
Taking the tensor product and inducing to $N=S_{d,n}^{B}$, we obtain the irreducible representation
\begin{equation}\label{chi b}\chi_{\Lambda}=\Ind_{\prod N_i}^{N}(\chi_{\lambda_1,\mu_1,1}^{N_1} \otimes \chi_{\lambda_2,\mu_2,-1}^{N_2})   \underset{i>2}{\bigotimes}  \chi_{\lambda_i,\psi_{i}}^{N_i}\in \Irr(N).\end{equation}
We know that for types $B$ and $C$ all the representations of $\Irr(S_{d,n})$ are of this type, so they are parametrized by the tuples of partitions 
\[ (\lambda_{1},\mu_{1},\lambda_{2},\mu_{2},\lambda_{3},\ldots ,\lambda_{(d+2)/2})\]
of total size $n$. We call this set $\mathcal{R}_{d,n}$ and we denote by $\chi_{\Lambda}$ the irreducible representation associated with $\Lambda \in \mathcal{R}_{d,n}$.

\medskip
{\bf Type $D$:} In this case, the group is $C_{d}^{n}\rtimes W_{n}\rq$. Observe that from the relations defining the group $C_{d}^{n}\rtimes W_{n}$ we can extend the morphism $\epsilon:W_{n}\rightarrow \{\pm 1\}$ to all the group $C_{d}^{n}\rtimes W_{n}$  by defining $\epsilon$ trivially in $C_{d}^{n}$. Then, we have that $\Ker(\epsilon)=C_{d}^{n}\rtimes W_{n}\rq$, which implies that this subgroup has index two in $C_{d}^{n}\rtimes W_{n}$. Since $S_{d,n}^{D}$ is a subgroup of index two in $S_{d,n}^{B}$, by using Subsection~\ref{caso ciclico}, we know that all the representations in $\Irr(S_{d,n}^{D})$ appear as factors of the restrictions of the representations in $\Irr(S_{d,n}^{B})$. For $\Lambda=(\lambda_{1},\mu_{1},\lambda_{2},\mu_{2},\lambda_{3},\ldots,\lambda_{(d+2)/2})\in \mathcal{R}_{d,n}$ we consider the representation $\chi_{\Lambda}$ of $N=S_{d,n}^{B}$,
\[\chi_{\Lambda}=\Ind_{\prod N_i}^{N}(\chi_{\lambda_1,\mu_1,1}^{N_1} \otimes \chi_{\lambda_2,\mu_2,-1}^{N_2})   \underset{i>2}{\bigotimes}  \chi_{\lambda_i,\psi_{i}}^{N_i},\]
where $N_{i}=C_{d}^{n_{i}}\rtimes W_{n_{i}}$ for $i=1,2$, and $N_{i}=C_{d}^{n_{i}}\rtimes \S_{n_{i}}$ for $i>2$. If we tensor with $\epsilon$, we obtain 
\begin{align*}
\chi_{\Lambda}\otimes \epsilon&=\Ind_{\prod N_i}^{N}\pare{\pare{\chi_{\lambda_1,\mu_1,1}^{N_1} \otimes \chi_{\lambda_2,\mu_2,-1}^{N_2}}  \underset{i>2}{\bigotimes}  \chi_{\lambda_i,\psi_{i}}^{N_i}}\otimes \epsilon \\
&=\Ind_{\prod N_i}^{N}\pare{\pare{\chi_{\lambda_1,\mu_1,1}^{N_1} \otimes \chi_{\lambda_2,\mu_2,-1}^{N_2}}  \underset{i>2}{\bigotimes}  \chi_{\lambda_i,\psi_{i}}^{N_i}\otimes \Res_{\prod N_i}^{N}(\epsilon)}\\
&=\Ind_{\prod N_i}^{N}\pare{\chi_{\mu_1,\lambda_1,1}^{N_1} \otimes \chi_{\mu_2,\lambda_2,-1}^{N_2}}   \underset{i>2}{\bigotimes}  \chi_{\lambda_i,\psi_{i}}^{N_i},
\end{align*}
since $\epsilon$ is trivial in $N_{i}$ for each $i>2$, and it exchanges the partitions 
in $N_{1}$ and $N_{2}$.
Then, if $(\lambda_{1},\lambda_{2})\neq (\mu_{1},\mu_{2})$ we have that $\chi_{\Lambda}\neq \chi_{\Lambda}\otimes \epsilon$, and the restriction of $\chi_{\Lambda}$ belongs to $\Irr(S_{d,n}^{D})$. 
While if $(\lambda_{1},\lambda_{2})= (\mu_{1},\mu_{2})$, we have that $\chi_{\Lambda}\otimes \epsilon=\chi_{\Lambda}$, which implies that the restriction of $ \chi_{\Lambda}^{}$ to $S_{d,n}^{D}$ decomposes into two irreducible factors $\chi_{\Lambda}^{\pm}$. Moreover, from what we saw in  Subsection~\ref{caso ciclico}, the representations $\chi_{\Lambda}$ and $\chi_{\Lambda}^{\pm}$ form $\Irr(S_{d,n}^{D})$. We denote by $\mathcal{S}_{d,n}$ the set that parametrizes $\Irr(S_{d,n}^{D})$.

\begin{notation}\label{blocks ni}
For $\Lambda=(\lambda_{1},\mu_{1},\lambda_{2},\mu_{2},\lambda_{3},\ldots,\lambda_{(d+2)/2}) \in \mathcal{R}_{d,n}$, we denote $n_{i}=\size{\lambda_{i}}+\size{\mu_{i}}$ for $i=1,2$, and $n_{i}=\size{\lambda_{i}}$ for $i>2$. We will sometimes denote $\Lambda\in \mathcal{R}_{d,n}$ by $(\Lambda_{1},\Lambda_{2},\lambda_{3},\ldots,\lambda_{(d+2)/2})$, where $\Lambda_{i}=(\lambda_{i},\mu_{i})$ for $i=1,2$.
For $\Lambda=(\lambda_{1},\ldots,\lambda_{d})\in \mathcal{Q}_{d,n}$, we denote $n_{i}=\size{\lambda_{i}}$.
\end{notation}


\subsection{Twisting by characters of  \texorpdfstring{$C_{d}$}{C} and \texorpdfstring{$C_{2}$}{C2} }\label{tensoriando}
This subsection aims to tensor the Iwahori--Hecke algebras of types $A$ and $B$ with the group algebras of the cyclic groups $C_{d}$ and $C_{2}$ respectively. We 
use this to 
compute $\Irr(Y_{d,n}^{A})$ and $\Irr(Y_{d,n}^{B})$, and to prove that they are split over $\C(u)$. To do this, we follow a similar argument to the one given in \textup{\cite[§~3.9.2]{hausel2019arithmetic}}. We extend the Iwahori--Hecke algebra in the following way. We consider the $\C[u^{\pm 1}]$-algebra  
\[H_{d,n}^{A}=H_n^{A}\corch{h}/(h^{d}-1)=H_n^{A}\otimes \C[C_{d}].\]
We denote by $H_{d,n}^{A}(u)$ the $\C(u)$-algebra $\C(u)\textstyle \otimes_{\C[u^{\pm 1}]}H_{d,n}^{A}$.
Since $H_{d,n}(u)$ is a tensor product of semisimple algebras, then it is semisimple and its characters are the cartesian product of the characters of each factor. Then, if $\lambda \in \mathcal{P}_n$ and $\psi \in \Irr(C_{d})$, we define the tensor product
\[\chi_{\lambda,\psi}^{H_{d,n}^{A}}=\chi_{\lambda}^{H_{n}^{A}}\otimes \psi \in H_{d,n}^{A}(u).\]
There exists a natural quotient 
\begin{equation} Y_{d,n}^A\rightarrow H_{d,n}^{A}, \label{cociente A} \end{equation}
defined by  $t_j\mapsto h$ and $\xi_{i}\mapsto T_{s_i}$. We can see from the 
presentation of $Y_{d,n}^A$ that the map is well-defined, since the elements $e_{i}$ and $h_{i}(-1)$ map to $1$. The specialization at $u=1$ gives us 
$\C[\mathbb{S}_n]\otimes \C[C_{d}]=\C[\mathbb{S}_n\times C_{d}].$ When specializing at $u=1$, the natural quotient gives us the surjection 
\begin{equation}
N_{d,n}^{A}\rightarrow \mathbb{S}_n\times C_{d}. \label{9}
\end{equation}
Looking at the presentations of $N_{d,n}^{A}$ and 
$\mathbb{S}_n$, we have that under this surjective map $\xi_i\mapsto s_{i},$ and $t \mapsto \det(t)$. 
Therefore, we can define the inflation 
\begin{equation} \Infl:\Irr(\mathbb{S}_n \times C_{d}) \rightarrow \Irr(N_{d,n}^{A}), \label{infl}\end{equation}
that we obtain by composing with the function of Equation~\eqref{9}. We have 
\[\Infl(\chi_{\lambda,\psi})(t \sigma)=\psi(\det(t))  \mathbb{S}^{\lambda}(\sigma).\]
If we consider $Z_{d,n}^{A}$, we also have a natural quotient
\begin{equation}Z_{d,n}^{A} \rightarrow \C[\mathbb{S}_n \times C_{d}], \label{cociente tipo A}\end{equation}
that maps $t_j$ to $h$. When specializing at $u=1$, this natural quotient gives us the surjective function $N_{d,n}^{A}\rightarrow  \mathbb{S}_n\times C_{d}.$
We have again that $\xi_i\mapsto s_{i}$ and $t\mapsto \det(t)$, so this is the same function as the one of Equation~\eqref{9}, and therefore if we consider 
$\Infl:\Irr(\mathbb{S}_n \times C_{d}) \rightarrow \Irr(N_{d,n}^{A})\label{inflacion Y_{d,n}_{A}},$
we obtain the same function as the one in Equation~\eqref{infl}. If we specialize at $u=0$, we obtain the surjective map $S_{d,n}^{A}\rightarrow \C[\mathbb{S}_n \times C_{d}],$
in which $\xi_i\mapsto s_{i}$ and $t\mapsto \det(t)$. Therefore, we consider
\begin{equation}\Infl:\Irr(\mathbb{S}_n \times C_{d}) \rightarrow \Irr(S_{d,n}^{A}), \label{Inflacion 1} \end{equation}
defined by 
\[\Infl(\chi_{\lambda,\psi})(t ,\sigma)=\psi(\det(t))  \mathbb{S}^{\lambda}(\sigma).\] 
Now we consider the $\C[u^{\pm 1}]$-algebra $H_{2,n}^{B}$ defined by
\[H_{2,n}^{B}=H_n^{B}\corch{h}/(h^{2}-1)=H_n^{B}\otimes \C[C_2].\]
We have that $H_{2,n}^B(u)$ is semisimple and its set of irreducible characters is the product of the irreducible characters of its factors. Looking at the presentation of $Y_{d,n}^B$ from Subsection~\ref{gen yoko}, we have a natural quotient
\begin{equation} Y_{d,n}^{B}\rightarrow H_{2,n}^{B}, \label{cociente B} \end{equation}
given by $t_j\mapsto h$, and $\xi_{i}\mapsto T_{\xi_{i}}$. Analogously, looking at the relations of $Z_{d,n}^{B}$, we can define a natural quotient 
\begin{equation}Z_{d,n}^{B}\rightarrow \C[W_n \times C_2],\label{cociente tipo B} \end{equation}
in which $t_j\mapsto h$, and $\xi_{i}\mapsto \xi_{i}$.  Then by specializing at $u=1$, both natural quotients give us the same surjective function
$\C[N_{d,n}^{B}]\rightarrow \C[W_n \times C_2],$
defined by  $\xi_i\mapsto \xi_i$ and $t$ maps to  $1$ or $h$  depending on whether $\det(t)$ is a square in $C_{d}$ or not, respectively. We denote this assignment by $\det(t)^{\ast}$. Then, if we compose with this map, we obtain 
\begin{equation}
\Infl:\Irr(W_n \times C_2)\rightarrow \Irr(N_{d,n}^{B})
\label{inflacion Y_{d,n}_{B}},
\end{equation}
defined by 
\[\Infl(\chi_{\lambda,\mu,\psi})(t \sigma)=\psi\pare{\det( t)^{\ast}}  \chi_{\lambda,\mu}(\sigma), \label{inflacion 2}\]
where $\psi$ is determined by the value of $\psi(h)=\pm 1$. If we specialize at $u=0$, the natural quotient gives us a surjective map
$\C[S_{d,n}^{B}]=\C[C_{d}^n \rtimes W_n]\rightarrow \C[W_n \times C_2],$
defined by $\xi_i\mapsto \xi_i$ and $t\mapsto \det( t)^{\ast}$. Then, if we compose with this map we obtain 
\begin{equation}
\Infl:\Irr(W_n \times C_2)\rightarrow \Irr(S_{d,n}^{B}), \label{Inflacion 2}
\end{equation}
defined by 
\begin{equation*}
\Infl(\chi_{\lambda,\mu,\psi})(t \sigma)=\psi(\det(t)^{\ast}) \chi_{\lambda,\mu}(\sigma),
\end{equation*}
where $\psi$ is determined by $\psi(h)=\pm 1$.

\subsection{
Splitting  over \texorpdfstring{$\C(u)$}{C(u)}}\label{son split}
In this section, we 
prove the splitness over $\C(u)$ of the algebras $Y_{d,n}$ and $ Z_{d,n}$ 
for types $A,A^{\ast},$ and $B$. We 
also give a complete set of irreducible characters of $Y_{d,n}^{B}$. To do this, we 
apply a similar argument to the one given in \textup{\cite[§~3.9.4]{hausel2019arithmetic}}. Proving 
splitness over $\C(u)$ 
implies that the values of the irreducible characters arrive at $\C[u]$, by \textup{\cite[Proposition~7.3.8]{geck2000characters}}. This will 
simplify the computations in the next section.  

\begin{lemma}
The algebra $Z_{d,n}^{B}(u)$ is split. \label{split 1}
\label{Tipo B es split 1}
\end{lemma}
\begin{proof}
Recall from Subsection~\ref{rep normalizador}, that  $\Irr(C_{d}^{n}\rtimes W_n)=\Irr(S_{d,n}^{B})$ is parametrized by the set
\[\mathcal{R}_{d,n}=\{\Lambda=(\lambda_{1},\mu_{1},\lambda_{2},\mu_{2},\lambda_{3},\ldots,\lambda_{(d+2)/2}): \text{$\Lambda$ has total size $n$}  \}.\]
We set $n_{i}=\size{\lambda_{i}}+\size{\mu_{i}}$ for $i=1,2$, and $n_{i}=\size{\lambda_{i}}$ for $i>2$. We also denote by $\psi_{i}$ the element of $C_{d}^{\ast}$ associated with the block $i$. For $i>2$, we consider the representation \[\chi_{\lambda_i,\psi_i}^{Z_{d,n_{i}}^{A}}\] of $Z_{d,n_{i}}^{A}$ that comes from inflating the representation $\mathbb{S}^{\lambda_i}\otimes \psi_i$ of $\C[C_{d}\times \S_{n}]$ via the map of Equation~\eqref{cociente tipo A}.
For $i=1,2$, we consider the representation $\chi_{\lambda_i,\mu_i,\psi_i}$ of $Z_{d,n_{i}}^{B}$ that comes from inflating the representation $\chi_{\lambda_i,\mu_i}\otimes \psi_i$ of $\C \corch{C_{2}\times W_{n}}$ via the map of Equation~\eqref{cociente tipo B}. We also consider the $\C[u^{\pm 1}]$-algebra 
\[Y_{(n_i)}=Z_{d,n_{1}}^{B}\otimes Z_{d,n_{2}}^{B} \underset{i>2}{\bigotimes} Z_{d,n_{i}}^{A}.\]
Since the algebras $Z_{d,n_{i}}^{A}$ are sub-algebras of $Z_{d,n_{i}}^{B}$, we have that $Y_{(n_i)}$ is embedded in $Z_{d,n}^{B}$ and we can consider the representation 
\[\chi_{\Lambda}^{Y_{(n_i)}}=\chi_{\lambda_1,\mu_1,1}^{Z_{d,n_{1}}^{B}}\otimes \chi_{\lambda_2,\mu_2,-1}^{Z_{d,n_{2}}^{B}} \underset{i>2}{\bigotimes} \chi_{\lambda_i,\psi_i}^{Z_{d,n_{i}}^{A}},\]
where the tensor product is taken over $\C[u^{\pm 1}]$. We define $\chi_{\Lambda}^{Y_{B}}$ the induced character
\[\chi_{\Lambda}^{Y_{B}}=\Ind_{Y_{(n_i)}}^{Z_{d,n}^{B}}\chi_{\Lambda}^{Y_{(n_i)}}.\]
By \textup{\cite[Theorem~2.9]{benson1972degrees}}, the characters of $\Irr(H_{n})$ are defined over $\C(u)$, which implies that the characters $\chi_{\Lambda}^{Y_{B}}$ are also defined over $\C(u)$.
Then, to see that $\chi_{\Lambda}^{Y_{B}}(u)$ is split, it would be enough to prove that these are all the irreducibles.
We denote $N_i=C_{d}^{n_{i}}\rtimes W_{n_i}$ for $i=1,2$,   $N_i=C_{d}^{n_{i}}\rtimes \mathbb{S}_{n_i}$ for $i>2$, and $N=S_{d,n}^{B}$. Then, if we specialize at $u=0$ the representation $\chi_{\Lambda}^{Y_{B}}$, we obtain 
\[\Ind_{\prod N_i}^{N}\Infl(\chi_{\lambda_1,\mu_1,1})\otimes \Infl(\chi_{\lambda_2,\mu_2,-1}) \underset{i>2}{\bigotimes}\Infl(\chi_{\lambda_i,\psi_i}),\]
with the inflations of Equations~\eqref{Inflacion 2} and~\eqref{Inflacion 1}. This representation is the same as the representation $\chi_{\Lambda}\in \Irr(S_{d,n}^{B})$ of Equation~\eqref{chi b}. Therefore, when specializing at $u=0$ the representations $\chi_{\Lambda}^{Y_{B}}$, we obtain the representations of $\Irr(S_{d,n}^{B})$. We take $\mathbb{K}$ a finite Galois extension of $\C(u)$ such that $\mathbb{K}Z_{d,n}^{B}$ is split, and by taking the specialization at $u=0$ we obtain a bijection between $\Irr(\mathbb{K}Z_{d,n}^{B})$ and $\Irr(S_{d,n}^{B})$. Since the specializations 
\[d(\chi_{\Lambda}^{Y_{B}})=\chi_{\Lambda}^{S_{d,n}^{B}}\]
are all the representations of $\Irr(S_{d,n}^{B})$, due to what we saw in Subsection~\ref{caracteres de tipo B}, we have that the representations $\chi_{\Lambda}^{Y_{B}}:\Lambda \in \mathcal{R}_{d,n}$ are all the irreducibles of $Z_{d,n}^{B}$, which completes the proof.
\end{proof}

\begin{lemma} The algebra $Y_{d,n}^{B}(u)$ is split. \label{split 2}
\end{lemma}
\begin{proof}
Let $\C(u)\subset \mathbb{K}$ be a finite Galois extension such that $\mathbb{K}Y_{d,n}^{B}$ and $\mathbb{K} Z_{d,n}^{B}$ are split. As we saw in Diagram~\eqref{diagrama}, we have the following diagram of bijections 
\begin{center}
\begin{tikzcd}
 \Irr(\mathbb{K} Y_{d,n}^{B}) \arrow[swap]{dr}{d_{\phi_{1}}} & &\Irr(\mathbb{K}Z_{d,n}^{B}) \arrow{dl}{d_{\tau_{1}}}  \\
 & \Irr(N_{d,n}^{B}) & 
\end{tikzcd}
\end{center}
Furthermore, by Lemma~\ref{split 1}, we have that the representations $\chi_{\Lambda}^{Y_{B}}:\Lambda \in \mathcal{R}_{d,n}$ are the irreducibles of $\Irr(\mathbb{K} Z_{d,n}^{B})$. We consider $\Lambda=(\lambda_1,\mu_1,\lambda_2,\mu_2,\lambda_{3},\ldots ,\lambda_{(d+2)/2}) \in \mathcal{R}_{d,n}$, with $n_i=\size{\lambda_i}+\size{\mu_{i}}$ for $i=1,2$, and  $n_i=\size{\lambda_i}$ for $i>2$.

For $i>2$, we consider 
$\chi_{\lambda_i,\psi_i}^{Y_{d,n_i}^{A}}\in \Rep(Y_{d,n_{i}}^{A})$ that comes from inflating the reresentation $\chi_{\lambda_i,\psi_i}$ of $H_{d,n}^{A}$ via the natural quotient of  Equation~\eqref{cociente A}. For $i=1,2$, we consider 
$\chi_{\lambda_i,\mu_i,\psi_i}^{Y_{d,n_i}^{B}}\in \Rep(Y_{d,n_{i}}^{B})$ 
that comes from inflating the representation $\chi_{\lambda_i,\mu_i\psi_i}$  of $H_{2,n}^{B}$ via the natural quotient of Equation~\eqref{cociente B}. We now consider the $\C[u^{\pm 1}]$-algebra 
\[Y_{(n_i)}'=Y_{d,n_1}^{B}\otimes Y_{d,n_2}^{B} \underset{i>2}{\bigotimes}  Y_{d,n_i}^{A}.\]
Since the algebras $Y_{d,n_i}^{A}$ are subalgebras of $Y_{d,n_i}^{B}$, then $Y_{(n_i)}'$ is embedded in $Y_{d,n}^{B}$, and we can consider the representation 
\[\chi_{\Lambda}^{Y_{(n_i)}'}=\chi_{\lambda_1,\mu_1,1}^{Y_{d,n_1}^{B}}\otimes \chi_{\lambda_2,\mu_2,-1}^{Y_{d,n_2}^{B}} \underset{i>2}{\bigotimes}\chi_{\lambda_i,\psi_i}^{Y_{d,n_i}^{A}},\]
where the tensor product is taken over $\C[u^{\pm 1}]$. We also define \[{\chi_{\Lambda}^{Y_B}}\rq=\Ind_{Y_{(n_i)}'}^{Y_{d,n}^{B}}\chi_{\Lambda}^{Y_{(n_i)}'}.\]
The characters ${\chi_{\Lambda}^{Y_B}}\rq$ are defined over $\C(u)$ since, by \textup{\cite[Theorem~2.9]{benson1972degrees}}, both 
$\chi_{\lambda}^{H_n^{A}}$ and $\chi_{\lambda,\mu}^{H_n^{B}}$ are. Then, to prove that $Y_{d,n}^{B}$ is split it is enough to see that these are all the irreducibles. We denote \[N=N_{d,n}^{B},\quad  
  N_1=N_{d,n_{1}}^{B}, \quad N_2=N_{d,n_{2}}^{B},\quad  \text{and} \ \quad    N_i=N_{d,n_{i}}^{A} \ \text{for $i>2$}.\] Then, if we specialize at $u=1$ the representation ${\chi_{\Lambda}^{Y_{B}}}\rq$, we obtain 
\[\Ind_{\prod N_i}^{N}\Infl(\chi_{\lambda_1,\mu_1,1})\otimes \Infl(\chi_{\lambda_2,\mu_2,-1})\underset{i>2}{\bigotimes}\Infl(\chi_{\lambda_i,\psi_i}),\]
where the inflations are the ones from Equations~\eqref{inflacion Y_{d,n}_{A}} and~\eqref{inflacion Y_{d,n}_{B}}.
Recall that by specializing at $u=1$ the representation $\chi_{\Lambda}^{Y_{B}}$ we obtained 
\[\Ind_{\prod N_i}^{N}\Infl(\chi_{\lambda_1,\mu_1,1})\otimes \Infl(\chi_{\lambda_2,\mu_2,-1}) \underset{i>2}{\bigotimes} \Infl(\chi_{\lambda_i,\psi_i}),\]
where the inflations are the ones of Equations~\eqref{inflacion Y_{d,n}_{B}} and~\eqref{infl}. Since these inflations are the same, we obtain that
\[d_{\tau_1}\pare{\chi_{\Lambda}^{Y_{B}}}=d_{\phi_1}\pare{{\chi_{\Lambda}^{Y_{B}}}\rq}.\]
We saw in Lemma~\ref{split 1} that the $\chi_{\Lambda}^{Y_B}$ form the set $\Irr(Z_{d,n}^{B})$. Therefore, since $d_{\phi_1}$ and $d_{\tau_1}$ are bijections, we have that the representations ${\chi_{\Lambda}^{Y_B}}\rq$ compose the set $\Irr(Y_{d,n}^{B})$, which completes the proof.
\end{proof}

\begin{remark}\normalfont
Following the same argument of Lemmas~\ref{split 1} and~\ref{split 2}, 
we obtain that $Z_{d,n}^{A}(u)$ and $ Y_{d,n}^{A}(u)$ are split. Moreover, since $Z_{d,n}^{A^{\ast}}(u)$ and $Y_{d,n}^{A^{\ast}}(u)$ are subalgebras of $Z_{d,n}^{A}(u)$ and $ Y_{d,n}^{A}(u)$ respectively, they are split too.
\end{remark}

\section{Character values at \texorpdfstring{$T_{w_{0}}$}{T1} and \texorpdfstring{$T_{w_{0}}^{2}$}{T2}}\label{sec 5}
The goal of this section is to compute the values of the irreducible characters of $\H(G,U)$ at the elements $T_{w_{0}}$ and $T_{w_{0}}^{2}$, defined in Equation~\ref{tw0}. From Lemma~\ref{central}, we know that $T_{0}^{2}$ is central in $Y_{d,n}$. Then, by Schur's lemma, we know that $T_{0}^{2}$ acts by scalar multiplication in any irreducible representation. We will compute this scalar for each irreducible representation and also the character values at $T_{0}$. Then, by specializing at $q^{-1}$ we 
obtain the 
character values of $\H(G,U)$ at $T_{w_{0}}$ and $T_{w_{0}}^{2}$.

\subsection{Character values at \texorpdfstring{$T_{0}^{2}$}{T} }
In this subsection, we 
compute the scalar with which $T_{0}^{2}$ acts on each irreducible representation. We will
express these scalars with a formula depending on the characters of $S_{d,n}$, which we already described in Subsection~\ref{rep normalizador}. To do this, we 
follow a similar argument to the one given in  \textup{\cite[Theorem~9.2.2]{geck2000characters}} and \textup{\cite[Theorem~4.3.4]{hausel2019arithmetic}}.

\begin{remark}\normalfont \label{Autovalores 1}
Let $\mathbb{K}$ be a finite Galois extension of $\C(u)$ such that $\mathbb{K}Y_{d,n}$ is split, and let $(V,\pi)\in \Rep(\mathbb{K}Y_{d,n})$. 
The idempotent $e_i$ is the projector to the subspace $V_i=\Im(e_{i})$, in which it acts trivially, and there is a direct sum decomposition $V=V_i\oplus W_i$, with $W_i=\Ker(e_i)$. We call $T_{i}$ the endomorphism induced by $\xi_{i}$. Lemma~\ref{conmute} tells us that $\xi_i$ commutes with $e_i$, hence $T_i$ preserves this decomposition. We denote by $h_{i}(-1)$ the endomorphism induced by $h_{i}(-1)$. Applying Lemma~\ref{conmute}, we have that $h_{i}(-1)e_{i}=e_{i}$, which implies that $h_{i}(-1)$ acts trivially on $V_i$. Over $V_i$ the endomorphism $T_i$ satisfies the relation
\[T_{i}|_{V_{i}}^{2}=u+(1-u)T_i|_{V_{i}}.\]
And over $W_i$ it satisfies the relation 
\[T_i|_{W_i}^{2}=u h_{i}(-1)|_{W_i},\]
but since $h_{i}(-1)^{2}=1$, we have that 
\[T_{i}|_{W_{i}}^{4}=u^{2}.\]
Hence, the possible eigenvalues of $T_i$ are $1,-u,\pm  
 \sqrt{u},\pm i\sqrt{u}$.
\end{remark}
Now we 
do a similar argument for $Z_{d,n}.$
\begin{remark}\normalfont \label{Autovalores 2}
Let $\mathbb{K}$ be a finite Galois extension of $\C(u)$ such that $\mathbb{K}Z_{d,n}$ is split, and let $(W,\sigma)\in \Rep(\mathbb{K}Z_{d,n})$. 
We 
call $T_i$ the endomorphism induced by $\xi_i$, and  $h_{i}(-1)$ the endomorphism induced by $h_i(-1)$. Then, we have the relation
\[T_{i}^{2}=h_{i}(-1)u+(1-u).\]
Since $h_{i}(-1)^{2}=1$, the possible eigenvalues of $h_{i}(-1)$ are $\pm 1$, hence the possible eigenvalues of $T_{i}^{2}$ are $1$ and $1-2u$, and the possible eigenvalues of $T_i$ are $\pm 1,\pm \sqrt{1-2u}.$
\end{remark}
The following theorem gives us, for each irreducible representation of $Z_{d,n}$, the scalar with which $T_{0}'^{2}$ acts. We will then do the same for $Y_{d,n}$, but for that we 
need this theorem.
\begin{theorem}\label{teo g}
Consider $\mathbb{K}$ such that $\mathbb{K}Z_{d,n}$ is split, then for $\chi \in \Irr(\mathbb{K} Z_{d,n})$ the element $T_{0}'^{2}$ acts by scalar multiplication by 
\[z_{\chi}=(1-2u)^{g_{\chi}},\]
where 
\begin{equation}\label{g ch}
g_{\chi}=\sum_{a_{i}}\frac{  \chi_{0}(1)-\chi_{0}(t_{H_{a_i}}^{d/2})}{2  \chi_{0}(1)},
\end{equation}
where we are using Notation~\ref{not t}, $\chi_{0}\in \Irr(S_{d,n})$ is the specialization of $\chi$ at $0$ and the sum is taken over all the simple roots appearing in the decomposition of $w_{0}$, counted with multiplicity.\label{g_{chi}}
\end{theorem}
\begin{proof}
Let $(V,\pi)$ be an irreducible representation with character $\chi$. Using Remark~\ref{Autovalores 2}, we have that the possible eigenvalues of $\pi(\xi_{i})$ are $\pm 1,\pm \sqrt{1-2u}$. Let 
$n_{0},\, n_{1},\, n^{+},\, n^{-}$ be the multiplicities of $1,-1,\sqrt{1-2u},-\sqrt{1-2u}$ respectively. Observe that the sum of the multiplicities of the eigenvalues $\pm \sqrt{1-2u}$ is the multiplicity of $-1$ as an eigenvalue of $h_{i}(-1)$, hence we have
\[n^{+}+n^{-}=\frac{\chi(1)-\chi(h_{i}(-1))}{2},\] and by specializing at $0$ we obtain that 
\[n^{+}+n^{-}=\frac{\chi_{0}(1)-\chi_{0}(t_{H_{a_i}}^{d/2})}{2}.\]
On the other hand, we have 
\[\det(\pi(\xi_i))=1^{n_0}(-1)^{n_1}\pare{\sqrt{1-2u}}^{n^{+}}(-\sqrt{1-2u})^{n^{-}},\]
hence 
\[\det(\pi(\xi_i))^{2}=(1-2u)^{n^{+}+n^{-}}.\]
Since $T_{0}'^{2}$ is central, Schur´s lemma implies that it acts by scalar multiplication by some scalar $z_{\chi}\in \mathbb{K}$. Taking determinants we have
\[z_{\chi}^{\chi_0(1)}=\det(\pi(T_{0}'^{2}))=\textstyle \underset{a_{i}}{\prod}\det(\pi(\xi_i)^{2})=(1-2u)^{\sum_{i}(\chi_{0}(1)-\chi_{0}(t_{H_{a_i}}^{d/2}))/2}.\]
Then, taking $\chi_0(1)$-th roots, we have that there exists $\eta$, a $\chi_{0}(1)$-th root of unity, such that 
\[z_{\chi}=\eta  (1-2u)^{g_{\chi}}.\]
To prove that $\eta=1$, we specialize at $u=0$ and we have that $\tau_{0}(T_{0}^{2})=\widetilde{w_{0}}^{2}=1$, by \textup{\cite[Lemma~1.5.3]{geck2000characters}}. Then, we have
\[\eta \chi_{0}(1)=\tau_{0}(\chi(T_{0}^{2}))=\chi_{0}(\tau_{0}(T_{0}^{2}))=\chi_{0}(1),\]
which implies that $\eta=1$. This completes the proof.
\end{proof}

\begin{corollary}\label{sigma_{0}}
If $\chi \in \Irr(\mathbb{K}Z_{d,n})$, we have by Theorem~\ref{teo g} that $\chi(T_{0}'^{2})=\chi(1)(1-2u)^{g_{\chi}}$, then by specializing at $u=1$ we obtain
\[\chi_{1}(\sigma_{0}^{2})=\chi_{1}(\tau_{1}(T_{0}'^{2}))=\tau_{1}(\chi(T_{0}'^{2}))=\tau_{1}((1-2u)^{g_{\chi}}\chi(1))=(-1)^{g_{\chi}}\chi(1).\]
\end{corollary}

\begin{theorem}\label{f_{chi}}
Assume that $\mathbb{K}Y_{d,n}$ is split, then for $\chi\in \Irr(\mathbb{K}Y_{d,n})$ the element $T_{0}^{2}$ acts by scalar multiplication by 
\[y_{\chi}=u^{f_{\chi}}(-1)^{g_{\chi\rq}},\]
where 
\begin{equation}\label{ex f chi}f_{\chi}=\sum_{a_{i}} 1-\frac{\chi_{0}(e_{i}\xi_{i})}{\chi_{0}(1)},\end{equation}
$\chi'\in \Irr(\mathbb{K} Z_{d,n})$ is such that $\chi \rq_{1}=\chi_{1}$, and $\chi_{0}=d_{\tau_{0}}(\chi \rq)\in \Irr(S_{d,n})$ is the character associated with $\chi$ in Diagram~\ref{diagrama}. The sum is taken over all simple roots appearing in the decomposition of $w_{0}$, counted with multiplicity. In particular, if we specialize at $u=1/q$ we have that the element $T_{w_{0}}^{2}\in \mathcal{H}(G,U)$ acts by the scalar $q^{-f_{\chi}}(-1)^{g_{\chi\rq}}$.
\end{theorem}
\begin{proof}
Let $(V,\pi)$ be an irreducible representation with character $\chi$. By Remark~\ref{Autovalores 1}, the eigenvalues of $\pi(\xi_i)$ are $1,-u,\pm \sqrt{u},\pm i\sqrt{u}$, and let  $m_{0},m_{1},m^{+},m^{-},m_{2}^{+},m_{2}^{-}$ be their multiplicities. 
Then, we have 
\begin{align}
\det(\pi(\xi_{i}))=& 1^{m_{0}}(-u)^{m_{1}}(\sqrt{u})^{m^{+}}(-\sqrt{u})^{m^{-}}(i\sqrt{u})^{m_{2}^{+}}(-i\sqrt{u})^{m_{2}^{-}} \nonumber \\
=& (-1)^{m_{1}+m_{2}^{-}+m^{-}}\sqrt{u}^{2m_{1}+m_{2}^{+}+m_{2}^{-}+m^{+}+m^{-}}i^{m_{2}^{+}+m_{2}^{-}}.\nonumber \\
 \chi_{1}(1)=& \ m_{0}+m_{1}+m^{+}+m^{-}+m_{2}^{+}+m_{2}^{-}.\nonumber
\end{align}
Therefore, we have
\begin{align}
\det(\pi(\xi_{i}))^{2}=&u^{2m_{1}+m_{2}^{+}+m_{2}^{-}+m^{+}+m^{-}}(-1)^{m_{2}^{+}+m_{2}^{-}} \nonumber \\
=& u^{\chi_{1}(1)+m_{1}-m_{0}}(-1)^{m_{2}^{+}+m_{2}^{-}}\label{pi(T_{0}^2}.
\end{align}
We know from Lemma~\ref{conmute} that $\xi_{i}$ commutes with $e_{i}$, and that in $\Ker(e_{i})$ the operator $\xi_{i}e_{i}$ is null. While in $\Im(e_{i})$, the operator $\xi_{i}e_{i}$ acts in the same way as $\xi_{i}$, where it has eigenvalues $1$ and $-u$ with multiplicities $m_{0}$ and $m_{1}$, respectively. Hence $\chi(\xi_{i}e_{i})=m_{0}-m_{1}u$, and by specializing at $1$ we obtain
\begin{equation}
\chi_{1}(\xi_{i}e_{i})=m_{0}-m_{1}. \label{chi_{1}}
\end{equation}
Consider the element $\xi_{i}e_{i}\in Z_{d,n}$, and let $S_{i}$ be the operator that induces in $\chi'$. We have that in $\Ker(e_{i})$ this operator is null, and in $\Im(e_{i})$ it satisfies the relation
\[S_{i}^{2}=u h_{i}(-1)+1-u.\]
Using that $h_{i}(-1)e_{i}=e_{i}$, we have that $S_{i}^{2}=1$ in $\Im(e_{i})$, hence the possible eigenvalues of $S_{i}$ are $0,\pm 1$. In particular, we have $\chi'(\xi_{i}e_{i})=\chi_{1}(\xi_i e_i)=\chi_{0}(\xi_i e_i)$, and replacing it in Equation~\eqref{chi_{1}} we obtain 
\begin{equation*}
\chi_{0}(\xi_{i}e_i)=m_{0}-m_{1}.
\end{equation*}
Replacing this in Equation~\eqref{pi(T_{0}^2}, we obtain
\begin{equation*}
\det(\pi(\xi_{i}))^{2}=u^{\chi_{0}(1)-\chi_{0}(\xi_{i}e_{i})}(-1)^{m_{2}^{+}+m_{2}^{-}}.
\end{equation*}
Since $T_0^{2}$ is central, by Schur's lemma it acts by scalar multiplication by some $y_{\chi}\in \mathbb{K}$. Taking determinants we have
\[y_{\chi}^{\chi_{0}(1)}=\det(\pi(T_0^{2}))=\prod_{i}\det(  \pi(\xi_{i})^{2})=u^{\sum_{i} \chi_0(1)-\chi_0(\xi_{i}e_{i})}(-1)^{\sum_{i}  m_{2}^{+}+m_{2}^{-}}.\]
Then, taking $\chi_0(1)$-th roots, we have that there exists $\eta\in \C$  such that 
\[y_{\chi}=\eta   u^{f_{\chi}}.\]
To find $\eta$ we specialize at $u=1$, and we have $\phi_1(T_0^{2})=\sigma_0^{2}$. Therefore, we have
\[ \eta  \chi_1(1)=\phi_1(\chi(T_0^{2}))=\chi_1(\phi_1(T_0^{2}))=\chi_1(\sigma_{0}^{2})=(-1)^{g_{\chi'}}\chi_{1}(1),\]
where for the last equality we are using Corollary~\ref{sigma_{0}}. Therefore $\eta=(-1)^{g_{\chi'}}$, which completes the proof.
\end{proof}

\begin{remark}\normalfont \label{los dos son split}
Assume that $Y_{d,n}(u)$ and $Z_{d,n}(u)$ are split (as we already proved for types $A,A^{\ast}$ and $B$ in Subsection~\ref{son split}), and let $\chi \in \Irr(Y_{d,n}(u))$. By \textup{\cite[Proposition~7.3.8]{geck2000characters}}, we have that $\chi(\xi_{i})\in \C[u^{\pm 1}]$. This implies that $m_{2}^{+}=m_{2}^{-}$ and $m^{+}=m^{-}$, hence we have
\[\chi(\xi_{i})=m_{0}-m_{1}u=\chi(e_{i}\xi_{i}).\]
On the other hand, let $n_{0},\, n_{1},\, n^{+},\, n^{-}$ be the multiplicities of $ 1,-1 ,{\sqrt{1-2u}}, -{\sqrt{1-2u}}$ as eigenvalues of $\xi_{i}$ in $Z_{d,n}$ in the representation $\chi'$. Since $Z_{d,n}(u)$ is split we have $\chi'(\xi_{i})\in \C[u^{\pm 1}]$, which implies that $n^{+}=n^{-}$. Then, we have that   
\[\chi'(\xi_{i})=n_{0}-n_{1}=\chi_{0}(\xi_{i})=\chi_{1}(\xi_{i}).\]
Therefore, we have 
\[m_{0}-m_{1}=\phi_{1}(m_{0}-m_{1}u)=\phi_{1}(\chi(\xi_{i}))=\chi_{1}(\xi_{i})=\chi_{0}(\xi_{i}),\]
which allows us to exchange  $\chi_{0}(\xi_{i}e_{i})$ with $\chi_{0}(\xi_{i})$ in Equation~\eqref{ex f chi}.
\end{remark}

\begin{notation}\normalfont
The 
set $\Irr(S_{d,n})$ is parametrized by  the sets $\mathcal{Q}_{d,n},\mathcal{R}_{d,n}$ and $\mathcal{S}_{d,n}$, depending on the type. For $\Lambda$ an element in one of these sets, we will denote by $\chi_{\Lambda}\in \Irr(Y_{d,n})$ the associated representation. Looking at the diagrams~\eqref{diagrama} and~\eqref{diagrama 2}, we also define the irreducible representations
\[\chi_{\Lambda}'\, , \,
\chi_{\Lambda}^{0}\, , \,
\chi_{\Lambda}^{1},\ \text{ and } \,
\chi_{\Lambda}^{q}\]
of $Z_{d,n},S_{d,n},N_{d,n}$ and $\H(G,U)$ respectively, such that $d_{\tau_{0}}(\chi_{\Lambda}')=\chi_{\Lambda}^{0},\ d_{\tau_{1}}(\chi_{\Lambda}')=\chi_{\Lambda}^{1},\ d_{\phi_{1}}(\chi_{\Lambda})=\chi_{\Lambda}^{1}$, and $d_{\phi_{q}}(\chi_{\Lambda})=\chi_{\Lambda}^{q}$, where the last one is for $d=q-1$.
We also denote by $f_{\Lambda}$ and $g_{\Lambda}$ the scalars of Theorems~\ref{f_{chi}} and~\ref{g_{chi}} associated with $\chi_{\Lambda}$.
\end{notation}

\subsection{Computing \texorpdfstring{$g_{\Lambda}$}{g}}\label{g lambda}

To compute the value of the irreducible characters at $T_{0}^{2}$ it is necessary to compute $g_{\Lambda}$ and $f_{\Lambda}$ for each type. We 
start by computing $g_{\Lambda}$, to achieve this we 
use the description of $\Irr(S_{d,n})$ given in Subsection~\ref{rep normalizador}. 

\begin{notation}\label{notation psi}
In what follows, given the representation $\chi_{\Lambda}^{0}\in \Irr(S_{d,n})$, we 
always denote by $\psi \in C_{d}^{n}$ the tuple with which we construct $\chi_{\Lambda}^{0}$. We also denote  \[N(\psi)=C_{d}^{n}\rtimes {\fontshape{ui}\selectfont {\bf Stab} (\psi)}=\prod_{i}N_{i},\] and we denote by
\[\chi_{\Lambda}^{\psi}\in \Irr\pare{{\fontshape{ui}\selectfont {\bf Stab} (\psi)}}, \qquad 
\text{ and }\qquad 
\chi_{\Lambda}^{N\pare{\psi}}\in \Irr\pare{N\pare{\psi}}\]
the irreducible representations of ${\fontshape{ui}\selectfont {\bf Stab} (\psi)}$ and $N\pare{\psi}$ that we constructed in Subsection~\ref{rep normalizador}, before inducing. With this notation, we have that $\chi_{\Lambda}=\Ind_{N(\psi)}^{S_{d,n}}\chi_{\Lambda}^{N\pare{\psi}}$. Also observe that 
\[\chi_{\Lambda}^{N\pare{\psi}}(1)=\chi_{\Lambda}^{\psi}(1), \qquad   \text{and} \qquad  \corch{S_{d,n}:C_{d}^{n}\rtimes {\fontshape{ui}\selectfont {\bf Stab} (\psi)}}=\corch{W:{\fontshape{ui}\selectfont {\bf Stab} (\psi)}},\]
which implies that 
\begin{equation}
\chi_{\Lambda}(1)=\chi_{\Lambda}^{\psi}(1)\corch{W:{\fontshape{ui}\selectfont {\bf Stab} (\psi)}}. \label{N psi}
\end{equation}
From now on we will use the letter $s$ to denote a transposition in $\S_{n}$, and we will denote by $s_{i}$ the transposition $(i,i+1)$. We denote by $^{g}$ the conjugation by $g$. Also recall Notation~\ref{blocks ni}.
\end{notation}

\begin{notation}\label{alt}
For $\Lambda=(\lambda_{1},\ldots,\lambda_{k})$ a tuple of partitions, we denote
\[\Lambda!=\prod_{i=1}^{k}\size{\lambda_{i}}!,\quad \binom{n}{\Lambda}=\frac{n!}{\Lambda!},\quad \mathbb{S}^{\Lambda}(1)=\prod_{i=1}^{k}\mathbb{S}^{\lambda_{i}}(1),\quad \text{and} \quad  \size{\Lambda}^{\text{alt}}=\sum_{i=1}^{k}(-1)^{i-1}\size{\lambda_{i}},\]
where $n=\size{\Lambda}=\sum_{i=1}^{k}\size{\lambda_{i}}$.
\end{notation}

\begin{lemma} In type $A$, for $\Lambda=(\lambda_1,\ldots,\lambda_d)\in \mathcal{Q}_{d,n}$ we have $g_{\Lambda}=M_1M_2$, where $M_1$ and $M_2$ are the number of entries of $\psi$ that belong to $C_{d}^{2}$ and $C_{d}\backslash C_{d}^{2}$, respectively. \label{lema a}
\end{lemma}
\begin{proof}
Using the decomposition of $w_{0}$ from Subsection~\ref{basic prop}, we know that there are $\textstyle \binom{n}{2}$ simple roots in the decomposition of $w_{0}$, and that  $\textstyle{t_{H_{a_i}}^{d/2}=t_{i}^{d/2}t_{i+1}^{d/2}}$. Since these elements are conjugated in $C_{d}^{n} \rtimes \S_{n}$, we have
\begin{align}
    g_{\Lambda} = 
    \binom{n}{2}\frac{\chi_{\Lambda}^{0}(1)-\chi_{\Lambda}^{0}(t_{1}^{d/2}t_{2}^{d/2})}{2  \chi_{\Lambda}^{0}(1)}. \label{3p}
\end{align}


Since $t_{1}^{d/2}t_{2}^{d/2}$ belongs to $C_{d}^{n}$, it is invariant under conjugation by elements of $C_{d}^{n}$. Therefore, we have
\begin{equation}\chi_{\Lambda}^{0}(t_{1}^{d/2}t_{2}^{d/2})=\Ind_{\prod N_i}^{N}(t_{1}^{d/2}t_{2}^{d/2})=\frac{1}{\Lambda!}\underset{s \in \S_{n}}{\sum}\chi_{\Lambda}^{N(\psi)}(^{s}t_{1}^{d/2}t_{2}^{d/2}),\label{34}
\end{equation}
where $^{s}$ refers to the conjugation by $s$. When we conjugate $t_{1}^{d/2}t_{2}^{d/2}$ by an element of $\S_{n}$ we obtain an element of the form $t_{i}^{d/2}t_{j}^{d/2}.$ The value of $\chi_{\Lambda}^{N\pare{\psi}}$ at this element gives us 
 $\chi_{\Lambda}^{\psi}(1)$ if both or neither of the coordinates $i$ and $j$ correspond to a square in $C_{d}$, and otherwise it gives us $-\chi_{\Lambda}^{\psi}(1)$.
Replacing this in Equation~\eqref{34} leads to
\[\chi_{\Lambda}^{0}(t_{1}^{d/2}t_{2}^{d/2})=\frac{\chi_{\Lambda}^{\psi}(1)}{\Lambda!}(K_1-K_2),\]
where $K_1$ is the number of ways to conjugate $t_{1}^{d/2}t_{2}^{d/2}$ by an element of $\S_{n}$ and obtain an element with both or neither of the coordinates corresponding to elements of $C_{d}^{2}$, and $K_{2}$ is the number of ways to conjugate $t_{1}^{d/2}t_{2}^{d/2}$ by an element of $\S_{n}$ and obtain an element with exactly one of the coordinates corresponding to an element of $C_{d}^{2}$. It is straightforward to see that
\[K_{1}=2\left(\textstyle 
\binom{M_1}{2}+\binom{M_2}{2}\right)(n-2)!, \qquad \text{and} \qquad K_{2}=2M_{1}M_2(n-2)!.\]
On the other hand, we have 
\[\frac{n!  \chi_{\Lambda}^{\psi}(1)}{\Lambda!}=\chi_{\Lambda}^{0}(1).\]  
Therefore, we obtain
\[\frac{\chi_{\Lambda}^{0}(t_{1}^{d/2}t_{2}^{d/2})}{\chi_{\Lambda}^{0}(1)}=\frac{\binom{M_1}{2}+\binom{M_2}{2}-M_1 M_2}{\binom{n}{2}}.\]
Replacing $\textstyle \binom{n}{2}=\binom{M_1}{2}+\binom{M_2}{2}+M_1M_2$
in Equation~\eqref{3p}, we obtain the result.
\end{proof}

\begin{lemma}\label{restringir a Sl el g}
Let $\mu \in \Irr(S_{d,n}^{A^{\ast}})$ be a factor of the restriction of $\chi_{\Lambda}^{0}\in \Irr(S_{d,n}^{A})$, for $\Lambda \in \mathcal{Q}_{d,n} $. Then $g_{\mu}=g_{\Lambda}$.
\end{lemma}

\begin{proof}\normalfont 
From Example~\ref{Sl_{n} caracteres}, we know that $(\chi_{\Lambda}^{0})'$ has $d/o(\Lambda)$ factors in its decomposition. Then, we have \begin{align}\chi_{\Lambda}^{0} (1)=(\chi_{\Lambda}^{0})'(1)=d \mu(1)/o(\Lambda).\label{ecuacion restringir 1} \end{align}
On the other hand, all the conjugates of $t_{i}^{d/2}t_{i+1}^{d/2}$ are of the same form, and 
are all conjugated in  $S_{d,n}^{A^{\ast}}$. Therefore, we have
\begin{align}\chi_{\Lambda}^{0}(t_{1}^{d/2}t_{2}^{d/2})=(\chi_{\Lambda}^{0})'(t_{1}^{d/2}t_{2}^{d/2})=\sum_{i=1}^{d/o(\Lambda)}\mu^{g_{i}}(t_{1}^{d/2}t_{2}^{d/2})=\frac{d}{o(\Lambda)}\mu(t_{1}^{d/2}t_{2}^{d/2}),\label{ecuacion restringir 2} \end{align}
where the sum is taken over the conjugated characters of $\mu$. Then, replacing Equations~\eqref{ecuacion restringir 1} and~\eqref{ecuacion restringir 2} in Equation~\eqref{g ch}, we obtain the result.
\end{proof}

\begin{lemma}In type $B$, for $\Lambda=(\lambda_1,\mu_{1},\lambda_2,\mu_{2},\lambda_{3},\ldots, \lambda_{(d+2)/2})\in \mathcal{R}_{d,n}$ we have $g_{\Lambda}=2N_1N_2$, where $N_1$ and $N_2$ are the number of entries of $\psi$ that belong to $C_{d}^{2}$ and $C_{d}\backslash C_{d}^{2}$, respectively. \label{g en tipo B}
\end{lemma}
\begin{proof}
We have that $t_{H_{a_1}}^{d/2}=t_{2H_1}^{d/2}=1$, then $\chi_{\Lambda}^{0}(t_{H_{a_1}}^{d/2})=\chi_{\Lambda}^{0}(1)$. We also have that $t_{H_{a_i}}^{d/2}=t_{i}^{d/2}t_{i+1}^{d/2}$ for $i>1$, and there are $n^{2}-n$ of these simple roots in the decomposition of $w_{0}$. Since all these elements are conjugated in  $S_{d,n}^{B}$, we have that 
\begin{equation}g_{\Lambda}=(n^{2}-n)\frac{\chi_{\Lambda}^{0}(1)-\chi_{\Lambda}^{0}(t_{1}^{d/2}t_{2}^{d/2})}{2\chi_{\Lambda}^{0}(1)}\label{n^{2}-n}.\end{equation}
Since $C_{d}^{n}$ is abelian, the element $t_{1}^{d/2}t_{2}^{d/2}$ is stable under conjugation by elements of $C_{d}^{n}$. It is also stable under conjugation by elements of $\set{\pm 1}^{n}\subset W_{n}$, since the conjugation by these elements consists of inverting some elements of the tuple. Therefore, we have
\begin{equation}\label{ñ1}
\chi_{\Lambda}^{0}(t_{1}^{d/2}t_{2}^{d/2})=\frac{2^{n}}{\size{{\fontshape{ui}\selectfont {\bf Stab} (\psi)}}}\underset{s \in \S_{n}}{\sum}\chi_{\Lambda}^{N(\psi)}(^{s}t_{1}^{d/2}t_{2}^{d/2}),
\end{equation}
where $\chi_{\Lambda}^{N(\psi)}$ is extended as $0$ outside $N(\psi)$. 
 On the other hand, we have
\begin{equation}\label{ñ2}
\chi_{\Lambda}^{0}(1)=\chi_{\Lambda}^{\psi}(1)\frac{2^{n} n!}{\size{{\fontshape{ui}\selectfont {\bf Stab} (\psi)}}}.
\end{equation}
Replacing Equations~\eqref{ñ1} and~\eqref{ñ2} inside Equation~\eqref{n^{2}-n}, we obtain 
\[g_{\Lambda}=\frac{n^{2}-n}{2}\pare{1-\frac{1}{n!\ \chi_{\Lambda}^{\psi}(1)}\ \underset{s \in \S_{n}}{\sum}\chi_{\Lambda}^{N(\psi)}\pare{^{s}t_{1}^{d/2}t_{2}^{d/2}}}.\]
This is computed in the same way as in Lemma~\ref{lema a}, but in this case, the result is multiplied by two since we are multiplying by $n^{2}-n$ instead of $\textstyle \binom{n}{2}$. This completes the proof.
\end{proof}
\begin{lemma}
Let $\mu\in \Irr(S_{d,n}^{D})$ be a factor of the restriction of $\chi_{\Lambda}^{0}\in \Irr(S_{d,n}^{B})$. Then $g_{\mu}=g_{\Lambda}$.
\end{lemma}
\begin{proof}
We know that $t_{H_{a_1}}^{d/2}=t_{H_1+H_2}^{d/2}=t_1^{d/2}t_2^{d/2}$, while $t_{H_{a_i}}^{d/2}=t_{i}^{d/2}t_{i+1}^{d/2}$ for $i>1$. There are $n^{2}-n$ of these simple roots in the decomposition of $w_{0}$, and since all these elements are conjugated in $S_{d,n}^{D}$ we have that
\begin{equation}\label{p3}
g_{\mu}=(n^{2}-n)\frac{\mu(1)-\mu(t_{1}^{d/2}t_{2}^{d/2})}{2\mu(1)}.
\end{equation}
If $\mu$ is the restriction of $\chi_{\Lambda}^{0}$ we get the equality. Assume now that $\mu$ is one of the two factors of the restriction of $\chi_{\Lambda}^{0}$, and let $\eta$ be the other factor.
We know by  Subsection~\ref{caso ciclico} that the characters $\eta$ and $\mu$ are conjugated by some element  $t \in \set{\pm 1}^{n}$. Then, we have $\eta=\mu^{t}$ and 
\begin{equation}\label{p1}
\chi_{\Lambda}^{0}(t_{1}^{d/2}t_{2}^{d/2})=(\chi_{\Lambda}^{0})'(t_{1}^{d/2}t_{2}^{d/2})=\mu(t_{1}^{d/2}t_{2}^{d/2})+\mu^{t}(t_{1}^{d/2}t_{2}^{d/2})=2\mu(t_{1}^{d/2}t_{2}^{d/2}),
\end{equation}
where the last equality holds since $t$ acts by conjugation inverting the elements of the tuples in $C_{d}^{n}$, and therefore leaves $t_{1}^{d/2}t_{2}^{d/2}$ invariant. Doing an analogous argument, we have 
\begin{equation}\label{p2}
\chi_{\Lambda}^{0}(1)=2\mu(1).
\end{equation}
Replacing Equations~\eqref{p1} and~\eqref{p2} in Equation~\eqref{p3}, we obtain
\[g_{\mu}=(n^{2}-n)\frac{\chi_{\Lambda}^{0}(1)-\chi_{\Lambda}^{0}(t_{1}^{d/2}t_{2}^{d/2})}{2\chi_{\Lambda}^{0}(1)}=g_{\Lambda}.\]
\end{proof}

\begin{lemma}
\label{g en tipo C 2}
In type $C$, for $\Lambda=(\lambda_{1},\mu_{1},\lambda_{2},\mu_{2},\lambda_{3},\ldots ,\lambda_{(d+2)/2})\in \mathcal{R}_{d,n}$ we have that $g_{\Lambda}=2N_{1}N_{2}+N_{2}$, where $N_1$ and $N_2$ are the number of entries of $\psi$ that belong to $C_{d}^{2}$ and $C_{d}\backslash C_{d}^{2}$, respectively.
\end{lemma}
\begin{proof}
We have that $t_{H_{a_1}}^{d/2}=t_{H_1}^{d/2}=t_{1}^{d/2}$. For $i>1$ we have that  $t_{H_{a_i}}^{d/2}=t_{i}^{d/2}t_{i+1}^{d/2}$, and all these elements are conjugated in $S_{d,n}^{C}$. There are respectively $n$ and $n^{2}-n$ simple roots of each type in the decomposition of $w_{0}$, therefore we have 
\begin{equation}
g_{\Lambda}=(n^{2}-n) \frac{\chi_{\Lambda}^{0}(1)-\chi_{\Lambda}^{0}(t_{1}^{d/2}t_{2}^{d/2})}{2\chi_{\Lambda}^{0}(1)}+ n \frac{\chi_{\Lambda}^{0}(1)-\chi_{\Lambda}^{0}(t_{1}^{d/2})}{2\chi_{\Lambda}^{0}(1)}. \label{g en tipo C}
\end{equation}
The first term on the right-hand side of Equation~\eqref{g en tipo C} can be computed in the same way as in Lemma~\ref{g en tipo B}, and is equal to $2N_{1}N_2$.
Now we will compute $\chi_{\Lambda}^{0}(t_{1}^{d/2})$. The element $t_{1}^{d/2}$ is stable under conjugation by elements of $C_{d}^{n}$ and 
by elements of $\{\pm 1\}^{n}\subset W_{n}$, therefore we have
\begin{equation*}
\chi_{\Lambda}^{0}(t_{1}^{d/2})=\Ind_{\prod N_{i}}^{N}\chi_{\Lambda}^{N(\psi)}(t_{1}^{d/2})=\frac{2^{n}}{\size{{\fontshape{ui}\selectfont {\bf Stab} (\psi)}}}\sum_{s\in \S_{n}}\chi_{\Lambda}^{N(\psi)}(^{s}t_{1}^{d/2}),
\end{equation*}
where we extend $\chi_{\Lambda}^{N(\psi)}$ as $0$ outside $N(\psi)$. When conjugating $t_{1}^{d/2}$, we obtain an element of the form $t_{j}^{d/2}$, and the value of $\chi_{\Lambda}^{N(\psi)}$ at this element is equal to $\chi_{\Lambda}^{\psi}(1)$ if $j$ corresponds to a coordinate in $C_{d}^{2}$  and is equal to $-\chi_{\Lambda}^{\psi}(1)$ otherwise. The number of permutations that fall into the first and second case are $(n-1)!N_{1}$, and $(n-1)!N_{2}$, respectively. Hence, we have
\begin{equation}\label{f3}
\chi_{\Lambda}^{0}(t_{1}^{d/2})=\frac{2^{n}\chi_{\Lambda}^{\psi}(1)}{\size{{\fontshape{ui}\selectfont {\bf Stab} (\psi)}}}(n-1)!(N_{1}-N_{2}).
\end{equation}
Replacing Equation~\eqref{f3} in Equation~\eqref{g en tipo C}, and using that 
\[\frac{2^{n}n! \chi_{\Lambda}^{\psi}(1)}{\size{{\fontshape{ui}\selectfont {\bf Stab} (\psi)}} }=\chi_{\Lambda}^{0}(1),\] 
we obtain that
\[n  \frac{\chi_{\Lambda}^{0}(1)-\chi_{\Lambda}^{0}(t_{1}^{d/2})}{2\chi_{\Lambda}^{0}(1)}=\frac{n-(N_{1}-N_{2})}{2}=\frac{(N_{1}+N_{2})-(N_{1}-N_{2})}{2}=N_{2}.
\]
Hence, we obtain the equality of the statement.
\end{proof}

\subsection{Computing \texorpdfstring{$f_{\Lambda}$}{f}}\label{f lambda}

Now we will compute the value of $f_{\Lambda}$ for each type, using the results of Subsection~\ref{rep normalizador}. We start by defining $k_{\Lambda}$ as
\begin{equation}\label{k lam}
k_{\Lambda}=\sum_{a_{i}}\frac{\chi_{\Lambda}^{0}(e_{i}\xi_{i})}{\chi_{\Lambda}^{0}(1)},
\end{equation}
where the sum is taken over all simple roots appearing in the decomposition of $w_{0}$, counted with multiplicity. Observe that $f_{\Lambda}=l(w_{0})-k_{\Lambda},$ 
so finding $f_{\Lambda}$ is equivalent to finding $k_{\Lambda}$ since we know the value of $l(w_{0})$. 

\begin{definition}\normalfont
Let $\lambda=(\lambda_{1},
\ldots, \lambda_{k})$ be a partition, we define \[n(\lambda)=\sum_{i=1}^{k}(i-1)\lambda_{i}.\]
If $\Lambda:X\rightarrow \mathcal{P}$, where $X$ is a finite set, we define 
\[n(\Lambda)=\sum_{x \in X}n(\Lambda(x)).\]
\end{definition} 
 The following lemma gives us the value of $k_{\Lambda}$ in type $A$.
\begin{lemma}\label{d=2}\textup{\cite[Lemma~4.3.14]{hausel2019arithmetic}}
Let $\Lambda \in \mathcal{Q}_{d,n}$. Then 
\[k_{\Lambda}=n(\Lambda \rq)-n(\Lambda),\]
where $\Lambda \rq$ is obtained by 
transposing each partition of $\Lambda$.  In particular, for $d=2$ and $\Lambda=\pare{\lambda,\mu}$, we have that $\chi_{\Lambda}^{0}$ is the representation $\chi_{\lambda,\mu}\in \Irr(W_n)$, so we have that 
\begin{equation*}\binom{n}{2}\frac{\chi^{}_{\lambda,\mu}(s)}{\chi^{}_{\lambda,\mu}(1)}=n(\Lambda \rq)-n(\Lambda),  \end{equation*} 
where $s\in \S_{n}$ denotes a transposition. 
\end{lemma}

Now we will give a procedure to compute the value of $\chi_{\Lambda}^{0}$ at an element of the Weyl group $W\subset S_{d,n}$, for any 
type. Let $\sigma \in W \subset S_{d,n}$, we denote by  $g=\xi \pi\in C_{d}^{n}\rtimes W$ a generic element $g\in S_{d,n}$, with $\xi \in C_{d}^{n}$ and $\pi \in W$. Then, when inducing, we have 
\begin{equation}\chi_{\Lambda}^{0}(\sigma)=\frac{1}{\size{C_{d}^{n}\rtimes {\fontshape{ui}\selectfont {\bf Stab} (\psi)}}}\underset{g \in 
S_{d,n}}{\sum} \chi_{\Lambda}^{N(\psi)}(^{g}\sigma ),\label{24}\end{equation}
where $\chi_{\Lambda}^{N(\psi)}$ is extended as $0$ outside $C_{d}^{n}\rtimes {\fontshape{ui}\selectfont {\bf Stab} (\psi)}$. We have that
\[^{g}\sigma=g\sigma g^{-1}=\xi(^{\pi}\sigma)\xi^{-1}=\xi ^{(^{\pi}\sigma)}(\xi^{-1})(^{\pi}\sigma).\]
Since $^{\pi}\sigma\in W$, we have $^{(^{\pi}\sigma)}(\xi^{-1})\in C_{d}^{n}$. When $^{\pi}\sigma\in {\fontshape{ui}\selectfont {\bf Stab} (\psi)}$, we have 
\[\psi(\xi ^{(^{\pi}\sigma)}(\xi^{-1}))=\psi(\xi)\psi^{^{\pi}\sigma}(\xi^{-1})=\psi(\xi)\psi(\xi^{-1})=1.\]
Then, we have 
$\chi_{\Lambda}^{N(\psi)}= \begin{cases}\chi_{\Lambda}^{\psi}(^{\pi}\sigma) \ \text{if $^{\pi}\sigma\in {\fontshape{ui}\selectfont {\bf Stab} (\psi)},$ } \\
\text{$0$ otherwise,}
\end{cases}$
and Equation~\eqref{24} gives us 
\begin{equation}
\chi_{\Lambda}^{0}(\sigma)=\frac{1}{\size{{\fontshape{ui}\selectfont {\bf Stab} (\psi)}}} \underset{\pi \in W}{\sum}\chi_{\Lambda}^{\psi}(^{\pi}\sigma)=\Ind_{{\fontshape{ui}\selectfont {\bf Stab} (\psi)}}^{W}\chi_{\Lambda}^{\psi}(\sigma).\label{Lambda en sigma}
\end{equation}

\begin{lemma}
Let $\mu \in \Irr(S_{d,n}^{A^{\ast}})$ be a factor of the restriction of $\chi_{\Lambda}^{0}$, for $\Lambda \in \mathcal{Q}_{d,n}$. Then $k_{\mu}=k_{\Lambda}$.

\end{lemma}

\begin{proof}
 We have that
\[\chi_{\Lambda}^{0}(1)=(\chi_{\Lambda}^{0})'(1)=\frac{d}{o(\Lambda)}\mu(1),\]
by Subsection~\ref{Sl_{n} caracteres}. By Subsection~\ref{caso ciclico}, we know that the irreducible factors appearing in the restriction of $\chi_{\Lambda}^{0}$ are of the form $\mu^{t}$ for some $t\in C_{d}^{n}$. Moreover, the conjugates $\{ts_{i}t^{-1}:t\in C_{d}^{n}\}$ of a simple transposition $s_{i}\in S_{d,n}^{A}$ are of the form $s_{i}t_{i}^{j}t_{i+1}^{-j}$, and all of them are conjugated in  $S_{d,n}^{A^{\ast}}$, hence we have 
\[\chi_{\Lambda}^{0}(s)=(\chi_{\Lambda}^{0})'(s)=\sum_{i=0}^{g_{i}}\mu^{g_{i}}(s)=\frac{d}{o(\Lambda)}\mu(s).\]
Replacing this in Equation~\eqref{k lam}, we obtain that $k_{\mu}=k_{\Lambda}$, as desired.
\end{proof}

\begin{lemma} Let $\Lambda=(\Lambda_1,\Lambda_2,\lambda_{3},\ldots,\lambda_{(d+2)/2})\in \mathcal{R}_{d,n}$, with $\Lambda_i=(\lambda_i,\mu_i)$ for $i=1,2$. Then  
\[(n^2-n)\frac{\chi_{\Lambda}^{0}(s)}{\chi_{\Lambda}^{0}(1)}=2n(\Lambda_1')-2n(\Lambda_1)+2n(\Lambda_2')-2n(\Lambda_2)+n(\lambda_{3}')-n(\lambda_{3})+\cdots +n(\lambda_{(d+2)/2}')-n(\lambda_{(d+2)/2}),\] \label{k en tipo B}
where $s$ denotes a transposition.
\end{lemma}
\begin{proof}
We define the groups $N_i$ as in Notation~\ref{notation psi}, and we denote by $n_{i}$ the block sizes, as in Notation~\ref{blocks ni}. Using Equation~\eqref{Lambda en sigma}, we have that 
\[\chi_{\Lambda}^{0}(s)=\frac{1}{\size{{\fontshape{ui}\selectfont {\bf Stab} (\psi)}}} \underset{\pi \in W_{n}}{\sum}\chi_{\Lambda}^{\psi}(^{\pi}s).\]
Since $s$ is a transposition, its conjugates are either a transposition or a transposition that has two $-1$. Using this, we have 
 \begin{equation*} \begin{array}{c}
 \chi_{\Lambda}^{0}(s)=\textstyle\frac{1}{\size{{\fontshape{ui}\selectfont {\bf Stab} (\psi)}}}(C_{1}\chi_{\Lambda_1}(s)\chi_{\Lambda_2}(1)\underset{i>2}{\prod}\S^{\lambda_i}(1)+C_{2}\chi_{\Lambda_1}(1)\chi_{\Lambda_2}(s)\underset{i>2}{\prod}\S^{\lambda_i}(1) \\
+ \textstyle C_{j}\underset{j>2}{\sum}\chi_{\Lambda_1}(1)\chi_{\Lambda_2}(1)\S^{\lambda_j}(s) \underset{i>2,i\neq j}{\prod}\S^{\lambda_i}(1)), \end{array} \end{equation*} 
where $C_j=\size{\set{\pi \in W_n:  ^{\pi}s \ \text{is in the block $j$} }}.$  Lets compute $C_j$ for $j>2$, where the block $j$ is $\mathbb{S}_{n_j}$. Let $g=\xi \sigma$, with $\xi\in \{\pm 1\}^{n}$ and $\sigma \in \mathbb{S}_n$, such that $gsg^{-1}$ belongs to the $j$-block. We have $gsg^{-1}=\xi\sigma s \sigma^{-1} \xi^{-1}$. The number of elements $\sigma\in \S_{n}$ such that $\sigma s \sigma^{-1}$ belongs to the $j$-block is $\textstyle \binom{n_j}{2}2(n-2)!.$
Then, for $\xi\sigma s \sigma^{-1} \xi^{-1}$ to be a transposition without $-1$, 
$\xi$ should have the same sign in the positions corresponding to this transposition, so there are $2^{n-1}$ possible $\xi$. Then, $\textstyle C_j=\binom{n_j}{2}2^{n}(n-2)!$.

Let us now compute $C_i$ for $i=1,2$, where the $i$-block is $W_{n_i}$. In this case, we have that $\xi\sigma s \sigma^{-1} \xi^{-1}\in W_{n_i}$ if and only if $\sigma s \sigma^{-1}$ is a transposition of this block, hence we have $C_i=\textstyle \binom{n_i}{2}2^{n+1}(n-2)!.$
Then, replacing the values of $C_{j}$ and using Equation~\eqref{N psi}, we have that 
\begin{equation*}
2\binom{n_1}{2} \frac{\chi_{\Lambda_1}(s)}{\chi_{\Lambda_1}(1)}+2\binom{n_2}{2} \frac{\chi_{\Lambda_2}(s)}{\chi_{\Lambda_2}(1)}+\underset{j>2}{\sum}\binom{n_j}{2}\frac{\S^{\lambda_j}(s)}{\S^{\lambda_j}(1)},
\end{equation*}
which gives us the desired equality by using Lemma~\ref{d=2}.
\end{proof}

\begin{lemma}
Let $\Lambda=(\Lambda_1,\Lambda_2,\lambda_3,\ldots, \lambda_{(d+2)/2})\in \mathcal{R}_{d,n}$, where $\Lambda_i=(\lambda_i,\mu_i)$ for $i=1,2$. Then, we have 
\[\frac{n \chi_{\Lambda}^{0}(t)}{\chi_{\Lambda}^{0}(1)}=\size{(\Lambda_{1},\Lambda_{2})}^{\text{alt}} =\size{\lambda_1}-\size{\mu_1}+\size{\lambda_2}-\size{\mu_2},\]

where $t=\Diag(-1,1,\ldots,1)\in W_{n}$. \label{en t}
\end{lemma}

\begin{proof}
We start by computing
\[m \chi_{\lambda,\mu}(t)/\chi_{\lambda,\mu}(1),\] for generic partitions $\lambda$ and $\mu$ with $k_1=\size{\lambda}$, $k_2=\size{\mu}$, and  $m=k_1+k_2$. Let $g=\sigma \pi \in \mathbb{S}_{m}\rtimes \{\pm 1\}^{m}$, then $gtg^{-1}=\sigma \pi t \pi^{-1}\sigma^{-1}=\sigma t \sigma^{-1}$, hence any conjugate of $t$ comes from permuting the coordinate that contains $-1$ and only depends on which element of $\mathbb{S}_{m}$ we conjugate. When inducing, we have
\[\chi_{\lambda,\mu}(t)=\Ind_{N_1 \times N_2}^{W_{m}}\chi_{\lambda,\mu}^{N(\psi)}(t)= \frac{1}{2^{m}k_1!k_2!}\left(\chi_{\lambda,\mu}^{N(\psi)}(t_1)C_1+\chi_{\lambda,\mu}^{N(\psi)}(t_2)C_2\right),\]
where $t_1$ and $t_2$ denote conjugates of $t$ that have the coordinate $-1$ in the first or second block respectively, and $C_1$ and $C_2$ denote the number of times that each of these possibilities occurs.
Let's compute $C_i$ for $i=1,2$. Since the element $\pi \in \{\pm 1\}^{m}$ does not influence whether the $-1$ belongs to the $i$-block or not, we have $2^{m}$ possibilities for this element. On the other hand, the permutation $\sigma$ has to move the $-1$ to any of the $k_i$ positions of the $i$-block, hence the number of ways of doing this is $(m-1)!k_i$. Therefore, we have 
$C_i=2^{m}(m-1)!\ k_i$. On the other hand, 
we have that 
\[\chi_{\lambda,\mu}^{N(\psi)}(t_i)=(-1)^{i-1}\S^{\lambda}(1)\S^{\mu}(1),\]  
for $i=1,2$.
We also have  
\[\chi_{\lambda,\mu}(1)=\Ind_ {N_{1}\times N_{2}}^{W_{m}}(1)=\binom{m}{k_{1}}\chi_{\lambda,\mu}^{N(\psi)}(1)=\binom{m}{k_{1}}\S^{\lambda}(1)\S^{\mu}(1).\]
Putting both things together, we have 
\begin{equation}\frac{m \chi_{\lambda,\mu}(t)}{\chi_{\lambda,\mu}(1)}=k_{1}-k_{2}. \label{k1-k2}\end{equation}
Using Equation~\eqref{Lambda en sigma}, we have that 
\[\chi_{\Lambda}^{0}(t)=\Ind_{{\fontshape{ui}\selectfont {\bf Stab} (\psi)}}^{W_n}\chi_{\Lambda}^{\psi}(t),\] hence 
\begin{equation*} 
 \chi_{\Lambda}^{0}(t)=\frac{1}{\size{{\fontshape{ui}\selectfont {\bf Stab} (\psi)}}}\pare{D_{1}\chi_{\Lambda_1}(t)\chi_{\Lambda_2}(1)\underset{i>2}{\prod}\S^{\lambda_i}(1)+D_{2}\chi_{\Lambda_1}(1)\chi_{\Lambda_2}(t)\underset{i>2}{\prod}\S^{\lambda_i}(1)},
 \end{equation*}
 where $D_i$ is the number of ways to conjugate $t$ so that the $-1$ lies in the $i$-block. 
 Then, using Equation~\eqref{N psi}, we have

 \begin{equation} \frac{n\chi_{\Lambda}^{0}(t)}{\chi_{\Lambda}^{0}(1)}=\frac{1}{2^{n}(n-1)!}\pare{\frac{\chi_{\Lambda_1}(t)}{\chi_{\Lambda_1}(1)}D_1+\frac{\chi_{\Lambda_2}(t)}{\chi_{\Lambda_2}(1)}D_2}.\label{D2}\end{equation}
It is straightforward to see that $D_{i}=(n-1)!\ n_i 2^{n}$. Then,  Equation~\eqref{D2} becomes
\[\frac{n\chi_{\Lambda}(t)}{\chi_{\Lambda}(1)}=\frac{\chi_{\Lambda_1}(t)}{\chi_{\Lambda_1}(1)}n_1+\frac{\chi_{\Lambda_2}(t)}{\chi_{\Lambda_2}(1)}n_2,\] 
that by Equation~\eqref{k1-k2} is equal to 
\begin{equation*}
\size{\lambda_1}-\size{\mu_1}+\size{\lambda_2}-\size{\mu_2}.
\end{equation*}
This completes the proof.
\end{proof}

\begin{lemma}\label{ecuacion para tipo B}
Let $\Lambda=(\Lambda_{1},\Lambda_{2},\lambda_{3},\ldots,\lambda_{(d+2)/2})\in \mathcal{R}_{d,n}$, where $\Lambda_{i}=(\lambda_{i},\mu_{i})$ for $i=1,2$. Then, for type $B$, we have that
\begin{align*}
\quad k_{\Lambda}=\, & 2n(\Lambda_1')-2n(\Lambda_1)+2n(\Lambda_2')-2n(\Lambda_2)
+\size{(\Lambda_{1},\Lambda_{2})}^{\text{alt}}
\\
&+n(\lambda_{3}')-n(\lambda_{3})+\cdots +n(\lambda_{(d+2)/2}')-n(\lambda_{(d+2)/2})
.
\end{align*}
\end{lemma}
\begin{proof}
We know by  Lemmas~\ref{Tipo B es split 1} and~\ref{split 2} that $Y_{d,n}^{B}(u)$ and $Z_{d,n}^{B}(u)$ are split. Then, by Remark~\ref{los dos son split}, we have that 
\[k_{\Lambda}=\sum_{a_{i}}\frac{\chi_{\Lambda}^{0}(\xi_{i})}{\chi_{\Lambda}^{0}(1)}.\]
In the decomposition of $w_{0}$ there are $n^2-n$ conjugates of the transposition $s$, and $n$ conjugates of $t=\Diag(-1,1,\ldots,1)\in W_{n}$. Hence,
\[k_{\Lambda}=\frac{(n^{2}-n)\chi_{\Lambda}^{0}(s)+n\chi_{\Lambda}^{0}(t)}{\chi_{\Lambda}^{0}(1)},
\]
which gives us the desired equality by using Lemmas~\ref{k en tipo B} and~\ref{en t}.
\end{proof}

\begin{lemma}
Let $\mu \in \Irr(S_{d,n}^{D})$ be a factor of the restriction of $\chi_{\Lambda}^{0}\in \Irr(S_{d,n}^{B})$, for $\Lambda=(\Lambda_{1},\Lambda_{2},\lambda_{3},\ldots,\lambda_{(d+2)/2})\in \mathcal{R}_{d,n}$, where $\Lambda_{i}=(\lambda_{i},\mu_{i})$ for $i=1,2$. Then, we have
\begin{equation*}\label{equt}k_{\mu}=2n(\Lambda_1')-2n(\Lambda_1)+2n(\Lambda_2')-2n(\Lambda_2)+n(\lambda_{3}')-n(\lambda_{3})+\cdots +n(\lambda_{(d+2)/2}')-n(\lambda_{(d+2)/2}).\end{equation*}
    
\end{lemma}

\begin{proof}

In type $D$, we have that all the simple symmetries appearing in the decomposition of $w_{0}$ are conjugated, and so are the elements $s_{i}e_{i}$. Since $l(w_{0})=n^{2}-n$, we have  
\[k_{\mu}=(n^2-n)\frac{\mu(s_{i}e_{i})}{\mu(1)}.\]
If $\mu$ is the restriction of $\chi_{\Lambda}^{0}$ it is clear that $k_{\mu}$ is equal to the expression of Equation~\eqref{equt}, by applying Lemma~\ref{k en tipo B} and Remark~\ref{los dos son split}. Assume that $({\chi_{\Lambda}^{0}})\rq$ has another irreducible factor, which we call $\eta$. By Subsection~\ref{caso ciclico}, we know that $\eta=\mu^{t}$, where $t=\Diag(-1,1,\ldots,1)$. Then, we have
\[\chi_{\Lambda}^{0}(s_{i}e_{i})=\eta(s_{i}e_{i})+\mu(s_{i}e_{i})=\mu(s_{i}e_{i})+\mu^{t}(s_{i}e_{i})=\mu(s_{i}e_{i})+\mu(ts_{i}e_{i}t).\]
Since  $ts_{i}e_{i}t$ and $s_{i}e_{i}$ are conjugated, we deduce that $\textstyle{\mu(s_{i}e_{i})=\frac{\chi_{\Lambda}^{0}(s_{i}e_{i})}{2}}.$
By an analogous argument we have that $\textstyle{\mu(1)=\frac{\chi_{\Lambda}^{0}(1)}{2}}$. Therefore, we have
\[k_{\mu}=(n^{2}-n)\frac{\mu(s_{i}e_{i})}{\mu(1)}=(n^{2}-n)\frac{\chi_{\Lambda}^{0}(s_{i}e_{i})}{\chi_{\Lambda}^{0}(1)}.\]
Then, applying Lemma~\ref{k en tipo B} and Remark~\ref{los dos son split}, we obtain the result.
\end{proof}

\begin{lemma}\label{ecuacion tipo c}
In type $C$, for $\Lambda=\pare{\Lambda_1,\Lambda_2,\lambda_{3},\ldots, \lambda_{(d+2)/2}}\in \mathcal{R}_{d,n}$,  where $\Lambda_{i}=(\lambda_{i},\mu_{i})$, we have
\[k_{\Lambda}=2n(\Lambda_{1}')-2n(\Lambda_{1})+2n(\Lambda_{2}')-2n(\Lambda_{2})+n(\lambda_{3}')-n(\lambda_{3})+\cdots+n(\lambda_{(d+2)/2}')-n(\lambda_{(d+2)/2})+\size{\lambda_{1}}-\size{\mu_{1}}.\]
\end{lemma}
\begin{proof}
We have 
\[\xi_{1}e_{1}=t \frac{1}{d}\sum_{j}t_{1}^{j}, \quad \text{ and } \quad 
\xi_{i}e_{i}=s_{i} \frac{1}{d}\underset{j}{\sum}t_{i}^{j}t_{i+1}^{-j},\]
where $t=\Diag(-1,1,\ldots,1)\in W_{n}$. It is straightforward to see that elements $s_{i}$ are conjugated in  $S_{d,n}^{C}$. From Equation~\eqref{k lam}, and from the decomposition of $w_{0}$ we have that
\begin{equation}
k_{\Lambda}=(n^{2}-n) \chi_{\Lambda}^{0}(s_{1}\frac{1}{d}\underset{j}{\sum}t_{1}^{j}t_{2}^{-j})/\chi_{\Lambda}^{0}(1)+n \chi_{\Lambda}^{0}(t\frac{1}{d}\underset{j}{\sum}t_{1}^{j})/{\chi_{\Lambda}^{0}(1)}. \label{segundo sumando}
\end{equation}
Since $Y_{d,n}^{B}(u)$ and $Z_{d,n}^{B}(u)$ are split, we have that the first term of Equation~\eqref{segundo sumando} is equal to $\textstyle{(n^{2}-n)\frac{\chi_{\Lambda}^{0}(s_{1})}{\chi_{\Lambda}^{0}(1)}}$, which we already computed in Lemma~\ref{k en tipo B}, and is equal to 
\[2n(\Lambda_{1}')-2n(\Lambda_{1})+2n(\Lambda_{2}')-2n(\Lambda_{2})+n(\lambda_{3}')-n(\lambda_{3})+\cdots+n(\lambda_{(d+2)/2}')-n(\lambda_{(d+2)/2}).\]
Let us compute the second term of Equation~\eqref{segundo sumando}. We have that 
\[\chi_{\Lambda}^{0}(te_{1})=\frac{1}{\size{C_{d}^{n}\rtimes {\fontshape{ui}\selectfont {\bf Stab} (\psi)} }} \underset{g \in C_{d}^{n}\rtimes W_{n}}{\sum}\chi_{\Lambda}^{N(\psi)}(^{g}te_{1}).\]
If $g=\xi \pi \in C_{d}^{n}\rtimes W_{n}$, using a similar argument as before, we have that
\begin{align*}
^{g}e_{1}t= ^{g}(\frac{1}{d}  \sum_{j}t_{1}^{j}\ t)=& \frac{1}{d} \sum_{j} \xi^{(^{\pi}t_{1}^{j}t)}(\xi^{-1})(^{\pi}t_{1}^{j}t) \\
=& \frac{1}{d} \sum_{j} \xi^{(^{\pi}t_{1}^{j}t)}(\xi^{-1})({t_{\pi(1)}^{j}} ^{\pi}t),
\end{align*}
and by evaluating $\chi_{\Lambda}^{N(\psi)}$ in this element we get $0$, unless that $^{\pi}t\in {\fontshape{ui}\selectfont {\bf Stab} (\psi)}$. In that case, we have
\begin{align*}
\chi_{\Lambda}^{N(\psi)}(^{g}e_{1}t)&=\frac{1}{d} \sum_{j}\chi_{\Lambda}^{N(\psi)}(\xi^{(^{\pi}t_{1}^{j}t)}(\xi^{-1})({t_{\pi(1)}^{j}}  ^{\pi}t))\\
& =\frac{1}{d} \sum_{j}\psi(\xi^{(^{\pi}t_{1}^{j}t)}(\xi^{-1})t_{\pi(1)}^{j})  \chi_{\Lambda}^{\psi}(^{\pi}t)\\
& =\frac{1}{d} \sum_{j}\psi(\xi)\psi^{^{\pi}t_{1}^{j}t}(\xi^{-1})\psi(t_{\pi(1)})^{j}  \chi_{\Lambda}^{\psi}(^{\pi}t)\\
&= \frac{1}{d} \sum_{j}\psi(\xi)\psi(\xi^{-1})\psi(t_{\pi(1)})^{j}  \chi_{\Lambda}^{\psi}(^{\pi}t)\\
&=\frac{1}{d}(\sum_{j}\psi(t_{\pi(1)})^{j})
 \chi_{\Lambda}^{\psi}(^{\pi}t).
\end{align*}
Therefore, we have
\begin{equation*}
\chi_{\Lambda}^{0}(te_{1})=\frac{1}{\size{ {\fontshape{ui}\selectfont {\bf Stab} (\psi)} }}\underset{\pi \in W_{n}}{\sum} \chi_{\Lambda}^{N(\psi)}(^{\pi}te_1).
\end{equation*}
It is straightforward to see that $te_{1}$ remains invariant under conjugation by elements of $\set{\pm 1}^{n}\subset W_{n}$, since $be_{1}b^{-1}=e_{1}$ for every $b\in \set{\pm 1}^{n}$.
If $\sigma\in \S_{n}$ we have $\sigma e_{1}t\sigma^{-1}=(\sigma e_{1}\sigma^{-1})(\sigma t\sigma^{-1})=z_{k}b_{k}$ for some $k$, where we denote by $b_{k}$ the element of $\{\pm 1\}^{n}$ that has $-1$ in the position $k$ and $1$ in the remaining positions, and we denote by $z_{k}$ the idempotent $1/d \textstyle\sum_{j}t_{k}^{j}$. These are the conjugates of $e_{1}t$. Then, following the same argument of Lemma~\ref{k en tipo B}, we have
\begin{equation*}
\chi_{\Lambda}^{0}(e_{1}t)=\textstyle{\frac{1}{\size{ {\fontshape{ui}\selectfont {\bf Stab} (\psi)} }}\pare{D_{1}\chi_{\Lambda_1}(t)\psi_{1}(z)\chi_{\Lambda_2}(1)\underset{i>2}{\prod}\S^{\lambda_i}(1)+D_2\chi_{\Lambda_1}(1)\chi_{\Lambda_2}(t)\psi_{2}(z)\underset{i>2}{\prod}\S^{\lambda_i}(1)}},
\end{equation*}
where $D_{i}$ for $i=1,2$ is the number of ways to conjugate $t$ in such a way that $-1$ belongs to the $i$-block. 
With $z$ we refer to an idempotent of that block. Observe that $\psi_{1}(z)=1$, and that $\psi_{2}$ is the character with exponent $d/2$, so we have 
\[\psi_{2}(z)=\psi_{2}(\frac{1}{d}\underset{j}{\sum} t_{1}^{j})=\frac{1}{d} \underset{j}{\sum} t_{g}^{\ j d/2}=0,\]
where $t_{g}$ is a primitive $d$-th root. Therefore, we obtain
\begin{equation}
\chi_{\Lambda}^{0}(e_{1}t)=\frac{1}{\size{ {\fontshape{ui}\selectfont {\bf Stab} (\psi)} }} D_{1} \chi_{\Lambda_1}(t) \chi_{\Lambda_2}(1) \prod_{i>2} \S^{\lambda_i}(1).  \label{e_{1}b_{1}}
\end{equation}
The values $D_{1}$ and $\chi_{\Lambda_{1}}(t)$ are computed as in Lemma~\ref{en t}, and they are
\begin{equation}
D_{1}=(n-1)! (\size{\lambda_{1}}+\size{\mu_{1}}) 2^{n}, \qquad \text{and} \qquad \frac{\chi_{\Lambda_{1}}(t)}{\chi_{\Lambda_{1}}(1)}=\frac{\size{\lambda_{1}}-\size{\mu_{1}}}{\size{\lambda_{1}}+\size{\mu_{1}}}. \label{D1'}
\end{equation}
 We also have that 
\begin{equation}
\frac{2^{n}n!}{\size{{\fontshape{ui}\selectfont {\bf Stab} (\psi)}}} \chi_{\Lambda}^{\psi}(1)=\chi_{\Lambda}^{0}(1). \label{chi(1)}
\end{equation}
Finally, replacing Equations~\eqref{chi(1)} and~\eqref{D1'} into  Equation~\eqref{e_{1}b_{1}}, we obtain that
\[\frac{n\chi_{\Lambda}^{0}(e_{1}t)}{\chi_{\Lambda}^{0}(1)}=\size{\lambda_{1}}-\size{\mu_{1}}.\]
This completes the proof.
\end{proof}

\begin{corollary}
The element $T_{0}^{2}$ acts by scalar multiplication on $\chi_{\Lambda}$ by scalar
\[z_{\Lambda}=u^{l(w_{0})-k_{\Lambda}}(-1)^{g_{\Lambda}}.\]
Specializing at $u=q^{-1}$ for $d=q-1$, we obtain that $T_{w_{0}}^{2}\in \mathcal{H}(G,U)$ acts by scalar multiplication in $\chi_{\Lambda}^{q}$ by scalar
\[q^{k_{\Lambda}-l(w_{0})}(-1)^{g_{\Lambda}}.\]
We computed the values of $k_{\Lambda}$ and $g_{\Lambda}$ in Subsections~\ref{f lambda} and~\ref{g lambda}, and we know the value of $l(w_{0})$ from \textup{\cite[§~1.5]{geck2000characters}}. Then, replacing these values, 
we find  $z_{\Lambda}$ for each type. In 
Table~\ref{table:1}, using Notations~\ref{notation psi} and~\ref{alt}, we show the values of $k_{\Lambda}$ and $g_{\Lambda}$ for each type. 

\end{corollary}

\subsection{Character values at \texorpdfstring{$T_{w_{0}}$}{t0}}\label{T_{0}}

The goal of this subsection is to find the values of the characters $\chi_{\Lambda}^{q}$ at the element $T_{w_{0}}$. We already computed the scalar with which $T_{w_{0}}^{2}$ acts. Using this, the computation of $\chi_{\Lambda}^{q}$ at $T_{w_{0}}$ is reduced to compute the value of the character $\chi_{\Lambda}^{0}$ at $\widetilde{w_{0}}$, the maximal length element of $S_{d,n}$.

We take $\mathbb{K}$ a finite Galois extension of $\C(u)$ such that $\mathbb{K}Y_{d,n}$ and $\mathbb{K}Z_{d,n}$ are split. We denote by $\phi_{1}$ and $\phi_{q}$, respectively, the specializations at $1$ and $1/q$, defined on the integral closure of $\C[u^{\pm 1}]$, which determine the bijections $d_{\phi_{1}}$ and $d_{\phi_{q}}$ of Diagram~\ref{diagrama}. We also denote by $\tau_{1}$ and $\tau_{0}$, respectively, the specializations at $1$ and $0$, which determine the bijections $d_{\tau_{1}}$ and $d_{\tau_{0}}$ of the same diagram. We fix $v,w\in \mathbb{K}$ such that $v^{2}=u$ and $w^{2}=1-2u$. Note that  $\mathbb{K}$ can be taken large enough such that $v$ and $w$ exist. Also observe that $v$ and $w$ belong to the integral closure of $\C [u^{\pm 1}]$. Then, using \textup{\cite[§~9.2.7]{geck2000characters}}, we can normalize the specializations in such a way that 
\[\phi_{q}(v)=\sqrt{q^{-1}},\quad \phi_{1}(v)=1,\quad \tau_{1}(w)=-i,\quad \text{and} \quad \tau_{0}(w)=1.\]
From now on the chosen specializations will be of this form.

\begin{lemma}\label{T0 1}
Let $\chi_{\Lambda}\in \Irr(\mathbb{K}Y_{d,n})$. Keeping the above notation, we have that \[\chi_{\Lambda}(T_{0})=\chi_{\Lambda}^{1}(\sigma_0)v^{f_{\Lambda}}=\chi_{\Lambda}^{q}(T_{w_0})v^{f_{\Lambda}}\sqrt{q}^{f_{\Lambda}}.\]
\end{lemma}
\begin{proof}
Let $\sigma_1,
\ldots,\sigma_d$ be the eigenvalues of $T_{0}$ taken with the multiplicity on its characteristic polynomial. The eigenvalues of $T_{0}^{2}$ are $\sigma_{1}^{2},
\ldots, \sigma_{d}^{2}$, with the same multiplicities, then by Theorem~\ref{f_{chi}} we have
\[\sigma_{j}^{2}=u^{f_{\Lambda}}(-1)^{g_{\Lambda}}=(v^{f_{\Lambda}}i^{g_{\Lambda}})^{2}=a^{2} \quad \text{for all $j$}, \ \text{where} \ a=v^{f_{\Lambda}}i^{g_{\Lambda}}.\]
Hence, for each $i$ there exists $\eta_i=\pm 1$ such that $\sigma_{i}=\eta_{i}a$. Therefore, we have $\chi_{\Lambda}(T_{0})=\textstyle \sum_{i}\eta_i a$. Taking the specialization $\phi_{1}$, we obtain
\[\chi_{\Lambda}^{1}(\sigma_0)=\phi_{1}(\chi_{\Lambda}(T_{0}))=\textstyle \sum_{i}\eta_i\phi_{1}(a)=\sum_{i} \eta_i i^{g_{\Lambda}},\]
which gives us the first equality. When specializing at $\phi_{q}$, we obtain
\[\chi_{\Lambda}^{q}(T_{w_0})=\phi_{q}(\chi(T_{0}))=\textstyle \sum_{i}\eta_i \phi_{q}(a)=\sum_{i}\eta_i i^{g_{\Lambda}} \sqrt{q}^{-f_{\Lambda}},\]
which gives us the second equality.
\end{proof}

\begin{lemma}\label{T0 2}Let $\chi_{\Lambda}'\in \Irr(\mathbb{K}Z_{d,n})$. Then, with the above notation, we have 
\[\chi_{\Lambda}'(T_{0}')=\chi_{\Lambda}^{0}(\widetilde{w_{0}}) w^{g_{\Lambda}}=\chi_{\Lambda}^{1}(\sigma_{0})w^{g_{\Lambda}}i^{g_{\Lambda}}.\]  
\end{lemma}
\begin{proof}
Let $\gamma_{1},
\ldots,\gamma_{k}$ be the eigenvalues of $T_{0}'$. Then, $\gamma_{1}^{2},
\ldots,\gamma_{k}^{2}$ are the eigenvalues of $T_{0}'^{2}$, and applying Theorem~\ref{g_{chi}} we have 
\[\gamma_{i}^{2}=(1-2u)^{g_{\Lambda}}=(w^{g_{\Lambda}})^{2}=b^{2}\quad \text{for all $i$},\ \text{where}\ b=w^{g_{\Lambda}}.\]
Hence, for each $i$ there exists $\kappa_{i}=\pm 1$ such that $\gamma_{i}=\kappa_{i}b$. Then $\chi_{\Lambda}'(T_{0}')=\textstyle \sum_{i}\kappa_{i}b$. Taking the specialization $\tau_{0}$, we obtain
\[\chi_{\Lambda}^{0}(\widetilde{w_{0}})=\tau_{0}(\chi_{\Lambda}'(T_0'))=\textstyle \underset{i}{\sum}\kappa_{i}\tau_{0}(b)=\underset{i}{\sum}\kappa_{i},\]
which gives us the first equality. If we take the specialization $\tau_{1}$, we obtain
\[\chi_{\Lambda}^{1}(\sigma_{0})=\tau_{1}(\chi_{\Lambda}'(T_0'))=\textstyle \underset{i}{\sum}\kappa_{i}\phi_{1}(b)=\textstyle \underset{i}{\sum}\kappa_{i}(-i)^{g_{\Lambda}},\]
which gives us the second equality.
\end{proof}

\noindent{}
We will now compute $\chi_{\Lambda}^{0}(\widetilde{w_{0}})$ for each type.

\medskip
{\bf Type $A$:} In this case, the maximal length element is a product of $\textstyle \floor{ \frac{n}{2}}$ disjoint transpositions. Let $\Lambda=(\lambda_1,\ldots,\lambda_d)\in \mathcal{Q}_{d,n}$, and let $n_{i}=\size{\lambda_{i}}$. We want to compute $\chi_{\Lambda}^{0}$ at an element that is a product of $\textstyle \floor{\frac{n}{2}}$ disjoint transpositions. Since $\widetilde{w_{0}} \in \mathbb{S}_n$, using the same argument as in Lemma~\ref{k en tipo B}, we obtain that 
\begin{equation}\chi_{\Lambda}^{0}(\widetilde{w_{0}})=\Ind_{{\fontshape{ui}\selectfont {\bf Stab} (\psi)}}^{\mathbb{S}_n}\chi_{\Lambda}^{\psi}(w_0).\label{w0 tipo A}\end{equation}

\medskip
{\bf Case~1.} Assume that $n$ is even. 
We divide it into subcases.

\medskip 
{\bf Case~1.1.} Assume that not all the sizes of the blocks of $\psi$ are even. Since any conjugate of $w_{0}$ is the product of $n/2$ disjoint transpositions, it is clear that no conjugate of $w_{0}$ belongs to ${\fontshape{ui}\selectfont {\bf Stab} (\psi)}$. Therefore, we have that $\chi_{\Lambda}^{0}(\widetilde{w_{0}})=0$.

\medskip
{\bf Case~1.2.} Assume that the sizes of all the blocks of $\psi$ are even. Each time that a conjugate of $w_{0}$ belongs to ${\fontshape{ui}\selectfont {\bf Stab} (\psi)}$, we have a product of $n_{j}/2$ disjoint transpositions in the $j$-block, hence $\chi_{\Lambda}^{\psi}$ has the same value at all the conjugates of $w_{0}$. Then, using Equation~\eqref{w0 tipo A}, we have
\begin{equation}
\chi_{\Lambda}^{0}(\widetilde{w_{0}})=\frac{C}{\Lambda!} \S^{\lambda_1}(\sigma_1)
\cdots \S^{\lambda_d}(\sigma_{d}),\label{w0 en tipo A 2} 
\end{equation}
where 
$C$ is the number of ways to conjugate $w_0$ and obtain an element in ${\fontshape{ui}\selectfont {\bf Stab} (\psi)}$, and each of the $\sigma_i$ denotes a product of $n_{i}/2$ disjoint transpositions. It is straightforward to check that 
\[C=\frac{\frac{n}{2}! \Lambda!}{\frac{n_{1}}{2}!
\cdots \frac{n_{d}}{2}!}.\]
Replacing this in Equation~\eqref{w0 en tipo A 2}, we obtain
\[
\chi_{\Lambda}^{0}(\widetilde{w_{0}})=\frac{\frac{n}{2}!}{\frac{\size{\lambda_1}}{2}!
\cdots  \frac{\size{\lambda_d }}{2}!}\S^{\lambda_1}(\sigma_1)
\cdots  \S^{\lambda_d}(\sigma_d).
\]

\medskip
{\bf Case~2.} Assume that $n$ is odd. In this case, $\widetilde{w_{0}}$ is the product of $(n-1)/2$ disjoint transpositions. 

\medskip
{\bf Case~2.1.} Assume that there is not exactly one block of $\psi$ of odd size. Since any conjugate of $w_{0}$ is the product of $(n-1)/2$ disjoint transpositions, it is clear that no conjugate of $w_{0}$ belongs to ${\fontshape{ui}\selectfont {\bf Stab} (\psi)}$. Therefore, we have $\chi_{\Lambda}^{0}(\widetilde{w_{0}})=0$.

\medskip
{\bf Case~2.2.} Assume that there is exactly one block of $\psi$ of odd size, and assume without loss of generality that it is the last block, corresponding to $\lambda_d$. Then, we have
\begin{equation*}
\chi_{\Lambda}^{0}(\widetilde{w_{0}})=\frac{C\rq}{\Lambda!}\S^{\lambda_1}(\sigma_1)
\cdots \S^{\lambda_d}(\sigma_d),
\end{equation*}
where $C\rq$ is the number of ways to conjugate $w_0$ and obtain an element in ${\fontshape{ui}\selectfont {\bf Stab} (\psi)}$, and each of the $\sigma_i$ denote a product of $\textstyle \floor{\frac{n_i}{2}} $ disjoint transpositions. It is straightforward to compute $C\rq$, and we have that
\begin{equation*}\chi_{\Lambda}^{0}(\widetilde{w_{0}})=\frac{((n-1)/2)!}{(\size{\lambda_{1}}/2)!
\cdots (\size{\lambda_{d-1}}/2)! ((\size{\lambda_{d}}-1)/2)!}\S^{\lambda_1}(\sigma_1)
\cdots \S^{\lambda_d}(\sigma_d).\end{equation*}

In Table~\ref{table:2}, we show the value of $\chi_{\Lambda}^{0}(\widetilde{w_{0}})$ for type $A$. 

\medskip
{\bf Type $A^{\ast}$:} Let $\mu\in \Irr(S_{d,n}^{A^{\ast}})$ be a factor in the restriction of $\chi_{\Lambda}^{0}\in \Irr(S_{d,n}^{A})$, for $\Lambda \in \mathcal{Q}_{d,n}$. Let us compute the value $\mu(\widetilde{w_{0}})$. We separate into cases.

\medskip
{\bf Case~1.} Assume that $n$ is odd. Notice that for $t=(t_{1},
\ldots,t_{n})\in C_{d}^{n}$ we have 
\[t\widetilde{w_{0}}t^{-1}=\widetilde{w_{0}}t_{w_{0}}t^{-1}
= \widetilde{w_{0}} (t_{1}^{-1}t_{n},t_{2}^{-1}t_{n-1}, \ldots,
t_{1}t_{n}^{-1}).
\]
Observe that if we only change the central coordinate of $t$, the element $t\widetilde{w_{0}}t^{-1}$ would not change. Using this, we can see that all the elements $\{t\widetilde{w_{0}}t^{-1}:t\in C_{d}^{n}\}$ are conjugated in $S_{d,n}^{A^{\ast}}$. Therefore, all the conjugates of $\widetilde{w_{0}}$ in $S_{d,n}^{A}$ are also conjugated in $S_{d,n}^{A^{\ast}}$. Then, we have
\[\chi_{\Lambda}^{0}(\widetilde{w_{0}})=\sum_{i=1}^{d/o(\Lambda)}\mu^{g_{i}}(\widetilde{w_{0}})=\frac{d}{o(\Lambda)}\mu(\widetilde{w_{0}}),\] 
where the sum is taken over all conjugates of $\mu$. Therefore, we have  
\[u(\widetilde{w_{0}})=\frac{o(\Lambda)}{d}\chi_{\Lambda}^{0}(\widetilde{w_{0}}).\]

\medskip
{\bf Case~2.} Assume that $n$ is even. We will prove the following claim.

\medskip
{\bf Claim.} Let $t,s\in C_{d}^{n}$. The elements $t\widetilde{w_{0}}t^{-1}$ and $s\widetilde{w_{0}}s^{-1}$ are conjugated in $S_{d,n}^{A^{\ast}}$ if and only if $\det(ts)\in C_{d}^{2}$.

\medskip
Assume that $\det(ts)\in C_{d}^{2}$. Let us denote by $r\in C_{d}^{n}$ the element with $r_{i}=t_{i}t_{n-i}^{-1}s_{i}^{-1}s_{n-i}$ for $i\leq n/2$ and $r_{i}=1$ for $i>n/2$.
Since $\det(ts)\in C_{d}^{2}$, we have that $\det(r)\in C_{d}^{2}$. Hence, there exists $x\in C_{d}$ such that $\det(r)=x^{2}$, and let $r\rq=r\cdot (x^{-1},1,\ldots,1,x^{-1})$. Then, we have that $\det(r\rq)=1$, and $r\rq s \widetilde{w_{0}}s^{-1}{r\rq}^{-1}=t\widetilde{w_{0}}t^{-1}$. Hence, $t\widetilde{w_{0}}t^{-1}$ and $s\widetilde{w_{0}}s^{-1}$ are conjugated in $S_{d,n}^{A^{\ast}}$.

Now assume that there exists $u=vz\in S_{d,n}^{A^{\ast}}$ such that $ut\widetilde{w_{0}}t^{-1}u^{-1}=s\widetilde{w_{0}}s^{-1}$, where $v\in \mathbb{S}_{n}$ and $z\in C_{d}^{n}$ with $\det(z)=1$. Then, we have that 
\[\widetilde{w_{0}}s_{w_{0}}s^{-1}=s\widetilde{w_{0}}s^{-1}=vzt\widetilde{w_{0}}t^{-1}z^{-1}v^{-1}=(v\widetilde{w_{0}}v^{-1})(vz_{w_{0}}t_{w_{0}}t^{-1}z^{-1}v^{-1}),\]
which implies that $\widetilde{w_{0}}=v\widetilde{w_{0}}v^{-1}$ and $s_{w_{0}}s^{-1}=vz_{w_{0}}t_{w_{0}}t^{-1}z^{-1}v^{-1}$. From the first equality, we deduce that the permutation $v$ sends any transposition of the form $(i,n-i)$ to another transposition of the form $(j,n-j)$. Hence, 
the first $n/2$ coordinates of $vz_{w_{0}}t_{w_{0}}t^{-1}z^{-1}v^{-1}$ are a permutation of 
\[(t_{1}^{-1}z_{1}^{-1}t_{n}z_{n},
\ldots,
t_{\frac{n}{2}}^{-1}z_{\frac{n}{2}}^{-1}t_{\frac{n}{2}+1}z_{\frac{n}{2}+1}),\]
but also with the possibility of inverting them. From the second equality, we have that these $n/2$ coordinates are equal to $(s_{1}^{-1}s_{n},
\ldots,
s_{\frac{n}{2}}^{-1}s_{\frac{n}{2}+1}).$ From this, we deduce that $\det(zst)\in C_{d}^{2}$, but since $\det(z)=1$ we obtain that $\det(ts)\in C_{d}^{2}$. This proves the claim. 

\medskip
Using the claim, it is easy to infer that there are exactly two conjugacy classes among the elements $\{y\widetilde{w_{0}}y^{-1}:y\in S_{d,n}^{A}\}$, with representatives $\widetilde{w_{0}}$ and $g\widetilde{w_{0}}g^{-1}$, where $g=(h,1,\ldots,1)$ and $h\notin C_{d}^{2}$. In particular, the above argument does not work for $n$ even. We were not able to compute the value of $\mu(\widetilde{w_{0}})$ for this case.

\begin{question} Find the value of $\mu(\widetilde{w_{0}})$ for $\mu\in \Irr(S_{d,n}^{A^{\ast}})$ and $n$ even.
\end{question}

{\bf Types $B$ and $C$:}
In this case $\widetilde{w_{0}}=-\Id \in W_{n}$, and the only conjugate of $-\Id$ is itself.

\medskip
{\bf Case~1.} Assume that $\Lambda$ is not of the form $(\Lambda_1,\Lambda_2)=((\lambda_1,\mu_1),(\lambda_2,\mu_2))$. Then, it is straightforward to see that $-\Id$ does not belong to ${\fontshape{ui}\selectfont {\bf Stab} (\psi)}=W_{n_1}\times W_{n_2}\times \mathbb{S}_{n_{3}}\times   \cdots  \times \mathbb{S}_{n_{(d+2)/2}}$. Since the only conjugate of $-\Id$ is $-\Id$, we obtain that $\chi_{\Lambda}^{0}(-\Id)=0$.

\medskip
{\bf Case~2.} Assume that $\Lambda=(\Lambda_1,\Lambda_2)=((\lambda_1,\mu_1),(\lambda_2,\mu_2))$. If $n_{i}=\size{\lambda_{i}}+\size{\mu_{i}}$ 
for $i=1,2$, we have 
\[\chi_{\Lambda}^{0}(-\Id)=\Ind_{W_{n_1}\times W_{n_2} }^{W_n} \chi_{\lambda_{1},\mu_{1}}\otimes  \chi_{\lambda_{2},\mu_{2}}(-\Id).\]
Using Equation~\eqref{chi en tipo B} for $\chi_{\lambda,\mu}$, we have that
\[\chi_{\Lambda}^{0}(-\Id)=\Ind_{\prod_{i=1,2} W_{\lvert\lambda_{i}\rvert} \times W_{\lvert\mu_{i}\rvert}}^{W_{n}}  \S^{\lambda_1}\otimes \pare{\S^{\mu_1}\otimes \sgn}\otimes \S^{\lambda_2}\otimes \pare{\S^{\mu_2}\otimes \sgn}(-\Id).\] 
Since the only conjugate of $-\Id$ is itself, we obtain that
\begin{align*}
 \chi_{\Lambda}^{0}(-\Id)&=\binom{n}{\Lambda}  \S^{\lambda_1}(-\Id) \pare{\S^{\mu_1}\otimes \sgn}(-\Id)  \S^{\lambda_2}(-\Id) \pare{\S^{\mu_2}\otimes \sgn}(-\Id) \nonumber\\
&=\binom{n}{\Lambda}(-1)^{\size{\mu_1}+\size{\mu_2}}\  
\S^{\Lambda}(1).
\end{align*}

In Table~\ref{table:3}, we show the value of $\chi_{\Lambda}^{0}(-\Id)$ for types $B$ and $C$. 
As a consequence of these computations, we can see that the algebras $Z_{d,n}^{C}$ and $Y_{d,n}^{C}$ are not split over $\CC(u)$.

\begin{remark}
Let $\Lambda=(\Lambda_{1},\Lambda_{2})\in \mathcal{R}_{d,n}$. Using Lemmas~\ref{T0 2} and~\ref{g en tipo C 2}, we have that 
\[\chi_{\Lambda}'(T_{0}')=\chi_{\Lambda}^{0}(-\Id)w^{2N_{1}N_{2}+N_{2}},\]
where $w^{2}=1-2u$ and  $N_{1}$ and $N_{2}$ are the number of entries of $\psi$ that belong to $C_{d}^{2}$ and $C_{d}\backslash C_{d}^{2}$, respectively. If $N_{2}$ is not even, we have that not all characters of $Z_{d,n}^{C}$ are defined over $\C(u)$, which implies that $Z_{d,n}^{C}$ is not split over $\C(u)$. On the other hand, using Lemmas~\ref{T0 1} and~\ref{T0 2}, we have that 
\[\chi_{\Lambda}(T_{0})=\chi_{\Lambda}^{0}(-\Id)(-i)^{2N_{1}N_{2}+N_{2}}v^{l(w_{0})-k_{\Lambda}^{C}},\]
where $v^{2}=u$, and $k_{\Lambda}^{C}$ denotes the value of $k_{\Lambda}$ for type $C$, as in Lemma~\ref{ecuacion tipo c}. Observe that for type $B$, using the splitness of $Y_{d,n}^{B}$ over $\C(u)$, we have that $l(w_{0})-k_{\Lambda}^{B}$ is even, where $k_{\Lambda}^{B}$ denotes the value of $k_{\Lambda}$ for type $B$, as in Lemma~\ref{ecuacion para tipo B}. Since $k_{\Lambda}^{B}=k_{\Lambda}^{C}+\size{\Lambda_{2}}^{\text{alt}}$, we have that $l(w_{0})-k_{\Lambda}^{C}$ is odd for the cases in which $\size{\Lambda_{2}}^{\text{alt}}$ is odd. For these cases,  the value of $\chi_{\Lambda}(T_{0})$ is not defined over $\C(u)$ for type $C$, which implies that $Y_{d,n}^{C}$ is not split over $\C(u)$.
\end{remark}

\medskip
{\bf Type $D$:}
Let $\mu\in \Irr(S_{d,n}^{D})$ be a factor of the restriction of  $\chi_{\Lambda}^{0}\in\Irr (S_{d,n}^{B})$, for $\Lambda=(\Lambda_{1},\Lambda_{2},\lambda_{3},\ldots,\lambda_{(d+2)/2})\in \mathcal{R}_{d,n}$, and $\Lambda_{i}=(\lambda_{i},\mu_{i})$ for $i=1,2$. If $\mu$ is the restriction of $\chi_{\Lambda}^{0}$, we have that $\mu(\widetilde{w_{0}})=\chi_{\Lambda}^{0}(\widetilde{w_{0}})$. Assume now that $\chi_{\Lambda}^{0}$ has another irreducible factor in its restriction, which we call $\eta$. By Subsection~\ref{caso ciclico}, we know that $\eta=\mu^{t}$, where  $t=\Diag(-1,1,\ldots,1)$. Then, we have
\[\chi_{\Lambda}^{0}(\widetilde{w_{0}})=\mu(\widetilde{w_{0}})+\eta(\widetilde{w_{0}})=\mu(\widetilde{w_{0}})+\mu^{t}(\widetilde{w_{0}})=\mu(\widetilde{w_{0}})+\mu(\widetilde{w_{0}})=2\mu(\widetilde{w_{0}}),\]
since  $\widetilde{w_{0}}\in \set{\pm 1}^{n}$ and $t\widetilde{w_{0}}t^{-1}=\widetilde{w_{0}}$. Then, we have 
\[\mu(\widetilde{w_{0}})=\frac{\chi_{\Lambda}^{0}(\widetilde{w_{0}})}{2}, \qquad \text{or} \qquad \ \mu(\widetilde{w_{0}})=\chi_{\Lambda}^{0}(\widetilde{w_{0}}), \]
depending on whether $(\lambda_{1},\lambda_{2})= (\mu_{1},\mu_{2})$ or not, respectively. Therefore, to compute the value of $\mu(\widetilde{w_{0}})$ we need to find the value of $\chi_{\Lambda}^{0}(\widetilde{w_{0}})$. We separate into cases.

\medskip
{\bf Case~1.} Assume that $n$ is even. In this case, we have $\widetilde{w_{0}}=-\Id$, and we already found the value of $\chi_{\Lambda}^{0}(-\Id)$ for Type~B.

\medskip
{\bf Case~2.} Assume that $n$ is odd. In this case, we have  $\widetilde{w_{0}}=\Diag(1,-1,-1,\ldots,-1)$, and let us denote this element by $b$. Since  $b\in W_{n}$, we have 
that
\begin{equation}\chi_{\Lambda}^{0}(b)=\Ind_{{\fontshape{ui}\selectfont {\bf Stab} (\psi)}}^{W_n}\chi_{\Lambda}^{\psi}(b).\end{equation}

\medskip
{\bf Case~2.1.} Assume that $\Lambda $ is not of the form $(\Lambda_{1},\Lambda_{2})$ or $(\Lambda_{1},\Lambda_{2},\lambda_{3})$, with $\Lambda_{i}=(\lambda_{i},\mu_{i})$ for $i=1,2$, and $\size{\lambda_{3}}=1$. Since the conjugates of $b$ are the diagonal elements with exactly one $1$, it is clear that no conjugate of $b$ belongs to ${\fontshape{ui}\selectfont {\bf Stab} (\psi)}$. Therefore, we have that $\chi_{\Lambda}^{0}(b)=0$.

\medskip
{\bf Case~2.2.1.} Assume that  $\Lambda $ is of the form $(\Lambda_{1},\Lambda_{2})$, with $\Lambda_{i}=(\lambda_{i},\mu_{i})$ for $i=1,2$. Then, we have
\begin{align*}
 \chi_{\Lambda}^{0}(b)&=\Ind_{\prod_{i=1,2} W_{\lvert\lambda_{i}\rvert} \times W_{\lvert\mu_{i}\rvert}}^{W_{n}} \textstyle{\S^{\lambda_1}\otimes (\S^{\mu_1}\otimes \sgn)\otimes \S^{\lambda_2}\otimes (\S^{\mu_2}\otimes \sgn)(b)}\nonumber \\
&=\frac{\S^{\Lambda}(1)}
{\Lambda!}
(C_{1}-C_{2}+C_{3}-C_{4}),
\end{align*}
where $C_{1},C_{2},C_{3}$, and $C_{4}$ are the number of ways to conjugate $b$ by an element of $\S_{n}$ and obtain an element in which its only $1$ 
belongs to the blocks corresponding to  $\lambda_{1},\mu_{1},\lambda_{2}$, and $\mu_{2}$, respectively. We 
are multiplying them by the value of $\chi_{\Lambda}^{\psi}$ at these elements. 
Observe that  $\{\pm 1\}^{n}$ acts trivially by conjugation on $b$, 
which is why it does not appear in the equation.
The 
numbers $C_i$ for $i=1,2,3,4$ are 
\[\textstyle{(n-1)!\size{ \lambda_1 },\quad (n-1)!\size{ \mu_1 }, \quad (n-1)!\size{ \lambda_2 },\quad \text{and} \quad (n-1)!\size{ \mu_2 }},\]
respectively. Therefore, we obtain that $\chi_{\Lambda}^{0}(b)$ is equal to
\[\frac{(n-1)!\S^{\Lambda}(1)}{\Lambda!}
\size{\Lambda}^{\text{alt}}.
\]

\medskip
{\bf Case~2.2.2.} Assume that  $\Lambda $ is of the form $(\Lambda_{1},\Lambda_{2},\lambda_{3})$ with $\Lambda_{i}=(\lambda_{i},\mu_{i})$ for $i=1,2$, and $\size{\lambda_{3}}=1$.
For this case, we have that the conjugates of $b$ should have the number $1$ in the position corresponding to $\lambda_{3}$. Then, this situation is equivalent to having $-\Id$ and $n-1$ even, which we already computed.

\medskip
In Table~\ref{table:4}, we show the value of $\mu(\widetilde{w_{0}})$ for type $D$.

\begin{corollary}
Using Lemmas~\ref{T0 1} and~\ref{T0 2}, we have that
\[\chi_{\Lambda}^{q}(T_{w_{0}})=\chi_{\Lambda}^{0}(\widetilde{w_{0}})(-i)^{g_{\Lambda}}\sqrt{q^{-f_{\Lambda}}}.\]
Thanks to the calculations of $g_{\Lambda},f_{\Lambda}$, and $\chi_{\Lambda}^{0}(\widetilde{w_{0}})$, which we did in Subsections~\ref{g lambda},~\ref{f lambda} and~\ref{T_{0}}, we can easily find this value for each 
classical type.
\end{corollary}

\section{Tables}\label{tables}

In Table~\ref{table:1}, using Notations~\ref{notation psi} and~\ref{alt}, we show the values of $k_{\Lambda}$ and $g_{\Lambda}$ for each type, where $\Lambda \in \mathcal{Q}_{d,n}, \mathcal{R}_{d,n}$ or $\mathcal{S}_{d,n}$. 
The values $M_{1}$ and $M_{2}$ represent the number of entries of $\psi$ that belong to $C_{d}^{2}$ and $C_{d}\backslash C_{d}^{2}$, respectively. Recall that $T_{w_{0}}^{2}\in \mathcal{H}(G,U)$ acts by scalar multiplication in $\chi_{\Lambda}^{q}$ by scalar
\[q^{k_{\Lambda}-l(w_{0})}(-1)^{g_{\Lambda}},\]
which implies that 
\[\chi_{\Lambda}^{q}(T_{w_{0}}^{2})=\chi_{\Lambda}^{0}(1)\  q^{k_{\Lambda}-l(w_{0})}(-1)^{g_{\Lambda}}.\]

\begin{table}[H] \normalfont
\centering
\begin{tabular}{ |c|c|c|c| } 
\hline
Type & $k_{\Lambda}$  & $g_{\Lambda}$   \\ \hline
$A$ & $n(\Lambda')-n(\Lambda)$ & $M_{1}M_{2}$  \\   \hline 
\rule{0pt}{2.3ex} $A^{\ast}$ & $n(\Lambda \rq)-n(\Lambda)$ & $M_{1}M_{2}$ \\ \hline 
\rule{0pt}{4ex} $B$ & \begin{tabular}{@{}c@{}}$n(\Lambda \rq)-n(\Lambda)+n(\Lambda_{1}\rq,\Lambda_{2}\rq)-n(\Lambda_{1},\Lambda_{2})+$ \\ $\size{(\Lambda_{1},\Lambda_{2})}^{\text{alt}}$
\end{tabular}  & $2M_{1}M_{2}$\\  \hline 
\rule{0pt}{4ex}$C$& 
\begin{tabular}{@{}c@{}}$n(\Lambda \rq)-n(\Lambda)+n(\Lambda_{1}\rq,\Lambda_{2}\rq)-n(\Lambda_{1},\Lambda_{2})$ \\ $+\size{\Lambda_{1}}^{\text{alt}}$ 
\end{tabular}
& $2M_{1}M_{2}+M_{2}$ \\ \hline
\rule{0pt}{2.3ex}$D$& $n(\Lambda \rq)-n(\Lambda)+n(\Lambda_{1}\rq,\Lambda_{2}\rq)-n(\Lambda_{1},\Lambda_{2})$& $2M_{1}M_{2}$ \\ \hline
\end{tabular}
\caption{Values of $k_{\Lambda}$ and $g_{\Lambda}$.}
\label{table:1}
\end{table}

In Tables~\ref{table:2} and~\ref{table:3}, we show the value of $\chi_{\Lambda}^{0}(\widetilde{w_{0}})$ for types $A,B$ and $C$. In Table~\ref{table:2}, each $\sigma_i$ denotes a product of $\textstyle \floor{\frac{\size{\lambda_i}}{2}} $ disjoint transpositions. In Table~\ref{table:4}, we show the value of $\mu(\widetilde{w_{0}})$, for $\mu\in\Irr(S_{d,n}^{D})$ a factor in the restriction of $\chi_{\Lambda}^{0}\in\Irr (S_{d,n}^{B})$. Recall that 
\[\chi_{\Lambda}^{q}(T_{w_{0}})=\chi_{\Lambda}^{0}(\widetilde{w_{0}})(-i)^{g_{\Lambda}}\sqrt{q^{-f_{\Lambda}}},\]
where $f_{\Lambda}=l(w_{0})-k_{\Lambda}$.
\begin{table}[H] \normalfont
\centering
\begin{tabular}{ |c|c|c| } 
\hline
\rule{0pt}{2.3ex} & $\chi_{\Lambda}^{0}(\widetilde{w_{0}})$  \\[0.3ex] \hline
$n$ even, not all blocks of even size & $0$   \\   \hline 
$n$ odd, more than one block of odd size  & $0$\\ \hline 
$n$ even, all blocks of even size \rule{0pt}{4ex}  & $\frac{(n/2)!}{(\size{\lambda_1}/2)!
\cdots  (\size{\lambda_d }/2)!}\ \S^{\lambda_1}(\sigma_1)
\cdots  \S^{\lambda_d}(\sigma_d)$  \\[3 ex] \hline 
$n$ odd, one block of odd size \rule{0pt}{4ex} & $\frac{((n-1)/2)!}{(\size{\lambda_1}/2)!
\cdots (\size{\lambda_{d-1} }/2)! ((\size{\lambda_{d}}-1)/2)!}\ \S^{\lambda_1}(\sigma_1)
\cdots \S^{\lambda_d}(\sigma_d)$  \\[3 ex] \hline
\end{tabular}
\caption{Value of $\chi_{\Lambda}^{0}(\widetilde{w_{0}})$ in type $A$.}
\label{table:2}
\end{table}

\begin{table}[H] \normalfont
\centering
\begin{tabular}{ |c|c|c| } 
\hline
\rule{0pt}{2.3ex} & $\chi_{\Lambda}^{0}(-\Id)$  \\ \hline
$\Lambda$ is not of the form $(\Lambda_{1},\Lambda_{2})$ & $0$   \\   \hline 
\rule{0pt}{4ex}$\Lambda=(\Lambda_{1},\Lambda_{2})$  & 
$\binom{n}{\Lambda}(-1)^{\size{\mu_1}+\size{\mu_2}}\S^{\Lambda}(1)$
\\[2ex] \hline 
\end{tabular}
\caption{Value of $\chi_{\Lambda}^{0}(-\Id)$ in types $B$ and $C$.}
\label{table:3}
\end{table}

\begin{table}[H] \normalfont
\centering
\begin{tabular}{ |c|c|c| } 
\hline
\rule{0pt}{2.3ex} & $\mu(\widetilde{w_{0}})$  \\ \hline
\rule{0pt}{2.7ex} $n$ even, $(\lambda_{1},\lambda_{2})\neq(\mu_{1},\mu_{2})$ & $\chi_{\Lambda}^{0}(-\Id)$ \\[0.5ex]   \hline 
 \rule{0pt}{4ex} $n$ even, $(\lambda_{1},\lambda_{2})=(\mu_{1},\mu_{2})$  & $\frac{\chi_{\Lambda}^{0}(-\Id)}{2}$\\[1.2 ex] \hline 
\rule{0pt}{4ex} \begin{tabular}{@{}c@{}}$n$ odd, $\Lambda \neq (\Lambda_{1},\Lambda_{2})$ and $\Lambda \neq (\Lambda_{1},\Lambda_{2},\lambda_{3})$,\\
where $\size{\lambda_{3}}=1$ 
\end{tabular}  & $0$ \\[2 ex] \hline 
$n$ odd, $\Lambda=(\Lambda_{1},\Lambda_{2})$, and $(\lambda_{1},\lambda_{2})\neq(\mu_{1},\mu_{2})$ & \rule{0pt}{5ex} 
$\frac{(n-1)!\ \S^{\Lambda}(1)}{\Lambda!}\size{\Lambda}^{\text{alt}}$ 
\\[2 ex] \hline 
$n$ odd, $\Lambda=(\Lambda_{1},\Lambda_{2})$, and $(\lambda_{1},\lambda_{2})=(\mu_{1},\mu_{2})$ \rule{0pt}{3ex} 
&  $0$
\\[0.5 ex] \hline
\rule{0pt}{4ex} \begin{tabular}{@{}c@{}} $n$ odd, $\Lambda=(\Lambda_{1},\Lambda_{2},\lambda_{3})$ with $\size{\lambda_{3}}=1$,\\ and $(\lambda_{1},\lambda_{2})\neq (\mu_{1},\mu_{2})$  \end{tabular} & $\chi_{(\Lambda_{1},\Lambda_{2})}^{0}(-\Id)$ \\ \hline
\rule{0pt}{5ex} \begin{tabular}{@{}c@{}} $n$ odd, $\Lambda=(\Lambda_{1},\Lambda_{2},\lambda_{3})$ with $\size{\lambda_{3}}=1$,\\ and $(\lambda_{1},\lambda_{2})= (\mu_{1},\mu_{2})$  \end{tabular} & $\frac{\chi_{(\Lambda_{1},\Lambda_{2})}^{0}(-\Id)}{2}$ \\[2 ex] \hline
\end{tabular}
\caption{Value of $\mu(\widetilde{w_{0}})$ in type $D$.}
\label{table:4}
\end{table}

For $\mu \in \Irr(S_{d,n}^{A^{\ast}})$ a factor in the restriction of $\chi_{\Lambda}^{0}\in \Irr(S_{d,n}^{A})$, where $\Lambda\in \mathcal{Q}_{d,n}$ and $n$ odd, we have
\[\mu(\widetilde{w_{0}})=\frac{o(\Lambda)}{d}\chi_{\Lambda}^{0}(\widetilde{w_{0}}),\]
where $o(\Lambda)$ is defined as in Example~\ref{Sl_{n} caracteres}. We were not able to compute this value for $n$ even.

\medskip\noindent{\bf Acknowledgement.} We thank Leandro Vendramin and Marco Farinati for their helpful comments. The first author is supported by the FWO grant G0F5921N. The second author by
CONICET PIP 11220210100220CO.

\printbibliography

\medskip
\noindent 

\noindent
Emiliano Liwski \\
Department of Mathematics, KU Leuven, Celestijnenlaan 200B, B-3001 Leuven, Belgium
\\ E-mail address: {\tt emiliano.liwski@kuleuven.be}

\medskip  

\noindent
Martín Mereb \\
 Departamento de Matemática, Facultad de Ciencias Exactas y Naturales \& IMAS \\
 Universidad de Buenos Aires \&  CONICET Argentina
 \\  E-mail address: 
 {\tt mmereb@dm.uba.ar}
 \medskip

\end{document}